\documentclass[twoside, 11pt]{article}
\usepackage{amsmath,amssymb,amsthm,mathrsfs}
\usepackage[margin=2cm]{geometry} 
\usepackage{lipsum}
\usepackage{titlesec,hyperref}
\usepackage{fancyhdr}
\usepackage[numbers,sort&compress]{natbib}
\usepackage{color}

\fancypagestyle{plain}
{
\fancyhf{}

}

\titleformat{\subsection}{\it}{\thesubsection.\enspace}{1.5pt}{}
\titleformat{\subsubsection}{\it}{\thesubsubsection.\enspace}{1.5pt}{}

\newtheorem{theorem}{Theorem}[section]
\newtheorem{proposition}[theorem]{Proposition}

\newtheorem{remark}{Remark}[section]
\newtheorem{lemma}[theorem]{Lemma}
\numberwithin{equation}{section}

\allowdisplaybreaks

\def\la{\langle}
\def\ra{\rangle}
\def\ep{\varepsilon}
\def\es{\epsilon}
\def\k{\kappa}
\def\d{\delta}
\def\r{\rho^\epsilon}
\def\vr{\varrho^\epsilon}
\def\u{u^\epsilon}
\def\v{v^\epsilon}
\def\h{h^\epsilon}
\def\g{g^\epsilon}
\def\ps{\psi^\epsilon}
\def\y{\langle y \rangle^{2l}}
\def\yl{\langle y \rangle^{2l-1}}
\def\ya{\langle y \rangle^{2l}}
\def\yb{\langle y \rangle^{2l-1}}
\def\p{\partial}
\def\a{\alpha}
\def\b{\beta}
\def\ga{\gamma}
\def\H{\mathcal{H}}
\def\z{\mathcal{Z}}
\def\C{\mathcal{C}}
\def\O{\Omega}
\def\T{\mathbb{T}}
\def\t{\tau}

\def\e{\mathcal{E}}
\def\D{\mathcal{D}}
\def\n{\mathcal{N}}
\def\ta{\Theta}
\def\orho{\overline{\rho}}
\def\ou{\overline{u}}
\def\ov{\overline{v}}
\def\oh{\overline{h}}
\def\og{\overline{g}}
\def\ophi{\overline{\psi}}
\def\irho{\widehat{\rho}}
\def\iu{\widehat{u}}
\def\ih{\widehat{h}}

\begin{document}
\title{Boundary Layer Problems for the Two-dimensional Inhomogeneous
Incompressible Magnetohydrodynamics Equations\hspace{-4mm}}
\author{Jincheng Gao$^{\dag}$   \quad    Daiwen Huang$^{\ddag}$   \quad    Zheng-an Yao$^{\dag}$\\[10pt]
\small {$^\dag $School of Mathematics, Sun Yat-sen University,}\\
\small {Guangzhou, 510275,  P.R.China}\\[5pt]
\small {$^\ddag $Institute of Applied Physics and Computational Mathematics,}\\
\small {Beijing, 100088, P.R.China}\\[5pt]
}

\footnotetext{Email: \it gaojch5@mail.sysu.edu.cn(J.C. Gao),
                     \it hdw55@tom.com(D.W.Huang)
                     \it mcsyao@mail.sysu.edu.cn(Z.A.Yao).}
\date{}

\maketitle

\begin{abstract}

In this paper, we study the well-posedness of boundary layer problems for the
inhomogeneous incompressible magnetohydrodynamics(MHD) equations,
which are derived from the two-dimensional density-dependent incompressible MHD equations.
Under the assumption that initial tangential magnetic field is not zero
and density is a small perturbation of the outer constant flow in supernorm,
the local-in-time existence and uniqueness of inhomogeneous incompressible MHD boundary layer equation
are established in weighted Conormal Sobolev spaces by energy method.
As a byproduct, the local-in-time well-posedness of homogeneous incompressible MHD boundary layer
equations with any large initial data can be obtained.

%
%

\end{abstract}

\tableofcontents

\section{Introduction and Main Result}

In this paper, we consider the boundary layer problems in the small viscosity and resistivity
limit for the two-dimensional inhomogeneous incompressible Magnetohydrodynamics(MHD) equation
in a period domain $\Omega =: \left\{(x, y): x\in \mathbb{T}, y\in \mathbb{R}^+\right\}$:
\begin{equation}\label{eq1}
\left\{
\begin{aligned}
&\p_t \rho^\ep+{\rm div}(\rho^\ep \mathbf{u}^\ep)=0,\\
&\rho^\ep \p_t \mathbf{u}^\ep+\rho^\ep(\mathbf{u}^\ep \cdot \nabla) \mathbf{u}^\ep
 -\mu \ep \Delta \mathbf{u}^\ep+\nabla p^\ep=(\mathbf{h}^\ep \cdot \nabla)\mathbf{h}^\ep,\\
&\p_t \mathbf{h}^\ep-\nabla \times (\mathbf{u}^\ep \times \mathbf{h}^\ep)
-\k \ep \Delta \mathbf{h}^\ep=0,\\
&{\rm div}\mathbf{u}^\ep=0,\quad {\rm div}\mathbf{h}^\ep=0.
\end{aligned}
\right.
\end{equation}
Here, we assume the viscosity and resistivity coefficients have the same order of a small parameter $\ep$.
The unknown functions $\rho^\ep$ denotes the density of fluid,
$\mathbf{u}^\ep=(u_1^\ep, u_2^\ep)$ denotes the velocity vector,
$\mathbf{h}^\ep=(h_1^\ep, h_2^\ep)$ denotes the magnetic field,
and $p^\ep=\widetilde{p}^\ep+\frac{|\mathbf{h}^\ep|^2}{2}$
represents the total pressure with $\widetilde{p}^\ep$ the pressure of fluid.
This system \eqref{eq1} can be used as model to descritbe a viscous fluid that
is incompressible but has nonconstant density,
and hence, it is much more complex than the classical incompressible MHD equation with constant density.
To complete the system \eqref{eq1}, the boundary conditions are given by
\begin{equation}\label{bc1}
u_1^\ep|_{y=0}=u_2^\ep|_{y=0}=0,\quad \p_y h_1^\ep|_{y=0}=h_2^\ep|_{y=0}=0.
\end{equation}
As the parameter $\varepsilon$ tends to zero in the system \eqref{eq1},
we obtain the following system formally
\begin{equation*}\label{ideal}
\left\{
\begin{aligned}
&\p_t \rho^0+{\rm div}(\rho^0 \mathbf{u}^0)=0,\\
&\rho^0 \p_t \mathbf{u}^0+\rho^0(\mathbf{u}^0 \cdot \nabla) \mathbf{u}^0
 +\nabla p^0=(\mathbf{h}^0 \cdot \nabla)\mathbf{h}^0,\\
&\p_t \mathbf{h}^0-\nabla \times (\mathbf{u}^0 \times \mathbf{h}^0)=0,\\
&{\rm div}\mathbf{u}^0=0,\quad {\rm div}\mathbf{h}^0=0.
\end{aligned}
\right.
\end{equation*}
which is the inhomogeneous incomrpessible ideal MHD system
with the unknown function $(\rho^0, \mathbf{u}^0, \mathbf{h}^0)$.
To find out the terms in \eqref{eq1} whose contributions are essential
for the boundary layer, we use the same scaling as the one used in
\cite{{Oleinik2},{Liu-Xie-Yang}}
$$
t=t,\quad x=x,\quad \widetilde{y}=\varepsilon^{-\frac{1}{2}}y,
$$
and set
$$
\begin{aligned}
&\rho(t, x, \widetilde{y})=\rho^\ep (t, x, y),\quad
  p(t, x, \widetilde{y})=p^\ep (t, x, y),\\
&u_1(t, x, \widetilde{y})=u_1^\ep (t, x, y),\quad
u_2(t, x, \widetilde{y})=\ep^{-\frac{1}{2}}u_2^\ep (t, x, y),\\
&h_1(t, x, \widetilde{y})=h_1^\ep (t, x, y),\quad
 h_2(t, x, \widetilde{y})=\ep^{-\frac{1}{2}}h_2^\ep(t, x, y),
\end{aligned}
$$
then the system \eqref{eq1}, after taking the leading order, is reduced to
\begin{equation}\label{eq2}
\left\{
\begin{aligned}
&\p_t \rho +u_1 \p_x \rho+u_2 \p_y \rho=0,\\
&\rho \p_t u_1+\rho u_1 \p_x u_1+\rho u_2 \p_y u_1-\mu \p_y^2 u_1+\p_x p=h_1 \p_x h_1+h_2 \p_y h_1,\\
&\partial_y p=0,\\
&\p_t h_1+\p_y(u_2 h_1-u_1 h_2)=\k \p_y^2 h_1,\\
&\p_t h_2-\p_x(u_2 h_1-u_1 h_2)=\k \p_y^2 h_2,\\
&\p_x u_1+\p_y u_2=0,\quad \p_x h_1+\p_y h_2=0,
\end{aligned}
\right.
\end{equation}
where $(t, x, y)\in [0, T]\times \Omega$, here we have replaced
$\widetilde{y}$ by $y$ for simplicity of notations.
Indeed, the nonlinear boundary layer system \eqref{eq2}
becomes the classical well-known unsteady boundary layer system
if the density becomes constant and magnetic field vanishes(cf.\cite{Schlitchting}).

The third equation of system \eqref{eq2} implies that the leading order of
boundary layers for the total pressure $p^\varepsilon(t, x, y)$ is
invariant across the boundary layer, and should be matched to the outflow pressure
$P(t, x)$ on top of boundary layer, that is, the trace of pressure of idea MHD flow.
Hence, we obtain
$$
p(t, x, y)\equiv P(t, x).
$$
Furthermore, the density $\rho(t, x, y)$, tangential component $u_1(t, x, y)$ of velocity flied,
$h_1(t, x, y)$ of magnetic field, should match the outflow density $\theta(t, x)$,
tangential velocity $U(t, x)$ and tangential magnetic field $H(t, x)$, on the top of boundary layer, that is
\begin{equation*}
\rho(t, x, y)\rightarrow  \theta(t, x), \quad
u_1(t, x, y)\rightarrow  U(t, x), \quad
h_1(t, x, y) \rightarrow  H(t, x),
~{\rm as}~y \rightarrow +\infty,
\end{equation*}
where $\theta(t, x), U(t, x)$ and $H(t, x)$ are the trace of density,
tangential velocity and tangential magnetic field respectively.
Then, we have the following matching conditions:
\begin{equation}\label{ma-co}
\p_t \theta+U\partial_x \theta=0,\quad
\theta \p_t U+\theta U\p_x U+\p_x P=H\p_x H,\quad
\p_t H+U\partial_x H-H\partial_x U=0.
\end{equation}
Moreover, by virtue of the boundary condition \eqref{bc1}, one attains  the following boundary condition
\begin{equation}\label{bc2}
\left.u_1 \right|_{y=0}=
\left.u_2 \right|_{y=0}=
\left.\partial_y h_1\right|_{y=0}=
\left. h_2\right|_{y=0}=0.
\end{equation}

In this paper, we consider the outer flow $(\theta, U, H)=(1, 1, 1)$,
which implies the pressure $p$ being a constant.
On the other hand, it is noted that the fifth equation of \eqref{eq2}
is a direct consequences of the fourth equation of \eqref{eq2}.
Hence, we only need to study the following initial boundary value problem for
the inhomogeneous incompressible MHD boundary layer equation
\begin{equation}\label{eq3}
\left\{
\begin{aligned}
&\p_t \rho+u_1 \p_x \rho +u_2\p_y \rho=0,\\
&\rho \p_t u_1+\rho u_1\p_x u_1+ \rho u_2 \p_y u_1-\mu \partial_y^2 u_1
=h_1 \p_x h_1+h_2\p_y h_1,\\
&\p_t h_1+\p_y(u_2 h_1-u_1 h_2)-\k \p_y^2 h_1=0,\\
&\p_x u_1+\p_y u_2=0, \quad \p_x h_1+\p_y h_2=0,
\end{aligned}
\right.
\end{equation}
where the density $\rho:=\rho(t, x, y)$, velocity field $(u_1, u_2):=(u_1(t,x,y), u_2(t,x,y))$,
the magnetic field $(h_1, h_2):=(h_1(t,x,y),h_2(t,x,y))$ are unknown functions.
The boundary conditions for equation \eqref{eq3} are given by
\begin{equation}\label{bc3}
\left\{
\begin{aligned}
&\left.u_1\right|_{y=0}=\left. u_2\right|_{y=0}
=\left.\p_y h_1 \right|_{y=0}=\left. h_2\right|_{y=0}=0,\\
&\lim_{y\rightarrow +\infty }\rho(t,x,y)
=\lim_{y\rightarrow +\infty }u_1(t,x,y)
=\lim_{y\rightarrow +\infty }h_1(t,x,y)=1.
\end{aligned}
\right.
\end{equation}

Let us first introduce some weighted Sobolev spaces for later use.
For any $l \in \mathbb{R}$, denote by $L^2_l(\Omega)$ the weighted Lebesgue space
with respect to the spatial variables:
$$
L_l^2(\Omega):= \{f(x, y):\Omega \rightarrow \mathbb{R},\
\|f\|_{L_l^2(\Omega)}^2 := \int_\Omega \la y \ra^{2l} |f(x,y)|^2 dxdy<+\infty, \
\la y \ra\triangleq 1+y\},
$$
and denote the weighted $L^\infty_l(\Omega)$ Lebesgue space by
$$
L^\infty_l(\O):=\{f(x, y):\Omega \rightarrow \mathbb{R},\
\|f\|_{L_l^\infty(\Omega)}:=
\underset{(x, y)\in \O}{\rm esssup}|\la y \ra^{l}f(x, y)|<+\infty,\
\la y \ra\triangleq 1+y\}.
$$
To define the conormal Sobolev spaces, we will use the notation:
$
Z_1=\p_x, Z_2=\varphi(y)\p_y,
$
where the function  $\varphi(y)\triangleq \frac{y}{1+y}$.
Then, we can define the conormal Sobolev spaces as follows:
$$
H^{m,l}_{co}\triangleq \{f\in L^2_l(\Omega)|\
Z^I f \in L_l^2(\Omega),\ |I|\le m\},
$$
where $I=(I_1, I_2)$ and $Z^I=Z_1^{I_1} Z_2^{I_2}$. We also use the notation
\begin{equation*}\label{ndef1}
\|u\|_{m,l}^2=\sum_{|\a|\le m}\|Z^\a u\|_{L^2_l(\O)}^2,\quad
\|u\|_{m,l,\infty}^2=\sum_{|\a|\le m}\|Z^\a u\|_{L_l^\infty(\O)}^2.
\end{equation*}
It is easy to check that
$$
Z_i Z_j=Z_j Z_i, \quad j,k=1,2,
$$
and
$$
\p_y Z_1 =Z_1 \p_y,\quad \p_y Z_2 \neq Z_2 \p_y.
$$
For later use and notational convenience, set
$Z_\tau=(\p_t, Z_1)$ and
$
\z^\a =Z_\tau^{\a_1} Z_2^{\a_2}=\p_t^{\a_{11}}Z_1^{\a_{12}} Z_2^{\a_2},
$
where $\a, \a_1, \a_2$ are the differential multi-indices with
$\a=(\a_1, \a_2), \a_1=(\a_{11}, \a_{12})$,
and we also use the notation
\begin{equation*}\label{ndef2}
\|f(t)\|_{\H^m_l}^2=\sum_{|\a|\le m}\|\z^\a f(t)\|_{L^2_l(\Omega)}^2,
\quad \|f(t)\|_{\H^{m,\infty}_l}^2=\sum_{|\a|\le m}\| \z^{\a} f(t)\|_{L_l^\infty(\O)}^2
\end{equation*}
for smooth space-time function $f(x,t)$. We also use
\begin{equation*}\label{ndef3}
\|f(t)\|_{\H^{m}_{l, tan}}^2
=\sum_{|\a_1|\le m}\| Z_\t^{\a_1}f(t)\|_{L_l^2(\O)}^2, \quad
\|f(t)\|_{\H^{m, \infty}_{l, tan}}^2
=\sum_{|\a_1|\le m}\| Z_\t^{\a_1}f(t)\|_{L_l^\infty(\O)}^2.
\end{equation*}

Finally, we define the functional space $\mathcal{B}_l^m(T)$ for a pair of function
$(\rho, u_1, h_1)=(\rho, u_1, h_1)(x, y, t)$ as follows:
\begin{equation}
\mathcal{B}^m_l(T)=\{(\rho-1, u_1-1, h_1-1)\in L^\infty([0, T]; L^2_l(\O)):
           \underset{0\le t \le T}{\rm esssup}
           \|(\rho, u_1, h_1)(t)\|_{\mathcal{B}^m_l}<+\infty\},
\end{equation}
where the norm $\|\cdot\|_{\mathcal{B}^m_l}=\|\cdot\|_{\overline{\mathcal{B}}^m_l}
+\|\cdot\|_{\widehat{\mathcal{B}}^m_l}$ is given by
\begin{equation}\label{norm-BX}
\|(\rho, u_1, h_1)(t)\|_{\overline{\mathcal{B}}^m_l}:=
\|(\rho-1, u_1-1, h_1-1)(t)\|_{\H^m_l}^2
+\|\p_y(\rho, u_1, h_1)(t)\|_{\H^{m-1}_l}^2
+\| \p_y \rho(t)\|_{\H^{1,\infty}_1}^2,
\end{equation}
and
\begin{equation}\label{norm-BY}
\begin{aligned}
\|(\rho, u_1, h_1)(t)\|_{\widehat{\mathcal{B}}^m_l}:=
&\sum_{i=0}^{m-1} \|\p_t^i (\p_x \rho, \p_y \rho, \p_x u_1, \p_x h_1)(t)\|_{\H^m_l}^2
 +\sum_{i=0}^{m-1} \|\p_t^i \p_y(\p_x^2 \rho, \p_y^2 \rho)(t)\|_{\H^{1,\infty}_0}^2\\
&+\sum_{i=0}^{m-1} \|\p_t^i \p_y(\p_x \rho, \p_y \rho, \p_x u_1, \p_x h_1)(t)\|_{\H^{m-1}_l}^2.
\end{aligned}
\end{equation}
In the present article, we supplement the MHD boundary layer equation \eqref{eq3}
with the initial data
\begin{equation}
(\rho, u_1, h_1)(0, x, y)=(\rho_0, u_{10}, h_{10})(x, y),
\end{equation}
satisfying
\begin{equation}\label{id0}
0<m_0 \le \rho_0 \le M_0 <+\infty,
\end{equation}
and
\begin{equation}\label{id1}
\|(\rho_0, u_{10}, h_{10})\|_{\mathcal{B}^m_l}\le C_0<+\infty,
\end{equation}
where $m_0, M_0, C_0>0$ are positive constants and
\textbf{ the time derivatives of initial data in \eqref{id1}
are defined through the MHD boundary layer equations \eqref{eq3}}.
Hence, we set
\begin{equation}\label{id2}
\mathcal{B}^{m, l}_{BL, ap}=\{(\rho-1, u_1-1, h_1-1)\in H_l^{4m}|\p_t(\rho, u_1,h_1), k=1,...,m
~{\rm are~ defined~ throgh~Eq.}~\eqref{eq3}\}
\end{equation}
and
\begin{equation}\label{id3}
\mathcal{B}^{m,l}_{BL}={\rm the ~closure~of}~\mathcal{B}^{m, l}_{BL, ap}~{\rm in~the~norm~}\|\cdot\|_{\mathcal{B}^m_l}.
\end{equation}

Now, we can state the main results with respect to the well-posedness theory for the inhomogeneous
incompressible MHD boundary layer equations \eqref{eq3}-\eqref{bc3} in this paper as follows.

\begin{theorem}[Main Thoerem]\label{local}
Let $m \ge 5$ be an integer and $l \ge 2$ be a real number.
Assume the initial data $(\rho_0, u_{10}, h_{10})\in \mathcal{B}^{m,l}_{BL}$
given in \eqref{id3} and satisfying \eqref{id0} and \eqref{id1}.
Moreover, there exists a small constant $\delta_0>0$ such that
\begin{equation}\label{condition1}
h_{10}(x,y)\ge 2\delta_0,~{\rm for~all}~(x, y)\in \Omega,
\end{equation}
and
\begin{equation}\label{condition2}
\|\rho_0-1\|_{L^\infty_0(\O)}\le \frac{2l-1}{16}\delta^2_0, \quad
\|\p_y u_{10}\|_{L^\infty_1(\Omega)}\le (2\delta_0)^{-1}.
\end{equation}
Then, there exist a positive time $0<T_0=T_0(\mu, \k,m, l, \d_0,
\|(\rho_0, u_{10}, h_{10})\|_{\overline{\mathcal{B}}^m_l},
\|(\rho_0, u_{10}, h_{10})\|_{\widehat{\mathcal{B}}^m_l})$ and a unique solution
$(\rho, u_1, u_2, h_1, h_2)$ to the initial boundary value problem \eqref{eq3}-\eqref{bc3}, such that
\begin{equation}\label{main-estimate}
\begin{aligned}
&\sup_{0\le t \le T_0}\{\|(\rho-1, u_1-1, h_1-1)(t)\|_{\H^m_l}^2
+\|\p_y(\rho, u_1, h_1)(t)\|_{\H^{m-1}_l}^2
+\| \p_y \rho(t)\|_{\H^{1,\infty}_1}^2\}\\
&\quad +\int_0^{T_0} \|\p_y(\sqrt{\mu} u_1, \sqrt{\k}h_1)(t)\|_{\H^m_l}^2 dt
      +\int_0^{T_0} \|\p_y^2(\sqrt{\mu} u_1, \sqrt{\k}h_1)(t)\|_{\H^{m-1}_l}^2 dt
      \le \widehat{C}_0<+\infty,
\end{aligned}
\end{equation}
where $\widehat{C}_0$ depends only on $l, \d_0$, and $\|(\rho_0, u_{10}, h_{10})\|_{\overline{\mathcal{B}}^m_l}$.
\end{theorem}

\begin{remark}
Note that we choose the initial data with higher regularity and Conormal Sobolev space
as our basic space since we construct the approximation system \eqref{eq4} to
establish the well-posedness for the MHD boundary layer system \eqref{eq3}.
\end{remark}

\begin{remark}
Note that the approach for the well-posedness result in Theorem \ref{local} can be generalized to
study the nonlinear problem \eqref{eq3} with a non-trivial Euler outflow $(1, U, H)$
satisfying the equation \eqref{ma-co}.
\end{remark}

\begin{remark}
We should point out that the initial condition \eqref{condition2} is not required when
the incompressible magnetohydrodynamics flows are the case of homogeneous
(cf. Remark \ref{remark-condition}).
In other words, the local-in-time well-posedness of boundary layer system \eqref{eq3}
with any large initial data can be obtained only under the condition \eqref{condition1}
when the density is constant. This will improve the recent interesting result \cite{Liu-Xie-Yang}.
\end{remark}

We now review some related works to the problem studied in this paper.
The MHD system \eqref{eq1} is a combination of the inhomogeneous incompressible
Navier-Stokes equations of fluid dynamics and Maxwell's equations of electromagnetism.
Since the study for MHD system has been along with that for Navier-Stokes one,
let us recall some results about the inhomogeneous incompressible Navier-Stokes equations.
If the initial density is bounded away from zero, Kazhikov \cite{Kazhikov} proved that:
the inhomogeneous incompressible Navier-Stokes equations have at least one global weak solutions in the energy space.
This result can be generalized to case of initial data with vacuum(cf.\cite{{Simon},{Lions}}).
Choe and Kim \cite{Choe-Kim} proved the existence and uniqueness of local
strong solutions to the initial value problem or the initial boundary
value problem even though the initial vacuum exists.
Recently, the large-time decay and stability to any given global
smooth solutions of the 3D incompressible inhomogeneous Navier-Stokes equations were obtained \cite{Abidi}.
Let's go back to the MHD system \eqref{eq1}, it is known that Gerbeau and Le Bris\cite{GL}
(see also  Desjardins and  Le Bris \cite{BD}) established the global existence of weak solutions
of finite energy in the whole space or in the torus.
The global existence of strong solution with small initial data in some Besov spaces was
considered by Abidi and Paicu \cite{Abidi-Paicu}.
Recently, Gui \cite{Gui}  has shown that the 2D incompressible inhomogeneous magnetohydrodynamics system
with a constant viscosity is globally well-posed for a generic family of the variations of the
initial data and an inhomogeneous electrical conductivity.

When magnetic field vanishes, the MHD system \eqref{eq1} turns to be
the classical well-known incompressible Navier-Stokes equations if the density being constant.
As the viscosity $\varepsilon$ tends to zero, the Navier-Stokes equations
will become the Euler equations.
There are lots of literatures on the uniform bounds and the vanishing viscosity
limit for the Navier-Stokes equations without boundaries \cite{{Constantin}, {Constantin-Foias},{Kato}, {Masmoudi1}}.
The time of existence $T^\varepsilon$ always depend on the viscosity coefficient
when the boundary appears. It is difficult to prove that
the existence of time stays bounded away from zero.
However, for the domain with some special types of Navier-slip boundary conditions,
some uniform $H^3$ (or $W^{2,p}$, with $p$ large enough) estimates and
a uniform time of existence have recently been established \cite{{Beira1},{Beira2},{Xiao-Xin1}}.
This uniform control in some limited regularity Sobolev spaces
can be obtained because these special boundary conditions gives arise to
the main part of the boundary layer vanishes.
For the three dimensional domain with smooth boundary,
Masmoudi and Rousset \cite{Masmoudi-Rousset} recently
obtained conormal uniform estimates for the incompressible
Navier-Stokes equations with Naiver-slip type boundary condition.
Furthermore, they also applied the compact argument to establish
the convergence of the viscous solution to the inviscid ones.
This result was generalized to the compressible flow \cite{Wang-Xin-Yong},
which also shown that the boundary layers for density must be weaker than the one for the velocity.

The vanishing viscosity limit of the incompressible Navier-Stokes equations that,
in a bounded domain with Dirichlet boundary condition, is an important problem in
both physics and mathematics.
This is due to the formation of a boundary layer, where the solution undergoes a sharp transition
from a solution of the Euler system to the zero non-slip boundary condition on boundary of the
Navier-Stokes system. This boundary layer satisfies the Prandtl system formally.
Indeed, Prandtl \cite{Prandtl} derived the Prandtl equations for boundary layer
from the incompressible Navier-Stokes equations with non-slip boundary condition.
The first systematic work in rigorous mathematics was obtained by Oleinik \cite{{Oleinik3},{Oleinik4}},
in which she established the local in time well-posedness of the Prandtl equations
in dimension two by applying the Crocco transformation under the monotonicity condition
on the tangential velocity field in the normal direction to the boundary.
For more extensional mathematical results, the interested readers can refer to
the classical book finished by Oleinik and Samokhin \cite{Oleinik2}.
By taking care of the cancelation in the convection term to overcome the
loss of derivative in the tangential direction of velocity, the researchers
in \cite{Xu-Yang-Xu} and \cite{Masmoudi} independently used the simply
energy method to establish well-posedness theory for the
two-dimensional Prandtl equations in the framework of Sobolev spaces.
For more results in this direction, the interested readers can refer to \cite{{Grenier-Guo-Nguyen1},{Grenier-Guo-Nguyen2},{Gie-Temam}}
and references therein.

Under the influence of electro-magnetic field, the system of
magnetohydrodynamics(denoted by MHD) is a fundamental system
to describe the movement of electrically conducting fluid,
for example plasmas and liquid metals(cf.\cite{Alfven}).
On one hand, G\'{e}rard-Varet and Prestipino \cite{G-V-P}
provided a systematic derivation of boundary layer models in magnetohydrodynamics,
through an asymptotic analysis of the incompressible MHD system.
Furthermore, they also performed some stability analysis for the boundary layer
system, and emphasized the stabilizing effect of the magnetic field.
On the other hand, if both the hydrodynamic Reynolds numbers and magnetic Reynolds numbers
tend to infinity at the same rate, the local in time well-posedness of
the boundary layer system was obtained \cite{Liu-Xie-Yang} if
there exists a small constant $\d_0$ such that
\begin{equation}\label{Yang-C}
|\la y \ra^{l+1}\p_y^i(u_{10}, h_{10})(x, y)|\le (2\d_0)^{-1},~{\rm for~}i=1,2,~(x, y)\in \O,
\end{equation}
and \eqref{condition1} hold on. In other words, the local in time well-posedness of MHD boundary layer
system holds on under the condition on the initial tangential magnetic field is not zero
instead of the monotonicity condition on the tangential velocity field.

Finally, we point out that it is an outstanding open problem to rigorously justify
the validity of expansion in the inviscid limit.
On one hand, Sammartino and Caflisch \cite{{Sammartino-Caflisch1},{Sammartino-Caflisch2}}
obtained the well-posedness in the framework of analytic functions without the
monotonicity condition on the velocity field and justified the boundary layer expansion
for the unsteady incompressible Navier-Stokes equations.
Furthermore, Guo and Nguyen \cite{Guo-Nguyen1} concerned nonlinear ill-posedness of the Prandtl equation
and an invalidity of asymptotic boundary layer expansion of incompressible fluid flow
near a solid boundary. Furthermore, they also shown that the asymptotic boundary layer expansion was not
valid for nonmonotonic shear layer flow in Sobolev spaces and verified that Oleinik's monotonic solutions were well-posed.
For the incompressible steady Navier-Stokes equations, Guo and Nguyen\cite{Guo-Nguyen2} justified the boundary
layer expansion for the flow with a non-slip boundary condition on a moving plate.
As the magnetic field appears, the Prandtl ansatz boundary layer expansion for the unsteady MHD system
was justified \cite{Liu-Xie-Yang2} when no-slip boundary and
perfect conducting boundary conditions are imposed on velocity field and magnetic field respectively.

The rest of this paper is organized as follows.
In section \ref{sa}, we explain the main difficulty and our approach to establish the
local-in-time well-posedness theory for the Prandtl type equation \eqref{eq3}.
In Section \ref{A-P-Est}, one establishes the a priori estimates for the nonlinear problem \eqref{eq5}.
The local-in-time existence and uniqueness of equation \eqref{eq3}
in Weighted Conormal Sobolev space are given in Section \ref{local-in-times}.
Finally, some useful inequalities and important equivalent relations\
will be stated in Appendixs \ref{appendixA} and \ref{appendixB}.

Before we proceed, let us comment on our notation.
Throughout this paper, all constants $C$ may be different in different lines.
Subscript(s) of a constant illustrates the dependence of the constant, for example,
$C_s$ is a constant depending on $s$ only.
Denote by $\p_y^{-1}$ the inverse of the derivative $\p_y$, i.e.,
$(\p_y^{-1}f)(y):=\int_0^y f(z)dz$.
Moreover, we also use the notation $[A, B]=AB-BA$, to denote the commutator between $A$ and $B$.
Finally, $\mathcal{P}_i(\cdot, \cdot)$ stands for a polynomial function independent of $\es$,
and the index $i$ denote it changing from line to line.

\section{Difficulties and Outline of Our Approach}\label{sa}

The main of this section is to explain main difficulties of proving Theorem \ref{local}
as well as our strategies for overcoming them.
In order to solve the Prandtl type equation \eqref{eq3} in certain $H^m$ Sobolev space,
the main difficulty comes from the vertical velocity $u_2=-\p_y^{-1}\p_x u_1$
(and vertical magnetic field $h_2=-\p_y^{-1}\p_x h_1$) creates
a loss of $x-$derivative, so the standard energy estimates can not apply directly.

The main idea of establishing the well-posedness of inhomogeneous incompressible
MHD boundary layer equations \eqref{eq3}-\eqref{bc3} is to apply the so-called
\emph{vanishing viscosity and nonlinear cancelation methods}.
To this end, we consider the following approximate problem:
\begin{equation}\label{eq4}
\left\{
\begin{aligned}
&\p_t \r+u_1^\es \p_x \r+ u_2^\es \p_y \r-\es\p_x^2 \r-\es \p_y^2 \r
  =-\es \p_x r_1-\es \p_y r_2,\\
&\r \p_t u_1^\es+\r u_1^\es \p_x u_1^\es+ \r u_2^\es \p_y u_1^\es
-\es \p_x^2 u_1^\es-\mu \p_y^2 u_1^\es=h_1^\es \p_x h_1^\es+h_2^\es \p_y h_1^\es-\es \p_x r_u,\\
&\p_t h_1^\es+\p_y(u_2^\es h_1^\es-u_1^\es h_2^\es)
-\es \p_x^2 h_1^\es-\k \p_y^2 h_1^\es=-\es \p_x r_h,\\
&\p_x u_1^\es+\p_y u_2^\es=0, \quad \p_x h_1^\es+\p_y h_2^\es=0,
\end{aligned}
\right.
\end{equation}
for any parameter $\es>0$. Here the functions $r_1, r_2, r_u$ and $r_h$ are defined by
\begin{equation}\label{rdef}
(r_1, r_2, r_u, r_h)(t, x, y)=\sum_{i=0}^{m-1} \frac{t^i}{i!}\partial_t^i
(\p_x \rho, \p_y \rho, \p_x u_{1}, \p_x h_{1})(0,x,y),
\end{equation}
which gives that by direct calculation
\begin{equation}\label{ID-def}
\p_t^i(\r, u_1^\es, h_1^\es)(0, x, y)=\p_t^i(\rho, u_1, h_1)(0, x, y),\quad 0\le i \le m.
\end{equation}
To complete the system \eqref{eq4}, the boundary conditions are given by
\begin{equation}\label{bc4}
\left\{
\begin{aligned}
&\left.\p_y \r|_{y=0}=u^{\es}_1\right|_{y=0}=\left. u^{\es}_2\right|_{y=0}
=\left.\p_y h^{\es}_1 \right|_{y=0}=\left. h^{\es}_2\right|_{y=0}=0,\\
&\lim_{y\rightarrow +\infty }\r(t,x,y)
=\lim_{y\rightarrow +\infty }u^{\es}_1(t,x,y)
=\lim_{y\rightarrow +\infty }h^{\es}_1(t,x,y)=1.
\end{aligned}
\right.
\end{equation}
Since the local-in-time existence and uniqueness of regularized Eqs.\eqref{eq4}-\eqref{bc4}
can be obtained easily in $H^m$ Sobolev space for any $\es>0$, we hope that
the solution $(\rho^\es, u_1^{\es}, u_2^\es, h_1^\es, h_2^\es)$ of regularized equation \eqref{eq4}
will converge to the solution $(\rho, u_1, u_2, h_1, h_2)$ of original Prandtl type
equation \eqref{eq3} as $\es$ tends to zero.
To this end, we need to get the uniform \emph{a priori} estimates of solution
$(\rho^\es, u_1^{\es}, u_2^\es, h_1^\es, h_2^\es)$
in an existence time independent of $\es$.
Although the idea of local-in-times well-posedness of MHD boundary layer equation,
which only needs that the background tangential magnetic field has a lower positive bound instead of monotonicity
assumption on the tangential velocity, comes from the recent result \cite{Liu-Xie-Yang},
we have to overcome some essential difficulties when the density of fluid changes from a constant to unknown quantity.

First of all, we should work on the Conormal Sobolev space to obtain some energy estimates
independent of small coefficient $\es$ since there is boundary condition for the
first equation of \eqref{eq4}.
To control the vertical velocity $u_2^\es=-\p_y^{-1} \p_x u_1^\es$ by the horizontal velocity $u_1^\es$,
we need to apply the Hardy type inequality by adding a weight $(1+y)^1$.
Since the conormal derivative $Z_2=\varphi(y) \p_y$ does not communicate with the normal
derivative $\p_y$, we need to choose the Sobolev space with suitable weight(actually taking $l \ge 2$)
to close the energy estimate, which is the first novelty in our paper.

Similar to the Prandtl equation, the main difficulty in the analysis on the system \eqref{eq4}
in the Sobolev framework is the loss of $x-$derivative in the vertical components
$u_2^\es$ and $h_2^\es$ appearing in the terms
$u_2^\es \p_y \r, \r u_2^\es \p_y u_1^\es-h_2^\es \p_y h_1^\es$ and
$u_2^\es \p_y h_1^\es-h_2^\es \p_y u_1^\es$ in the first, second and third equations of \eqref{eq4}, respectively.
Motivated by the recent result \cite{Liu-Xie-Yang}, we construct some quantities $(\vr_m, \u_m, \h_m)$
(see the definitions \eqref{equi-r}, \eqref{eqi-u}, \eqref{equi-h} respectively) to avoid the loss of $x$ derivative
and obtain the estimate for $(\vr_m, \u_m, \h_m)$ in $L^2_l-$norm independent of $\es$ by the energy method.
To establish the relation between the quantities $(\vr_m, \u_m, \h_m)$ and $\p_x^m(\rho^\es, u_1^\es, h_1^\es)$,
we need to control the quantity $\p_y(\rho^\es, u_1^\es, h_1^\es)$ in $L^\infty_1-$norm.
To this end, we apply the low order tangential derivative estimate $\mathcal{E}_{m,l}(t)$(see the definition \eqref{eml})
to control the quantity $\p_y(u_1^\es, h_1^\es)$ in $L^\infty_1-$norm, which can be achieved by the
Sobolev embedding inequality. Then, it is easy to get the almost equivalent relation $X_{m, l}(t) \sim Y_{m,l}(t)$
(see the definitions in \eqref{ydef} and \eqref{xdef} respectively).
This is the second novelty in our paper, and avoid the important condition \eqref{Yang-C} required
in \cite{Liu-Xie-Yang} for the MHD boundary layer equation with constant density .


\section{A Priori Estimates}\label{A-P-Est}

In this section, we will establish a priori estimates(independent of $\es$),
which are crucial to prove the Theorem \ref{local}.
First of all, let us define
\begin{equation}\label{transf}
\vr:=\r-1,\
\u:=u_1^\es-1+e^{-y},\ \v=u_2^\es,\quad
\h:=h_1^\es-1, \ \g=h_2^\es,
\end{equation}
then it follows from equation \eqref{eq4} that
\begin{equation}\label{eq5}
\left\{
\begin{aligned}
&\p_t \vr+(u^\es+1-e^{-y})\p_x \vr +v^\es \p_y \vr
 -\es\p_x^2 \vr-\es \p_y^2 \vr=-\es \p_x r_1-\es \p_y r_2,\\
&\rho^\es \partial_t u^\es+\rho^\es(u^\es+1-e^{-y})\partial_x u^\es+\rho^\es v^\es\partial_y u^\es
+\rho^\es v^\es e^{-y}\\
&\quad =\es \partial_x^2 u^\es+\mu \partial_y^2 u^\es
+(h^\es+1)\partial_x h^\es+g^\es \partial_y h^\es-\es \p_x r_u-\mu e^{-y},\\
&\partial_t h^\es+(u^\es+1-e^{-y})\partial_x h^\es+v^\es \partial_y h^\es
-\es \partial_x^2 h^\es-\k \partial_y^2 h^\es
=(h^\es+1)\partial_x u^\es+g^\es\partial_y(u^\es-e^{-y})-\es \p_x r_h,\\
&\p_x \u+\p_y \v=0, \quad \p_x \h+\p_y \g=0,\\
&(\vr,\u,\h)|_{t=0}:=(\vr_0,\u_0,\h_0),
\end{aligned}
\right.
\end{equation}
with the boundary conditions
\begin{equation}\label{bc5}
\left\{
\begin{aligned}
&\partial_y \varrho^\es|_{y=0}=u^\es|_{y=0}=v^\es|_{y=0}=
 \partial_y h^\es|_{y=0}=g^\es|_{y=0}=0,\\
&\underset{y\rightarrow +\infty}{\lim}\varrho^\es
 =\underset{y\rightarrow +\infty}{\lim}u^\es
 =\underset{y\rightarrow +\infty}{\lim}h^\es=0,\\
\end{aligned}
\right.
\end{equation}
Due to the relation \eqref{transf}, we can get the relation between two initial data as follow:
\begin{equation}\label{id-relation}
\vr_0=\rho_0-1,\quad \u_0=u_{10}-1+e^{-y},\quad \h_0=h_{10}-1,
\end{equation}
and hence, we have the estimates:
\begin{equation*}
\|(\vr_0, \u_0, \h_0)\|_{\H^m_l}^2
+\|\p_y(\vr_0, \u_0, \h_0)\|_{\H^{m-1}_l}^2
+\| \p_y \vr_0\|_{\H^{1,\infty}_1}^2
\le C(1+\|(\rho_0, u_{10}, h_{10})\|_{\overline{{\mathcal{B}}}^m_l}),
\end{equation*}
and
\begin{equation*}
\|(r_1, r_2, r_u, r_h)(t)\|_{\H^m_l}^2
+\|\p_y(r_1, r_2, r_u, r_h)(t)\|_{\H^{m-1}_l}^2
+\|\p_y(\p_x r_1, \p_y r_2)(t)\|_{\H^{1,\infty}_0}^2
\le C\|(\rho_0, u_{10}, h_{10})\|_{\widehat{\mathcal{B}}^m_l}.
\end{equation*}
Here the norms $\|\cdot \|_{\overline{{\mathcal{B}}}^m_l}$ and
$\|\cdot \|_{\widehat{\mathcal{B}}^m_l}$
are defined by \eqref{norm-BX} and \eqref{norm-BY} respectively,
and the time derivatives of initial data
are defined through the MHD boundary layer equation \eqref{eq3}.

Let us define
\begin{equation}
\begin{aligned}
\ta_{m,l}(\vr, \u, \h)(t):=
&\sup_{0\le s \le t}\{1+\|(\vr, \u, \h)(s)\|_{\H^m_l}^2+\|\p_y(\vr, \u, \h)(s)\|_{\H^{m-1}_l}^2
+\| \p_y \vr (s)\|_{\H^{1,\infty}_1}^2\}\\
&+\es \int_0^t\|\p_x(\vr, \u, \h)\|_{\H^m_l}^2d\t
 +\int_0^t \|\p_y(\sqrt{\es}\vr, \sqrt{\mu}\u, \sqrt{\k}\h)\|_{\H^m_l}^2d\t\\
&+\es \int_0^t\|\p_{xy}(\vr, \u, \h)\|_{\H^{m-1}_l}^2d\t
 +\int_0^t \|\p_y^2(\sqrt{\es}\vr, \sqrt{\mu}\u, \sqrt{\k}\h)\|_{\H^{m-1}_l}^2d\t.
\end{aligned}
\end{equation}
Next, we will prove the following a priori estimates independent of $\es$
for the regularized MHD boundary layer equations \eqref{eq5}-\eqref{bc5}.

\begin{theorem}[\emph{a priori estimates}]\label{theo a priori}
Let $m \ge 5$ be an integer, $l \ge 2$ be a real number
and $\epsilon \in (0, 1]$, and $(\vr, \u, \v, \h, \g)$ be sufficiently smooth solution,
defined on $[0, T^\es]$, to the regularized MHD boundary layer equations \eqref{eq5}-\eqref{bc5}.
The initial data $(\vr_0,\u_0,\h_0)$ is defined by $(\rho_0, u_{10}, h_{10})$
given in Theorem \ref{local} through the relation \eqref{id-relation}.
Then, there exists a time $T_a=T_a(\mu, \k,m, l, \d_0,
\|(\rho_0, u_{10}, h_{10})\|_{\overline{\mathcal{B}}^m_l},
\|(\rho_0, u_{10}, h_{10})\|_{\widehat{\mathcal{B}}^m_l})>0$
independent of $\es$ such the following a priori estimates hold
for all $t \in [0, \min(T_a, T^\es)]$:
\begin{equation}\label{3a1}
\ta_{m,l}(\vr, \u, \h)(t) \le 2 C_l \mathcal{P}_0(\d^{-1}_0, \|(\rho_0, u_{10}, h_{10})\|_{\overline{\mathcal{B}}^m_l}),
\end{equation}
and
\begin{equation}\label{3a2}
\|\p_y(\u-e^{-y})(t)\|_{L^\infty_1(\O)} \le \d_0^{-1},\quad
\|\vr(t)\|_{L^\infty_0(\O)}\le \frac{3(2l-1)}{32}\delta^2_0,\quad
\h(t,x, y)+1 \ge \d_0,
\end{equation}
for all $(t, x, y) \in [0, \min(T_a, T^\es)] \times \O$.
\end{theorem}

\begin{remark}
When the parameter $\es=0$, the regularized Prandtl type equations \eqref{eq5} become
the original Prandtl type equation \eqref{eq3}, and hence Theorem \ref{theo a priori}
also provides a priori estimates for the original MHD boundary layer equation \eqref{eq3}.
\end{remark}

Throughout this section, for any small constant $\d$,
we assume that the following a priori assumptions:
\begin{equation}\label{a2}
h^\es(t,x,y)+1\ge \d,
\end{equation}
and
\begin{equation}\label{a1}
\|\vr(t)\|_{L^\infty_0(\O)}\le \frac{2l-1}{2}\d^2,\quad
\|\p_y (\u-e^{-y})(t)\|_{L^\infty_1(\O)}\le \d^{-1},
\end{equation}
hold on for any $(t, x, y) \in [0, T^\es] \times \O$.
Thanks to the smallness of $\d$, we find
\begin{equation}\label{low-bou}
\frac{1}{2}\le \rho^\es(t, x, y)\le \frac{3}{2}
\end{equation}
for $(t, x, y)\in [0, T^\es] \times \O$.

\subsection{Weighted $\H^m_l-$Estimates with Conormal Derivative}

In this subsection, we will derive the weighted estimates for the quantities
$Z_\t^{\a_1} Z_2^{\a_2}(\vr, \u, \h)$ with $|\a_1|+|\a_2|=m, |\a_1|\le m-1$.
This goal is easy to reach by the standard energy method because one order
tangential derivative loss is allowed.
For notational convenience, we denote
\begin{equation}\label{eml}
\mathcal{E}_{m,l}(t):=\sum_{\substack{ |\alpha| \le m \\
    |\a_1| \le m-1}}
    \|\z^\a(\vr, \u, \h)(t)\|_{L^2_l(\O)}^2,
\end{equation}
and
\begin{equation}\label{Q}
\begin{aligned}
Q(t):=\underset{0\le s \le t}{\sup}\{
&\|Z_\t \vr(s)\|_{L^\infty_0(\O)}^2+\|(\u, \h)(s)\|_{\H^{1,\infty}_{0, tan}}^2
+\|(\v, \g)(s)\|_{\H^{1, \infty}_{1, tan}}^2\\
&+\|(\p_y \vr, \p_y \u, \p_y \h, \frac{\v}{\varphi})(s)\|_{\H^{1,\infty}_1}^2\}.
\end{aligned}
\end{equation}

\begin{proposition}\label{Lower-estimate}
Let $(\vr, \u, \v, \h, \g)$ be sufficiently smooth solution, defined on $[0, T^\es]$,
to the equations \eqref{eq5}-\eqref{bc5}. Then, it holds on
\begin{equation*}
\begin{aligned}
&\sup_{0\le \t \le t}\e_{m,l}(\t)
+\sum_{\substack{0 \le |\alpha| \le m \\ |\a_1| \le m-1}}
\es \int_0^t \|\p_x  \z^\a(\vr, \u, \h)\|_{L^2_l(\O)}^2 d\tau\\
& + \sum_{\substack{0 \le |\alpha| \le m \\ |\a_1| \le m-1}}
\int_0^t \|(\sqrt{\es} \p_y  \z^\a \vr, \sqrt{\mu} \p_y  \z^\a \u, \sqrt{\k} \p_y  \z^\a\h)\|_{L^2_l(\O)}^2 d\tau\\
\le
& C\|(\vr_0, \u_0, \h_0)\|_{\H^m_l}^2
+C t \|(\rho_0, u_{10}, h_{10})\|_{\widehat{\mathcal{B}}^m_l}\\
&+C_{\mu, \k, m, l}(1+Q^2(t))\int_0^t (1+\|(\vr, \u, \h)\|_{\H^{m}_l}^2+\|\p_y(\vr, \u, \h)\|_{\H^{m-1}_l}^2)d\tau.
\end{aligned}
\end{equation*}
\end{proposition}

The Proposition \ref{Lower-estimate} will be proved in Lemma \ref{lemma32}.
Now, we give the proof for the case $m=0$.

\begin{lemma}\label{lemma31}
For smooth solution $(\vr, \u, \v, \h,\g)$ of the equations \eqref{eq5}-\eqref{bc5}, then it holds on
\begin{equation}\label{31}
\begin{aligned}
&\sup_{\tau\in [0, t]} \|(\vr, \u, \h)(\t)\|_{L^2_l(\Omega)}^2
+\es \int_0^t \|\p_x(\vr, \u, \h)\|_{L^2_l(\Omega)}^2 d\tau
+\int_0^t \|\p_y(\sqrt{\es} \vr, \sqrt{\mu} \u, \sqrt{\k} \h)\|_{L^2_l(\Omega)}^2 d\tau\\
&\le C\|(\vr_0, \u_0, \h_0)\|_{L^2_l(\Omega)}^2
+\!C\!\!\int_0^t\!\! \|(r_1, r_2, r_u, r_h)\|_{L^2_l(\Omega)}^2d\t
+\!C_{\mu, \k, l}(1+Q(t))\!\int_0^t \!\!(1+ \|(\vr,\u,\h)\|_{\H^1_l}^2)d\tau.
\end{aligned}
\end{equation}
\end{lemma}

\begin{proof}
First of all, multiplying \eqref{eq5}$_2$ by $\y \u$, integrating over $\O $
and integrating by parts, we find
\begin{equation}\label{311}
\begin{aligned}
&\frac{d}{dt}\frac{1}{2}\int_\O \ya \r |\u |^2 dxdy
+\es \int_\O  \y |\p_x \u|^2dxdy
+\mu \int_\O  \ya |\p_y \u |^2dxdy\\
&=\frac{1}{2} \int_\O  \ya |\u |^2(\p_t \r+(\u+1-e^{-y})\p_x \r
   +\v \p_y \r )dxdy\\
&\quad +l\int_\O  \yb \r  \v |\u |^2 dxdy
-2\l \mu \int_\O  \yb \u \cdot \p_y \u dxdy\\
&\quad
 -\int_\O  \rho^\es v^\es e^{-y}\cdot \la y\ra^{2l}u^\es dxdy
 +\int_\O  [(h^\es+1)\partial_x h^\es+g^\es \partial_y h^\es]\cdot \la y\ra^{2l}u^\es dxdy\\
&\quad +\int_\O (-\es \p_x r_u-\mu e^{-y})\cdot \ya \u dxdy,
\end{aligned}
\end{equation}
where we have used the boundary condition \eqref{bc5} and the divergence free condition \eqref{eq5}$_4$.
By routine checking, we have, after using the definition of $Q(t)$ in \eqref{Q},
\begin{equation*}\label{312}
\begin{aligned}
|\int_\O  \ya |\u |^2(\p_t \r+(\u+1-e^{-y})\p_x \r+\v \p_y \r )dxdy|
\le C(1+Q(t))\|\u\|_{L^2_l(\Omega)}^2,
\end{aligned}
\end{equation*}
By virtue of estimate for density \eqref{low-bou}, we get
\begin{equation*}\label{313}
|\int_\O \yb \r \v |\u|^2 dxdy|\le C\|\la y \ra^{-1} \v\|_{L^\infty_0(\O)}\|\u\|_{L^2_l(\Omega)}^2.
\end{equation*}
Using the H\"{o}lder and Cauchy-Schartz inequalities, it follows
\begin{equation*}\label{314}
|\mu \int_\O \yl \u \cdot \p_y \u dxdy|
\le \frac{\mu}{4}\int_\O \y |\p_y \u|^2 dxdy
+C_\mu \int_\O \la y \ra^{2(l-1)}|\u|^2 dxdy.
\end{equation*}
and
\begin{equation*}\label{315}
|\int_\O (-\es \p_x r_u-\mu e^{-y})\cdot \y \u dxdy|
\le \frac{1}{2}\es \|\p_x \u\|_{L^2_l(\O)}^2+ C(1+\|r_u\|_{L^2_l(\Omega)}^2+\|\u\|_{L^2_l(\Omega)}^2).
\end{equation*}
Applying the divergence-free condition \eqref{eq5}$_4$, H\"{o}lder and Hardy inequalities, we get
\begin{equation*}\label{316}
|\int_\O \r \v e^{-y}\cdot \y \u dxdy|
\le C\|\v\|_{L^2_{l-1}(\Omega)}\|\u\|_{L^2_{l}(\O )}
\le C_l\|\p_y \v\|_{L^2_{l}(\O )}\|\u\|_{L^2_{l}(\O )}
\le C_l\|\u\|_{\H^1_l}^2.
\end{equation*}
Integrating by part and applying the divergence-free condition \eqref{eq5}$_4$, we find
\begin{equation*}\label{317}
\begin{aligned}
&\int_\O [(\h+1)\p_x \h+\g \p_y \h]\cdot \y \u dxdy\\
=
&-\int_\O \y \p_x \h \u \h dxdy-\int \y \h (\h+1) \p_x \u dxdy\\
&-\int_\O \y \p_y \g \u \h dxdy-2l\int \yl \g  \u \h dxdy\\
&-\int_\O \y \h \g  \p_y \u dxdy\\
\le
&-\int_\O \y \h (\h+1) \p_x \u dxdy-\int \y \h \g  \p_y \u dxdy\\
&+C\|\la y \ra^{-1} \g\|_{L^\infty_0(\O)}\|\u\|_{L^2_l(\Omega)}\|\h\|_{L^2_l(\Omega)}.
\end{aligned}
\end{equation*}
Substituting the above estimates into \eqref{311},
and integrating the resulting inequality over $[0, t]$, we obtain
\begin{equation*}
\begin{aligned}
&\|\sqrt{\r} \u(t)\|_{L^2_l(\Omega)}^2
+\es \int_0^t\|\p_x \u\|_{L^2_l(\Omega)}^2 d\tau
+\mu \int_0^t\|\p_y \u\|_{L^2_l(\Omega)}^2 d\tau\\
&+\int_0^t \int_\O \y \h\cdot ( (\h+1) \p_x \u + \g  \p_y \u )dxdyd\tau\\
&\le \|\sqrt{\r_0} \u_0\|_{L^2_l(\Omega)}^2+C\int_0^t\|r_u\|_{L^2_l(\Omega)}^2d\t
+C_{\mu, l}(1+Q(t))\int_0^t (1+\|\u\|_{\H^1_l}^2+\|\h\|_{L^2_l(\Omega)}^2) d\tau.
\end{aligned}
\end{equation*}
Similarly, based on the equations \eqref{eq5}$_3$ and \eqref{eq5}$_1$, it follows directly
\begin{equation*}
\begin{aligned}
&\|\h(t)\|_{L^2_l(\Omega)}^2
+\es \int_0^t \|\p_x \h\|_{L^2_l(\O)}^2 d\tau
+\k  \int_0^t \|\p_y \h\|_{L^2_l(\O)}^2 d\tau\\
&-\int_0^t \int_\O \y \h \cdot [(\h+1) \p_x \u + \g  \p_y \u ]dxdyd\tau\\
&\le \|\h_0\|_{L^2_l(\Omega)}^2+C\int_0^t \|r_h\|_{L^2_l(\O)}^2d\t
      +C_{\k, l}(1+Q(t))\int_0^t (1+\|\h\|_{\H^1_l}^2) d\tau.
\end{aligned}
\end{equation*}
and
\begin{equation*}
\|\vr\|_{L^2_l(\Omega)}^2
\!+\es \!\int_0^t\! \|(\p_x \vr, \p_y \vr)\|_{L^2_l(\Omega)}^2 d\tau
\le \|\vr_0\|_{L^2_l(\Omega)}^2+C\!\int_0^t \!\|(r_1, r_2)\|_{L^2_l(\O)}^2d\t
\!+C(1+Q(t))\!\int_0^t\! \|\vr\|_{L^2_l(\Omega)}^2 d\tau.
\end{equation*}
Therefore, we collect above estimates to complete the proof of Lemma \ref{lemma31}.
\end{proof}

Now, we establish the following estiamte:

\begin{lemma}\label{lemma32}
For smooth solution $(\vr, \u, \v, \h,\g)$ of the equations \eqref{eq5}-\eqref{bc5}, then it holds on
\begin{equation}\label{321}
\begin{aligned}
&\sup_{0\le \t \le t}\e_{m,l}(\t)
+\sum_{\substack{|\alpha| \le m \\ |\a_1| \le m-1}}
\es \int_0^t \|\p_x  \z^\a(\vr, \u, \h)\|_{L^2_l(\O)}^2 d\tau\\
& + \sum_{\substack{|\alpha| \le m \\ |\a_1| \le m-1}}
\int_0^t \|(\sqrt{\es} \p_y  \z^\a \vr, \sqrt{\mu} \p_y  \z^\a \u, \sqrt{\k} \p_y  \z^\a\h)\|_{L^2_l(\O)}^2 d\tau\\
\le
& C\|(\vr_0, \u_0, \h_0)\|_{\H^m_l}^2
+C\int_0^t \|(r_1, r_2, r_u, r_h)\|_{\H^m_l}^2 d\t\\
&+C_{\mu, \k, m, l}(1+Q^2(t))\int_0^t (1+\|(\vr, \u, \h)\|_{\H^{m}_l}^2+\|\p_y(\vr, \u, \h)\|_{\H^{m-1}_l}^2)d\tau.
\end{aligned}
\end{equation}
\end{lemma}

\begin{proof}
The case $m=0$ has been proved in the Lemma \ref{lemma31}. Then, we give the proof for the case $m \ge 1$.

Step 1: Applying the differential operator $\z^\alpha(1 \le |\a|\le m, |\a_{1}| \le m-1)$
to the equation \eqref{eq5}$_2$, we can obtain the evolution equation for $\z^\a u^\es$:
\begin{equation}\label{322}
\begin{aligned}
&\r \p_t \z^\a u^\es+\r(\u+1-e^{-y})\p_x \z^\a \u
+\r \v \p_y \z^\a \u-\es \p_x^2 \z^\a \u-\mu \z^\a \p_y^2 \u\\
=&(\h+1)\p_x \z^\a \h+\g \p_y \z^\a \h
+\z^\a(-\es \p_x r_u-\mu e^{-y})+C_{11}^\a+\C_{12}^\a+\C_{13}^\a+\C_{14}^\a+\C_{15}^\a+\C_{16}^\a.
\end{aligned}
\end{equation}
where $\C_{1i}^\a(i=1,...,6)$ are defined by
\begin{equation*}\label{323}
\begin{aligned}
&\C_{11}^\a=-[\z^\a, \r \p_t]\u,\quad
\C_{12}^\a=-[\z^\a, \r(\u+1-e^{-y})\p_x]\u, \quad
\C_{13}^\a=-[\z^\a, \r \v \p_y]\u,\\
&\C_{14}^\a=[\z^\a, (\h+1)\p_x]\h,\quad
\C_{15}^\a=[\z^\a, \g \p_y]\h,\quad
\C_{16}^\a=-\z^\a(\r \v e^{-y}).
\end{aligned}
\end{equation*}
Multiplying the equation \eqref{322} by $\ya \z^\a \u$, integrating over $\Omega$ and
applying the boundary condition \eqref{bc5}, we have
\begin{equation}\label{324}
\begin{aligned}
&\frac{d}{dt}\frac{1}{2}\int_\O \ya \r |\z^\a \u|^2 dxdy
+\es \int_\O \ya |\p_x \z^\a \u|^2 dxdy\\
=
&\frac{1}{2}\int_\O \ya |\z^\a \u|^2(\partial_t \rho^\es+(u^\es+1-e^{-y})\partial_x \rho^\es+v^\es\partial_y \rho^\es)dxdy\\
&+l\int_\O \yb \r \v |\z^\a \u|^2 dxdy+\mu \int_\O \z^\a \p_y^2 \u \cdot \ya \z^\a \u dxdy\\
&+\int_\O \{(\h+1)\p_x \z^\a \h+\g \p_y \z^\a \h\}\cdot \ya \z^\a \u dxdy\\
&+\int_\O \z^\a(-\es \p_x r_u-\mu e^{-y})\cdot \ya \z^\a \u dxdy
 +\sum_{i=1}^6\int_\O \C_{1i}^\a \cdot \ya \z^\a \u dxdy.
\end{aligned}
\end{equation}
By routine checking, it follows that
\begin{equation}\label{325}
|\int_\O \ya |\z^\a \u|^2(\partial_t \rho^\es+(u^\es+1-e^{-y})\partial_x \rho^\es+v^\es\partial_y \rho^\es)dxdy|
\le C(1+Q(t))\|\z^\a \u\|_{L^2_{l }(\Omega)}^2,
\end{equation}
and
\begin{equation}\label{326}
|\int_\O \yb \r \v |\z^\a \u|^2 dxdy|
\le C\|\la y \ra^{-1} \v \|_{L^\infty_0(\Omega)}\|\z^\a \u\|_{L^2_{l }(\Omega)}^2.
\end{equation}
The integration by part with respect to $y$ variable yields directly
\begin{equation}\label{327}
\begin{aligned}
&\int_\O \z^\a \p_y^2 \u \cdot \ya \z^\a \u dxdy\\
=
&\int_\T \z^\a  \p_y \u \cdot  \z^\a \u|_{y=0} dx
-2l\int_\O  \z^\a  \p_y \u  \cdot \yb \z^\a  \u dxdy\\
&-\int_\O  \z^\a  \p_y \u  \cdot \ya  \p_y \z^\a \u dxdy
+\int_\O [\z^\a, \p_y] \p_y \u \cdot \ya \z^\a \u dxdy\\
=
&\int_\T \z^\a  \p_y \u \cdot  \z^\a \u|_{y=0} dx
-2l\int_\O  \z^\a  \p_y \u  \cdot \yb \z^\a  \u dxdy\\
&-\int_\O \ya  |\p_y \z^\a \u|^2 dxdy
-\int_\O  [\z^\a,  \p_y] \u  \cdot \ya  \p_y \z^\a \u dxdy\\
&
+\int_\O [\z^\a, \p_y] \p_y \u \cdot \ya \z^\a \u dxdy.
\end{aligned}
\end{equation}
If $\a_2=0$, the boundary condition \eqref{bc5} implies $\z^\a \u|_{y=0}=0$.
If $\a_2\neq 0$, we apply the property of $\varphi$, which vanishes on the boundary,
to get $\z^\a \u|_{y=0}$, and hence
\begin{equation}\label{328}
\int_\T \z^\a  \p_y \u \cdot  \z^\a \u|_{y=0} dx=0.
\end{equation}
Next, we deal with term involving $[\z^\a, \p_y]$.
It is worth noting that the operator $Z_\tau=(\p_t, \p_x)$ communicates with $\p_y$,
we obtain $[\z^\a, \p_y]\u=0$ for $\a_2=0$.
By direct computation, we find for $\a_2 \neq 0$
\begin{equation*}\label{329}
[Z_2^{\a_2}, \p_y]\u
=-\sum_{1\le k \le \a_2}C_{\a_2, k}\p_y Z_2^{k-1} \varphi \cdot Z_2^{\a_2-k}\p_y \u.
\end{equation*}
This and the H\"{o}lder inequality yield directly
\begin{equation}\label{3210}
|\int  [\z^\a,  \p_y] \u  \cdot \ya  \p_y \z^\a \u dxdy|
\le
\frac{1}{4}\|\p_y \z^\a \u\|_{L^2_{l }(\Omega)}^2
 +C_m \|\p_y \u\|_{\H^{m-1}_{l}}^2.
\end{equation}
Using the relation \eqref{329} and Cauchy-Schwartz inequality, we obtain
\begin{equation}\label{3211}
\begin{aligned}
&2l|\int_\O  \z^\a  \p_y \u  \cdot \yb \z^\a  \u dxdy|\\
=
& 2l|\int_\O  ([\z^\a, \p_y]\u+\p_y \z^\a \u) \cdot \yb \z^\a  \u dxdy|\\
\le
&\frac{1}{4}\|\p_y \z^\a \u\|_{L^2_{l}(\Omega)}^2
+C_{m,l}(\|[\z^\a, \p_y]\u\|_{L^2_{l}(\Omega)}^2
   +\|\z^\a \u\|_{L^2_{l}(\Omega)}^2)\\
\le
&
\frac{1}{4}\|\p_y \z^\a \u\|_{L^2_{l}(\Omega)}^2
 +C_{m,l}(\| \u\|_{\H^{m}_{l}}^2+\|\p_y \u\|_{\H^{m-1}_{l}}^2).
\end{aligned}
\end{equation}

Next, we deal with the term $\int_\O [\z^\a, \p_y] \p_y \u \cdot \ya \z^\a \u dxdy$.
For a smooth function $f$, it follows
\begin{equation}\label{3212}
[\z^\a, \p_y]f=\sum_{\b_2 \neq 0, \b_2+\ga_2=\a_2}
C_{\b_2, \ga_2}\varphi Z_2^{\b_2}(\frac{1}{\varphi})\p_y Z_2^{\ga_2}Z_\tau^{\a_1}f.
\end{equation}
By virtue of the definition of $\varphi$, it holds on by computating directly
\begin{equation}\label{3213}
|\varphi Z_2^{\b_2}(\frac{1}{\varphi})|\le C, \quad
|\p_y\{\varphi Z_2^{\b_2}(\frac{1}{\varphi})\}| \le C.
\end{equation}
For $\b_2 \neq 0$ and $\b_2+\ga_2=\a_2$, the integration by part yields immediately
\begin{equation}\label{3214}
\begin{aligned}
&\int_\O \varphi Z_2^{\b_2}(\frac{1}{\varphi})\p_y Z_2^{\ga_2}Z_\tau^{\a_1} \p_y \u \cdot \ya \z^\a \u dxdy\\
=
&
-\int_\O \varphi Z_2^{\b_2}(\frac{1}{\varphi}) Z_2^{\ga_2}Z_\tau^{\a_1} \p_y \u \cdot \ya \p_y  \z^\a \u dxdy\\
&
-2l\int_\O \varphi Z_2^{\b_2}(\frac{1}{\varphi}) Z_2^{\ga_2}Z_\tau^{\a_1} \p_y \u \cdot \yb \z^\a \u dxdy\\
&
-\int_\O \p_y\{\varphi Z_2^{\b_2}(\frac{1}{\varphi})\}  Z_2^{\ga_2}Z_\tau^{\a_1}  \p_y \u \cdot \ya \z^\a \u dxdy,
\end{aligned}
\end{equation}
where the boundary term in the above equality vanishes since the quantity $\z^{\a} \u|_{y=0}=0$.
Then, applying the relation \eqref{3212}, estimate \eqref{3213}, H\"{o}lder and Cauchy inequalities, we obtain
\begin{equation*}
\begin{aligned}
\int_\O \varphi Z_2^{\b_2}(\frac{1}{\varphi})\p_y Z_2^{\ga_2}Z_\tau^{\a_1} \p_y \u \cdot \ya \z^\a \u dxdy
\le
\frac{1}{4}\int \ya |\p_y  \z^\a \u|^2 dxdy+C_{m,l}(\|\u \|_{\H^{m}_l}^2+\|\p_y \u\|_{\H^{m-1}_l}^2),
\end{aligned}
\end{equation*}
and hence
\begin{equation}\label{3215}
\begin{aligned}
|\int_\O [\z^\a, \p_y] \p_y \u \cdot \ya \z^\a \u dxdy|
\le \frac{1}{4}\int \ya |\p_y  \z^\a \u|^2 dxdy+
C_{m, l}(\| \u\|_{\H^{m}_{l}}^2+\|\p_y\u\|_{\H^{m-1}_{l-1}}^2).
\end{aligned}
\end{equation}
Plugging the estimates  \eqref{328}, \eqref{3210},
\eqref{3211} and \eqref{3215} into \eqref{327}, we conclude
\begin{equation}\label{3216}
\mu\int_\O \z^\a \p_y^2 \u \cdot \ya \z^\a \u dxdy
\le -\frac{1}{2}\mu\|\p_y \z^\a \u\|_{L^2_{l}(\Omega)}^2
+C_{\mu, m, l}(\| \u\|_{\H^{m}_{l}}^2+\|\p_y \u\|_{\H^{m-1}_{l}}^2).
\end{equation}
Integrating by part, applying the boundary condition \eqref{bc5}
and divergence-free condition \eqref{eq5}$_4$, we find
\begin{equation}\label{3217}
\begin{aligned}
&\int_\O \{(\h+1)\p_x \z^\a \h+\g \p_y \z^\a \h\}\cdot \ya \z^\a \u dxdy\\
=
&-\int_\O \ya (\p_x \h+\p_y \g)\z^\a \u \cdot \z^\a \h dxdy\\
&-\int_\O ((\h+1)\p_x \z^\a \u+\g \p_y \z^\a \u)\cdot \ya \z^\a \h dxdy\\
&-2l\int_\O \yb \g \z^\a \u \cdot \z^\a \h dxdy\\
\le
&-\int_\O ((\h+1)\p_x \z^\a \u+\g \p_y \z^\a \u)\cdot \ya \z^\a \h dxdy\\
&+C_l \|\la y \ra^{-1}\g\|_{L^\infty_0(\O)}
\|\z^\a \u\|_{L^2_{l}(\Omega)}\|\z^\a \h\|_{L^2_{l}(\Omega)}.
\end{aligned}
\end{equation}
Using the H\"{o}lder and Cauchy inequalities, it follows
\begin{equation}\label{3218}
\begin{aligned}
&|\int_\O \z^\a(-\es \p_x r_u-\mu e^{-y})\cdot \ya \z^\a \u dxdy|\\
&\le \frac{1}{2}\es \|\p_x \z^\a \u\|_{L^2_l(\O)}^2
    +C(1+\|\z^\a r_u\|_{L^2_{l}(\Omega)}^2+\|\z^\a \u\|_{L^2_{l}(\Omega)}^2).
\end{aligned}
\end{equation}
Substituting the estimates \eqref{325}, \eqref{326}, \eqref{327},
\eqref{3216}, \eqref{3217} and \eqref{3218} into the equality \eqref{324},
and integrating over $[0, t]$, we conclude
\begin{equation}\label{3219}
\begin{aligned}
&\int_\O \ya \r |\z^\a \u|^2 dxdy
+\int_0^t \int_\O \ya (\es |\p_x \z^\a \u|^2+\mu|\p_y \z^\a \u|^2) dxdy\\
&+\int_0^t \int_\O ((\h+1)\p_x \z^\a \u+\g \p_y \z^\a \u)\cdot \ya \z^\a \h dxdyd\tau\\
\le
& \int_\O \ya \r_0 |\z^\a \u_0|^2 dxdy
  +C\int_0^t \|\z^{\a} r_u\|_{L^2_l(\Omega)}^2 d\t
  +\sum_{i=1}^6\int_0^t \|\C_{1i}^\a\|_{L^2_l(\Omega)}^2d\tau\\
&  +C_{\mu, m, l}(1+Q(t))\int_0^t (1+\|(\u, \h) \|_{\H^m_l}^2+\|\p_y \u \|_{\H^m_l}^2)d\tau.
\end{aligned}
\end{equation}

Now, we claim the following estimate, which will be shown later:
\begin{equation}\label{ulc1}
\sum_{i=1}^6 \int_0^t \|\C_{1i}^\a\|_{L^2_{l}(\Omega)}^2 d \tau
\le C_{m,l}(1+Q^2(t))\int_0^t (\|(\vr, \u, \h)\|_{\H^m_l}^2+\|(\p_y \u, \p_y \h)\|_{\H^{m-1}_l}^2) d\tau.
\end{equation}
For the moment we can substitute the estimate \eqref{ulc1} into inequality \eqref{3219}, and hence
\begin{equation}\label{3220}
\begin{aligned}
&\int_\O \ya \r |\z^\a \u|^2 dxdy
+\int_0^t \int_\O \ya (\es |\p_x \z^\a \u|^2+\mu|\p_y \z^\a \u|^2) dxdy\\
&+\int_0^t \int_\O ((\h+1)\p_x \z^\a \u+\g \p_y \z^\a \u)\cdot \ya \z^\a \h dxdyd\tau\\
\le
& \int_\O \ya \r_0 |\z^\a \u_0|^2 dxdy
   +C\int_0^t \|\z^{\a} r_u\|_{L^2_l(\Omega)}^2 d\t\\
& +C_{\mu, m, l}(1+Q^2(t))\int_0^t (1+\|(\vr, \u, \h)\|_{\H^m_l}^2+\|(\p_y \u, \p_y \h)\|_{\H^{m-1}_l}^2) d\tau.
\end{aligned}
\end{equation}

Step 2: Applying operator $\z^\a$ to the first equation of \eqref{eq5},
multiplying by $\ya \z^\a \vr$ and integrating over $\Omega$, we find
\begin{equation}\label{3221}
\begin{aligned}
&\frac{1}{2}\frac{d}{dt}\int_\O \ya |\z^\a \vr|^2 dxdy
+\es \int_\O \ya |\p_x \z^\a \vr|^2 dxdy\\
=
&\es\int_\O \z^\a \p_y^2 \vr \cdot \ya \z^\a \vr dxdy
+l \int_\O \yb \v |\z^\a \vr|^2 dxdy\\
&+\int_\O (\C_{21}^\a+\C_{22}^\a)\cdot \ya \z^\a \vr dxdy
-\es \int_\O \z^\a (\p_x r_1+\p_y r_2) \cdot \ya \z^\a \vr dxdy,
\end{aligned}
\end{equation}
where $\C_{21}^\a$ and $\C_{22}^\a$ are defined by
\begin{equation*}\label{3222}
\C_{21}^\a=-[\z^\a, (\u+1-e^{-y})\p_x]\vr, \quad
\C_{22}^\a=-[\z^\a, \v \p_y]\vr.
\end{equation*}
Similar to the equality \eqref{327}, the integration by part gives
\begin{equation}\label{3223}
\begin{aligned}
&\es\int_\O \z^\a \p_y^2 \vr \cdot \ya \z^\a \vr dxdy\\
=
&\es\int_{\mathbb{T}} \z^\a \p_y \vr \cdot \z^\a \vr|_{y=0} dx
-\es\int_\O \ya [\z^\a, \p_y] \vr \cdot \p_y  \z^\a \vr dxdy\\
&-\es\int_\O \ya |\p_y  \z^\a \vr|^2 dxdy
-2l \es\int_\O \yb \z^\a \p_y \vr \cdot  \z^\a \vr dxdy\\
&
+\es\int_\O [\z^\a, \p_y] \p_y \vr \cdot \ya \z^\a \vr dxdy.
\end{aligned}
\end{equation}
If $\a_2=0$, the boundary condition $\p_y \vr|_{y=0}=0$
implies $\z^\a \p_y \vr|_{y=0}=0$. If $\a_2\neq 0$, we get from the
definition of $Z_2=\varphi(y)\p_y=\frac{y}{y+1}\p_y$ that $\z^\a \p_y \vr|_{y=0}=0$,
and hence
\begin{equation*}\label{3224}
\int_{\mathbb{T}} \z^\a \p_y \vr \cdot \z^\a \vr|_{y=0} dx=0.
\end{equation*}
The other terms on the right hand side of \eqref{3223}
can take the idea as estimate \eqref{3215}, we conclude
\begin{equation}\label{3225}
\es\int_\O \z^\a \p_y^2 \vr \cdot \ya \z^\a \vr dxdy
\le
-\frac{3}{4}\es\int_\O \ya |\p_y  \z^\a \vr|^2 dxdy+C_{m, l}(\|\vr\|_{\H^m_l}^2+\|\p_y \vr\|_{\H^{m-1}_l}^2).
\end{equation}
By routine checking, it follows
\begin{equation}\label{3226}
\begin{aligned}
| \int_{\O} \yb \v |\z^\a \vr|^2 dxdy|
\le C \|\la y \ra^{-1} \v\|_{L^\infty_0(\O)}\|\z^\a \vr\|_{L^2_l(\Omega)}^2.
\end{aligned}
\end{equation}
Applying the integration by part and Cauchy inequality, we obtain
\begin{equation}\label{3226a}
\begin{aligned}
&|\es\int_{\O} \z^\a(\p_x r_1+\p_y r_2)\cdot \ya \z^\a \vr dxdy|\\
\le
&\frac{\es}{2}\|(\p_x \z^\a \vr, \p_y \z^\a \vr)\|_{L^2_l(\Omega)}^2
+C(\|\z^\a r_1\|_{L^2_l(\Omega)}^2+\|r_2\|_{\H^m_l}^2)+C_l\|\z^\a \vr\|_{L^2_l(\Omega)}^2.
\end{aligned}
\end{equation}
Substituting estimates \eqref{3225} and \eqref{3226} into  \eqref{3221},
integrating the resulting inequality over $[0, t]$, we get
\begin{equation}\label{3227}
\begin{aligned}
&\int_\O \ya |\z^\a \vr|^2 dxdy
+\es \int_0^t \int_\O \ya (|\p_x \z^\a \vr|^2+|\p_y \z^\a \vr|^2 )dxdy d\tau\\
\le
&\int_\O \ya |\z^\a \vr_0|^2 dxdy
+C\int_0^t \|(r_1, r_2)\|_{\H^m_l}^2 d\t
+\int_0^t \|(\C_{21}^\a, \C_{22}^\a)\|_{L^2_l(\Omega)}^2d\tau\\
&+C_{m, l}(1+\|\la y \ra^{-1} \v\|_{L^\infty_0(\O)})
  \int_0^t(\|\vr\|_{\H^m_l}^2+\|\p_y \vr\|_{\H^{m-1}_l}^2)d\tau.
\end{aligned}
\end{equation}
Similar to the claim estimates \eqref{ulc1}, we can justify the estimate
\begin{equation*}\label{3228}
|\int_0^t \|(\C_{21}^\a, \C_{22}^\a)\|_{L^2_l(\Omega)}^2d\tau|
\le C_{m, l}(1+Q(t))\int_0^t (\|(\vr, \u)\|_{\H^m_l}^2+\|\p_y \vr\|_{\H^{m-1}_l}^2)d\tau.
\end{equation*}
This and inequality \eqref{3227} yield directly
\begin{equation}\label{3229}
\begin{aligned}
&\int_\O \ya |\z^\a \vr|^2 dxdy
+\es \int_0^t \int_\O \ya (|\p_x \z^\a \vr|^2+|\p_y \z^\a \vr|^2 )dxdy d\tau\\
\le
&\!\int_\O \!\ya |\z^\a \vr_0|^2 dxdy
+C\!\int_0^t\!\! \| (r_1, r_2)\|_{\H^m_l}^2 d\t
+\!C_{m, l}(1+Q(t))\!\int_0^t\!\! (\|(\vr, \u)\|_{\H^m_l}^2+\|\p_y \vr\|_{\H^{m-1}_l}^2)d\tau.
\end{aligned}
\end{equation}

Step 3: Applying operator $\z^\a$ to the third equation of \eqref{eq5}
multiplying by $\ya \z^\a \h$ and integrating over $\Omega$, it follows
\begin{equation*}
\begin{aligned}
&\frac{1}{2}\frac{d}{dt}\int_\O \ya |\z^\a \h|^2 dxdy
+\es \int \ya |\p_x \z^\a \h|^2 dxdy
-\k \int_\O \z^\a \p_y^2 \h \cdot \ya \z^\a \h dxdy\\
=
&l \int_\O \yb \v |\z^\a \h|^2 dxdy
+\int_\O \{(\h+1)\p_x \z^\a \u+\g \p_y \z^\a \u\} \cdot \ya \z^\a \h dxdy\\
&+\int_\O \z^\a(\g e^{-y}-\es \p_x r_h)\cdot  \ya \z^\a \h dxdy
 +\int_\O (\C_{31}^\a+\C_{32}^\a+\C_{33}^\a+\C_{34}^\a)\cdot  \ya \z^\a \h dxdy,
\end{aligned}
\end{equation*}
where $\C_{3i}^\a(i=1,2,3,4)$ are defined by
$$
\begin{aligned}
&\C_{31}^\a=-[\z^\a, (\u+1-e^{-y})\p_x]\h, \quad
\C_{32}^\a=-[\z^\a, \v \p_y]\h,\\
&\C_{33}^\a=[\z^\a, (\h+1)\p_x]\u,\quad
\C_{34}^\a=[\z^\a, \g \p_y](\u-e^{-y}).
\end{aligned}
$$
Following the idea as the estimate \eqref{3229}, we can verify the following estimate
\begin{equation*}\label{3230}
\begin{aligned}
&\int_\O \ya |\z^\a \h|^2 dxdy
+\int_0^t \int_\O \ya (\es|\p_x \z^\a \h|^2+\k |\p_y \z^\a \h|^2) dxdyd\tau\\
&-\int_0^t \int_\O \{(\h+1)\p_x \z^\a \u+\g \p_y \z^\a \u\}
\cdot \ya \z^\a \h dxdyd\tau\\
\le
&\!\int_\O \!\ya |\z^\a \h_0|^2 dxdy+C\!\!\int_0^t\! \!\|\z^\a r_h\|_{L^2_l(\Omega)}^2 d\t
+C_{\k, m, l}(1+Q(t))\!\!\int_0^t \!(\|(\u, \h)\|_{\H^{m}_l}^2+\|(\p_y \u, \p_y \h)\|_{\H^{m}_l}^2)d\tau,
\end{aligned}
\end{equation*}
which, together with the estimates \eqref{3220} and \eqref{3229} completes the proof of lemma
after taking the summation over all $|\a|\le m$ and $|\a_1|\le m-1$ .
\end{proof}

\underline{\textbf{Proof of claim estimate \eqref{ulc1}}}
Now we give the estimate for $\int_0^t \|\C_{1i}^\a\|_{L^2_l(\Omega)}^2d\tau(i=1,...,6)$.
By virtue of the Moser type inequality \eqref{ineq-moser}, we find
\begin{equation*}
\begin{aligned}
\int_0^t \|\C_{11}^\a\|_{L^2_{l}(\Omega)}^2 d \tau
&\le C_m \sum_{|\b|\ge 1, \b +\ga =\a} \int_0^t \|\z^\b \vr \cdot \z^{\ga} \p_t \u \|_{L^2_l(\Omega)}^2 d\tau\\
&\le  C_m\|\z^{E_i} \vr \|_{L^\infty_0(\O)}^2 \int_0^t \|\p_t \u \|_{\H^{m-1}_{l}}^2 d \tau
     + C_m\|\p_t \u\|_{L^\infty_0(\O)}^2 \int_0^t \|\z^{E_i} \vr\|_{\H^{m-1}_{l}}^2 d \tau\\
&\le  C_m\|(\z^{E_i}\vr, \p_t \u)\|_{L^\infty_0(\O)}^2 \int_0^t \|(\vr, \u)\|_{\H^{m}_{l}}^2 d \tau.
\end{aligned}
\end{equation*}
Similarly, we conclude the following estimate
\begin{equation*}
\int_0^t \|\C_{12}^\a\|_{L^2_{l}(\Omega)}^2 d \tau
\le C_m(1+\|(\u, \z^{E_i} \vr, \z^{E_i} \u)\|_{L^\infty_0(\O)}^4)\int_0^t(1+\|(\vr, \u)\|_{\H^m_l}^2)d\tau,
\end{equation*}
and
\begin{equation*}
\int_0^t \|\C_{14}^\a\|_{L^2_{l}(\Omega)}^2 d \tau
\le  C_m \|\z^{E_i} \h\|_{L^\infty_0(\O)}^2 \int_0^t \|\h\|_{\H^m_l}^2 d\tau.
\end{equation*}

Finally, we deal with the term $\int_0^t \|\C_{13}^\a\|_{L^2_{l}(\Omega)}^2 d \tau$.
By direct computation, it is easy to check that
\begin{equation*}
\C_{13}^\a=\underset{|\b|\ge1,\ \b+\ga=\a}{\sum}C_{\b,\ga}
\z^\b(\r \v)\z^\ga \p_y \u+\r \v [\z^\a, \p_y]\u.
\end{equation*}
Using the estimate \eqref{low-bou}, we get
\begin{equation}\label{u3}
\int_0^t \|\r \v [\z^\a, \p_y]\u\|_{L^2_l(\Omega)}^2 d\tau
\le C_m\| \v\|_{L^\infty_0(\O)}^2 \int_0^t \|\p_y \u\|_{\H^{m-1}_l}^2 d\tau.
\end{equation}
In order to control the velocity $\v$, the idea is to apply the Hardy inequality
and divergence-free condition to transform into the velocity $\p_x \u$ in some weighted Sobolev norm.
By virtue of $|\a_1|\le m-1$ and $|\a_1|+\a_2=m$, it follows $\a_2 \ge 1$, and hence
we can get $\ga_2 \ge 1$ if $\b_2=0$. Thus it follows from the Moser type inequality \eqref{ineq-moser} that
\begin{equation}\label{u1}
\begin{aligned}
&\underset{|\b|\ge1,\ \b+\ga=\a}{\sum}\int_0^t
\|\z^\b(\r \v)\z^\ga \p_y \u\|_{L^2_l(\Omega)}^2 d\tau\\
\le
& C_m\|Z_\tau^{e_i} (\r \v)\|_{L^\infty_1(\O)}^2 \int_0^t \|Z_2 \p_y \u\|_{\H^{m-2}_{l-1}}^2 d\tau
+ C_m\|Z_2 \p_y \u\|_{L^\infty_1(\O)}^2 \int_0^t \| Z_\tau^{e_i} (\r \v)\|_{\H^{m-2}_{l-1}}^2 d\tau.
\end{aligned}
\end{equation}
Using the divergence-free condition, Hardy and Moser type inequalities \eqref{ineq-moser}, we get
\begin{equation*}
\begin{aligned}
\int_0^t \| Z_\tau^{e_i} (\r \v)\|_{\H^{m-2}_{l-1}}^2 d\tau
\le
&C_m\|\r\|_{L^\infty_0(\O)}^2 \int_0^t \|  \v\|_{\H^{m-1}_{l-1}}^2 d\tau
+C_m\|\v\|_{L^\infty_0(\O)}^2 \int_0^t \| \vr \|_{\H^{m-1}_{l-1}}^2 d\tau\\
\le
&C_{m,l}\int_0^t \| \p_y \v\|_{\H^{m-1}_{l}}^2 d\tau
+C_m\|\v\|_{L^\infty_0(\O)}^2 \int_0^t \| \vr \|_{\H^{m-1}_{l-1}}^2 d\tau\\
\le
&C_{m,l}(1+\|\v\|_{L^\infty_0(\O)}^2)\int_0^t  \|(\vr, \u)\|_{\H^{m}_{l}}^2 d\tau.
\end{aligned}
\end{equation*}
This and the inequality \eqref{u1} yield directly
\begin{equation}\label{u4}
\begin{aligned}
\underset{|\b|\ge1,\ \b+\ga=\a}{\sum}\int_0^t
\|\z^\b(\r \v)\z^\ga \p_y \u\|_{L^2_l(\Omega)}^2 d\t
\le C_{m,l}(1+Q^2(t))\int_0^t (\|(\vr, \u)\|_{\H^{m}_{l}}^2+\|\p_y \u\|_{\H^{m-1}_{l}}^2) d\t.
\end{aligned}
\end{equation}
If $\b_2 \ge 1$, we get after using the Moser type inequality \eqref{ineq-moser} that
\begin{equation}\label{u2}
\begin{aligned}
&\underset{|\b|\ge1,\ \b+\ga=\a}{\sum}\int_0^t
\|\z^\b(\r \v)\z^\ga \p_y \u\|_{L^2_l(\Omega)}^2 d\t\\
\le
&C_m \|Z_2 (\r \v)\|_{L^\infty_1(\O)}^2 \int_0^t \| \p_y \u\|_{\H^{m-1}_{l-1}}^2 d\t
+C_m \|\p_y \u\|_{L^\infty_1(\O)}^2 \int_0^t \| Z_2 (\r \v)\|_{\H^{m-1}_{l-1}}^2 d\t.
\end{aligned}
\end{equation}
In view of the fact $Z_2 (\r \v)=Z_2 \r \v+\r Z_2 \v$, we apply divergence-free condition,
Hardy and Moser type inequalities to get
\begin{equation*}
\begin{aligned}
\int_0^t \| Z_2 (\r \v)\|_{\H^{m-1}_{l-1}}^2 d\tau
\le
&C_m\|(Z_2 \r , \v)\|_{L^\infty_0(\O)}^2
 \int_0^t (\| \p_y \v\|_{\H^{m-1}_{l}}^2+\| \vr \|_{\H^{m}_{l-1}}^2) d\tau\\
&+C_m(1+\|Z_2 \v\|_{L^\infty_0(\O)}^2)
\int_0^t (\| \u \|_{\H^{m}_{l-1}}^2+ \| \vr \|_{\H^{m-1}_{l-1}}^2)d\tau\\
\le
&C_{m,l}(1+\|(\v, \p_x \u, Z_2 \r)\|_{L^\infty_0(\O)}^2)\int_0^t \|(\vr, \u)\|_{\H^{m}_{l}}^2 d\tau.
\end{aligned}
\end{equation*}
which, along with \eqref{u2}, gives directly
\begin{equation}\label{u5}
\underset{|\b|\ge1,\ \b+\ga=\a}{\sum}\int_0^t
\|\z^\b(\r \v)\z^\ga \p_y \u\|_{L^2_l(\Omega)}^2 d\tau
\le C_{m,l}(1+Q^2(t))\int_0^t (\|(\vr, \u)\|_{\H^{m}_{l}}^2+\|\p_y \u\|_{\H^{m-1}_{l}}^2) d\tau.
\end{equation}
The combination of the estimates \eqref{u3}, \eqref{u4} and \eqref{u5} yields directly
\begin{equation*}
\int_0^t \|\C_{13}^\a\|_{L^2_{l}(\Omega)}^2 d \tau
\le C_{m,l}(1+Q^2(t))\int_0^t (\|(\vr, \u)\|_{\H^m_l}^2+\|\p_y \u\|_{\H^{m-1}_l}^2)d\tau.
\end{equation*}
Similarly, by routine checking, we may conclude that
\begin{equation*}
\int_0^t (\|\C_{14}^\a\|_{L^2_{l}(\Omega)}^2+\|\C_{15}^\a\|_{L^2_{l}(\Omega)}^2+ \|\C_{16}^\a\|_{L^2_{l}(\Omega)}^2) d \tau
\le C_{m,l}(1+Q(t))\int_0^t (\|(\vr, \u, \h)\|_{\H^m_l}^2+\|\p_y \h\|_{\H^{m-1}_l}^2)d\tau.
\end{equation*}
Therefore, we complete the proof of the claim estimate \eqref{ulc1}.

\subsection{Weighted $\H^m_l-$Estimates only on Tangential Derivative}

In this subsection, we hope to establish the estimate for the quantity
$Z_\t^{\a_1}(\vr, \u, \h)$ with $|\a_1|=m$.
However, there is an essential difficulty to achieve this goal since
the terms $\v \p_y \vr$, $\r \v \p_y(\u-e^{-y})-\g \p_y \h$ and
$\v \p_y \h-\g \p_y (\u-e^{-y})$ will create the loss of one derivative in the tangential variable $x$.
In other words, $\v=-\p_y^{-1}\p_x u$ and $\g=-\p_y^{-1}\p_x \h$, by the divergence-free condition \eqref{eq5}$_4$,
create a loss of $x-$derivative that prevents us to apply the standard energy estimates.
To overcome this essential difficulty, we take the strategy of the recent interesting result
\cite{Liu-Xie-Yang} that only needs that the background tangential magnetic field
has a lower positive bound instead of monotonicity assumption on the tangential velocity.
However, due to the density being a unknown function instead of a constant, we need to take some
new ideas to deal with the terms $\v \p_y \vr$ and $\r \v \p_y(\u-e^{-y})$.

First of all, applying $Z_\t^{\a_1}(|\a_1|= m)$ differential operator to the equation \eqref{eq5}$_3$,
we find
\begin{equation}\label{3b1}
\begin{aligned}
&\{\p_t+(\u+1-e^{-y})\p_x +\v \p_y-\es \p_x^2 -\k \p_y^2\}(Z_\t^{\a_1} \h)
+Z_\t^{\a_1} \v \p_y \h\\
&=(\h+1)\p_x Z_\t^{\a_1} \u+\g \p_y Z_\t^{\a_1} \u
+Z_\t^{\a_1} \g \p_y(\u-e^{-y})-\es Z_\t^{\a_1} \p_x r_h+f_h,
\end{aligned}
\end{equation}
where the function $f_h$ is defined by
\begin{equation*}\label{3b2}
\begin{aligned}
f_h=
&-[Z_\t^{\a_1}, (\u+1-e^{-y})\p_x]\h
  +[Z_\t^{\a_1}, (\h+1)\p_x]\u\\
&-\sum_{\substack{\b_1+\ga_1=\a_1 \\ \b_1 \neq 0, \b_1 \neq \a_1}}
C_{\b_1, \ga_1}Z_\t^{\b_1} \v Z_\t^{\ga_1} \p_y \h
+\sum_{\substack{\b_1+\ga_1=\a_1 \\ \b_1 \neq 0, \b_1 \neq \a_1}}
C_{\b_1, \ga_1} Z_\t^{\b_1} \g Z_\t^{\ga_1} \p_y \u.
\end{aligned}
\end{equation*}

To eliminate to difficult term $Z_\t^{\a_1} \v \p_y \h$, following the idea as in \cite{Liu-Xie-Yang},
we introduce the stream function $\ps$  satisfying
\begin{equation*}\label{3b3}
\p_y \ps=\h, \quad  \p_x \ps = -\g, \quad \ps|_{y=0}=0.
\end{equation*}
Then, we can deduce from the equation \eqref{eq5}$_3$
and boundary condition \eqref{bc5} that
\begin{equation*}\label{3b4}
\p_t \ps+(\u+1-e^{-y})\p_x \ps+\v (\p_y \ps+1)
-\es \p_x^2 \ps-\k \p_y^2 \ps=-\es \p_y^{-1}\p_x r_h.
\end{equation*}
Applying $Z_\t^{\a_1}(|\a_1|= m)$ differential operator to above equation, it follows
\begin{equation}\label{3b5}
\{\p_t+(\u+1-e^{-y})\p_x+\v \p_y -\es \p_x^2 -\k \p_y^2 \}Z_\t^{\a_1}\ps
+Z_\t^{\a_1} \v(\h+1)=-\es \p_y^{-1} Z_\t^{\a_1}\p_x r_h+f_{\psi},
\end{equation}
where the function $f_{\psi}$ is defined by
\begin{equation*}\label{3b6}
f_{\psi}=
-[Z_\t^{\a_1}, (\u+1-e^{-y})\p_x] \ps
-\sum_{\substack{\b_1+\ga_1=\a_1 \\ \b_1 \neq 0, \b_1 \neq \a_1}}
C_{\b_1, \ga_1} Z_\t^{\b_1} \v Z_\t^{\ga_1} \p_y \ps.
\end{equation*}

Set $\eta_h :=\frac{\p_y \h}{\h+1}$ and define the quantity
\begin{equation}\label{equi-h}
\h_m :=Z_\t^{\a_1} \h-\eta_h Z_\t^{\a_1} \ps,
\end{equation}
then multiplying the equation \eqref{3b5}  by $\eta_h$ and substituting the equation \eqref{3b1},
the difficult term $Z_\t^{\a_1} \v \p_y \h$ in \eqref{3b1} can be eliminated.
Hence, we get the evolution equation for the quantity $\h_m$ as
\begin{equation}\label{eqhm}
\begin{aligned}
&\p_t \h_m+(\u+1-e^{-y})\p_x \h_m +\v \p_y \h_m-\es \p_x^2 \h_m-\k \p_y^2 \h_m\\
&-(\h+1)\p_x Z_\t^{\a_1}\u-\g \p_y Z_\t^{\a_1} \u-Z_\t^{\a_1} \g \p_y(\u-e^{-y})\\
=
& -\es Z_\t^{\a_1} \p_x r_h+\es \eta_h \p_y^{-1} Z_\t^{\a_1} \p_x r_h
  +f_h-\eta_h f_\psi+2 \es \p_x \eta_h \p_x Z_\t^{\a_1} \ps
  +2\k\p_y \eta_h \p_y Z_\t^{\a_1} \ps\\
&+Z_\t^{\a_1} \ps (\p_t+(\u+1-e^{-y})\p_x+\v \p_y-\es \p_x^2 -\k \p_y^2)\eta_h.
\end{aligned}
\end{equation}

Similarly, after applying $Z_\t^{\a_1}(|\a_1|= m)$ operator to the first equation of \eqref{eq5}, we get
\begin{equation}\label{3b7}
\{\p_t+(\u+1-e^{-y})\p_x +\v \p_y-\es \p_x^2 -\es \p_y^2\}(Z_\t^{\a_1} \vr)
+Z_\t^{\a_1} \v \p_y \vr
=-\es Z_\t^{\a_1}\p_x r_1-\es Z_\t^{\a_1}\p_y r_2+f_\rho,
\end{equation}
where the function $f_\rho$ is defined by
\begin{equation*}\label{3b8}
\begin{aligned}
f_\rho=
-[Z_\t^{\a_1}, (\u+1-e^{-y})\p_x]\vr
-\sum_{\substack{\b_1+\ga_1=\a_1 \\ \b_1 \neq 0, \b_1 \neq \a_1}}
C_{\b_1, \ga_1}Z_\t^{\b_1} \v Z_\t^{\ga_1} \p_y \vr.
\end{aligned}
\end{equation*}

Set $\eta_\rho=\frac{\p_y \vr}{\h+1}$ and define
\begin{equation}\label{equi-r}
\vr_m:=Z_\t^{\a_1} \vr-\eta_\rho Z_\t^{\a_1} \ps,
\end{equation}
we multiply the equation \eqref{3b5}
by $\eta_\rho$ and substitute to the equation \eqref{3b7},
and hence, the evolution for the quantity $\vr_m$ as follows
\begin{equation}\label{eqrm}
\begin{aligned}
&\p_t \vr_m+(\u+1-e^{-y})\p_x \vr_m +\v \p_y \vr_m-\es \p_x^2 \vr_m-\es \p_y^2 \vr_m\\
=
&f_\rho-\eta_\rho f_\psi+2 \es \p_x \eta_\rho \p_x Z_\t^{\a_1} \ps
  -\k \eta_\rho \p_y^2 Z_\t^{\a_1} \ps+\es \p_y^2 (\eta_\rho Z_\t^{\a_1} \ps)\\
&+Z_\t^{\a_1} \ps (\p_t+(\u+1-e^{-y})\p_x+\v \p_y-\es \p_x^2 )\eta_\rho
-\es Z_\t^{\a_1} (\p_x r_1+\p_y r_2)+\es \eta_\rho \p_y^{-1} Z_\t^{\a_1} \p_x r_h.
\end{aligned}
\end{equation}

Finally, applying $Z_\t^{\a_1}(|\a_1|= m)$ differential operator to the equation \eqref{eq5}$_2$, we get
\begin{equation}\label{3b9}
\begin{aligned}
&\{\r \p_t+\r (\u+1-e^{-y})\p_x +\r \v \p_y-\es \p_x^2 -\mu \p_y^2\}(Z_\t^{\a_1} \u)
+\r Z_\t^{\a_1} \v \p_y (\u-e^{-y})\\
&=(\h+1)\p_x Z_\t^{\a_1} \h+\g \p_y Z_\t^{\a_1} \h
+Z_\t^{\a_1} \g \p_y \h-\es Z_\t^{\a_1}\p_x r_u+f_u,
\end{aligned}
\end{equation}
where the function $f_u$ is defined by
\begin{equation*}\label{3b10}
\begin{aligned}
f_u=
&-[Z_\t^{\a_1}, \r \p_t]\u-[Z_\t^{\a_1}, \r(\u+1-e^{-y})\p_x]\u+[Z_\t^{\a_1}, (\h+1)\p_x]\h\\
&-\sum_{\substack{\b_1+\ga_1=\a_1 \\ \b_1 \neq 0 }}
C_{\b_1, \ga_1}Z_\t^{\b_1} \r Z_\t^{\ga_1} \v \p_y (\u-e^{-y})
-\sum_{\substack{\b_1+\ga_1=\a_1 \\ \b_1 \neq 0, \b_1 \neq \a_1}}
C_{\b_1, \ga_1}Z_\t^{\b_1} (\r \v) Z_\t^{\ga_1} \p_y \u\\
&+\sum_{\substack{\b_1+\ga_1=\a_1 \\ \b_1 \neq 0, \b_1 \neq \a_1}}
C_{\b_1, \ga_1} Z_\t^{\b_1} \g Z_\t^{\ga_1} \p_y \h.
\end{aligned}
\end{equation*}
Set $\eta_u=\frac{\p_y(\u-e^{-y})}{\h+1}$, multiplying the equation \eqref{3b5}
by $\r \eta_u$ and substituting to the equation \eqref{3b9}, we find for the quantity
\begin{equation}\label{eqi-u}
\u_m:=Z_\t^{\a_1} \u-\eta_u Z_\t^{\a_1} \ps
\end{equation}
satisfying the evolution as follows
\begin{equation}\label{equm}
\begin{aligned}
&\r \p_t \u_m+\r(\u+1-e^{-y})\p_x \u_m +\r \v \p_y \u_m-\es \p_x^2 \u_m-\mu \p_y^2 \u_m\\
&-(\h+1)\p_x Z_\t^{\a_1}\h-\g \p_y Z_\t^{\a_1} \h-Z_\t^{\a_1} \g \p_y \h\\
=\
&f_u-\r \eta_u f_\psi-Z_\t^{\a_1} \ps (\r \p_t+\r(\u+1-e^{-y})\p_x+\r\v \p_y)\eta_u\\
&-\es(\r-1)\eta_u \p_x^2 Z_\t^{\a_1}\ps-\k \r \eta_u \p_y^2 Z_\t^{\a_1} \ps
+2\es \p_x \eta_u \p_x Z_\t^{\a_1} \ps\\
&+\es \p_x^2 \eta_u Z_\t^{\a_1} \ps+\mu \p_y^2(  \eta_u Z_\t^{\a_1} \ps )
-\es Z_\t^{\a_1} \p_x r_u+\es \rho^\es \eta_u \p_y^{-1} Z_\t^{\a_1} \p_x r_h.
\end{aligned}
\end{equation}

Let us define the functional:
\begin{equation}\label{ydef}
X_{m,l}(t):=1+\e_{m,l}(t)+\|(\vr_m,\u_m, \h_m)(t)\|_{L^2_l(\O)}^2+\|\p_y(\vr, \u, \h)(t)\|_{\H^{m-1}_l}^2
+\| \p_y \vr (t)\|_{\H^{1,\infty}_1}^2,
\end{equation}
where $\e_{m,l}(\t)$ is defined by \eqref{eml}.
Then, we will establish the following estimate in this subsection.
\begin{proposition}\label{Tanential-estimate}
Let $(\vr, \u, \v, \h, \g)$ be sufficiently smooth solution, defined on $[0, T^\es]$,
to the equations \eqref{eq5}-\eqref{bc5}.
Under the assumptions of conditions \eqref{a2} and \eqref{a1}, it holds on
\begin{equation*}
\begin{aligned}
&\sup_{\t \in [0, t]}\|(\vr_m,\u_m, \h_m)(\t)\|_{L^2_l(\O)}^2
+\es\int_0^t \|(\p_x \vr_m, \p_x \u_m, \p_x \h_m)(\t)\|_{L^2_l(\O)}^2 d\t\\
&+\int_0^t(\es\|\p_y \vr_m(\t)\|_{L^2_l(\O)}^2+\mu\|\p_y \u_m(\t)\|_{L^2_l(\O)}^2+\k\|\p_y \h_m(\t)\|_{L^2_l(\O)}^2) d\t\\
\le
&C\|(\vr_m, \u_m , \h_m)(0)\|_{L^2_l(\O)}^2
+C t \|(\rho_0, u_{10}, h_{10})\|_{\widehat{\mathcal{B}}^m_l}^2
+C_{\mu, \k, m, l} \d^{-6}(1+Q^3(t))\int_0^t X_{m,l}(\t)d\t.
\end{aligned}
\end{equation*}
\end{proposition}

First of all, we establish the weighted $L^2-$estimate for the quantity $\vr_m$.

\begin{lemma}\label{lemma33}
Let $(\vr, \u, \v, \h, \g)$ be sufficiently smooth solution, defined on $[0, T^\es]$,
to the equations \eqref{eq5}.
Under the assumption of condition \eqref{a2}, then we have the following estimate for $0<\d_1<1$:
\begin{equation*}\label{331}
\begin{aligned}
&\sup_{\t \in [0, t]}\|\vr_m(\t)\|_{L^2_l(\O)}^2
+\es \int_0^t (\|\p_x \vr_m\|_{L^2_l(\O)}^2+\|\p_y \vr_m\|_{L^2_l(\O)}^2)d\tau\\
\le
&C\|\vr_{m}(0)\|_{L^2_l(\O)}^2
+C\!\int_0^t \|(r_1, r_2, r_h)\|_{\H^{m}_{l, tan}}^4 d\t
+C\d_1 \int_0^t ( \es \|\p_x  \h_m\|_{L^2_l(\O)}^2+\!\k \|\p_y  \h_m\|_{L^2_l(\O)}^2)d\t\\
&+C_{\k, m, l}\d^{-4}(1+Q^2(t))\int_0^t (\mathcal{E}_{m,l}(\t)
+\|(\vr_m, \u_m, \h_m)(\t)\|_{L^2_l(\O)}^2+\|\p_y \vr(\t)\|_{\H^{m-1}_l}^2)d\t.
\end{aligned}
\end{equation*}
\end{lemma}

\begin{proof}
Due to the fact $\p_y \vr|_{y=0}$ and $\ps|_{y=0}=0$, it follows $\p_y \eta_\rho|_{y=0}=0$.
Then, multiplying the equation \eqref{eqrm} by $\ya \vr_m$, integrating over $[0, t]\times \O$
and integrating by part, we get
\begin{equation}\label{333}
\begin{aligned}
&\frac{d}{dt}\int_\O \ya |\vr_m|^2 dxdy+\es \int_\O \ya |\p_x \vr_m|^2 dxdyd\t
+\es \int_\O \ya |\p_y \vr_m|^2 dxdyd\t \\
=
&l\int_\O \yb \v |\vr_m|^2 dxdy-2l\es \int_\O \ya \p_y \vr_m \cdot \vr_m dxdy
+2 \es \int_\O \p_x \eta_\rho \p_x Z_\t^{\a_1} \ps \cdot \ya \vr_m dxdy\\
&+\int_\O[-\k \eta_\rho \p_y^2 Z_\t^{\a_1} \ps+\es \p_y^2 (\eta_\rho Z_\t^{\a_1} \ps)
-\es    \p_x^2  \eta_\rho Z_\t^{\a_1} \ps]\cdot \ya \vr_m dxdy \\
&+\int_\O Z_\t^{\a_1} \ps (\p_t+(\u+1-e^{-y})\p_x+\v \p_y )\eta_\rho
  \cdot \ya \vr_m dxdy\\
&+\int_\O(f_\rho-\eta_\rho f_\psi-\es Z_\t^{\a_1} (\p_x r_1+\p_y r_2)+\es \eta_\rho \p_y^{-1} Z_\t^{\a_1} \p_x r_h)\cdot \ya \vr_m dxdy.
\end{aligned}
\end{equation}
Using the H\"{o}lder and Cauchy inequalities, it follows
\begin{equation*}\label{334}
\begin{aligned}
&|l\int_\O \yb \v |\vr_m|^2 dxdy-2l\es \int_\O \ya \p_y \vr_m \cdot \vr_m dxdy|\\
\le
&C_l \|\v\|_{L^\infty_{-1}(\O)}\|\vr_m\|_{L^2_l(\O)}^2+2l \es \|\p_y \vr_m\|_{L^2_l(\O)}\|\vr_m\|_{L^2_l(\O)}\\
\le
&\frac{1}{8} \es \|\p_y \vr_m\|_{L^2_l(\O)}^2+C_l(1+\|\v\|_{L^\infty_{-1}(\O)}^2)\|\vr_m\|_{L^2_l(\O)}^2,
\end{aligned}
\end{equation*}
and
\begin{equation}\label{335}
|\es \int_\O \p_x \eta_\rho \p_x Z_\t^{\a_1} \ps\cdot \ya \vr_m dxdy|
\le \es \|\p_x \eta_\rho \p_x Z_\t^{\a_1} \ps\|_{L^2_l(\O)}\|\vr_m\|_{L^2_l(\O)}.
\end{equation}
By virtue of the fact $\p_x \eta_\rho=\frac{1}{\h+1}\{{\p_{xy}^2 \vr} -\frac{\p_y \vr \p_x \h}{\h+1 }\}$, we get
\begin{equation*}\label{336}
\|\p_x \eta_\rho \p_x Z_\t^{\a_1} \ps\|_{L^2_l(\O)}
\le (\|\p_{xy}^2 \vr\|_{L^\infty_1(\O)}+\d^{-1}\| \p_{y} \vr\|_{L^\infty_1}\| \p_{x} \h\|_{L^\infty_0(\O)})
     \|\frac{\p_x Z_\t^{\a_1} \ps}{\h+1}\|_{L^2_{l-1}(\O)},
\end{equation*}
where we have used the fact $\h+1\ge \d$.
This and inequalities \eqref{b13} and \eqref{335} give
\begin{equation*}\label{337}
|\es \int_\O \p_x \eta_\rho \p_x Z_\t^{\a_1} \ps\cdot \ya \vr_m dxdy|
\le \d_1 \es \|\p_x \h_m\|_{L^2_l(\O)}^2+C_l \d^{-4}(1+Q^2(t))\|(\vr_m, \h_m)\|_{L^2_l(\O)}^2.
\end{equation*}
Using the H\"{o}lder and Cauchy inequalities, it follows
\begin{equation*}
\begin{aligned}
|\k \int_\O \eta_\rho \p_y^2 Z_\t^{\a_1} \ps \cdot \ya \vr_m dxdy|
\le \d_1 \k \|\p_y  Z_\t^{\a_1} \h\|_{L^2_l(\O)}^2+C_\k \|\eta_\rho\|_{L^\infty_0(\O)}^2 \|\vr_m\|_{L^2_l(\O)}^2,
\end{aligned}
\end{equation*}
which, together with the estimate \eqref{b14}, yields directly
\begin{equation*}\label{338}
|\k \int_\O \eta_\rho \p_y^2 Z_\t^{\a_1} \ps \cdot \ya \vr_m dxdy|
\le \d_1 \k \|\p_y  \h_m\|_{L^2_l(\O)}^2+C_{\k, l} \d^{-2} Q(t)\|(\vr_m, \h_m)\|_{L^2_l(\O)}^2.
\end{equation*}
In view of $\p_y \vr|_{y=0}=0$ and $\ps|_{y=0}=0$, it is easy
to justify the fact $\p_y (\eta_\rho Z_\t^{\a_1} \ps)|_{y=0}=0$,
and hence, we integrating by part to get directly
\begin{equation*}\label{339}
\begin{aligned}
&\es \int_\O \p_y^2 (\eta_\rho Z_\t^{\a_1} \ps)\cdot \ya \vr_m dxdy\\
=
&-\es \int_\O \p_y \eta_\rho Z_\t^{\a_1} \ps \cdot (2l \yb \vr_m+\ya \p_y \vr_m) dxdy\\
&-\es \int_\O   \eta_\rho Z_\t^{\a_1} \h \cdot(2l \yb \vr_m+\ya \p_y \vr_m) dxdy.
\end{aligned}
\end{equation*}
It follows from the H\"{o}lder inequality that
\begin{equation}\label{3310}
\begin{aligned}
&|\es \int_\O \p_y \eta_\rho Z_\t^{\a_1} \ps \cdot (2l \yb \vr_m+\ya \p_y \vr_m) dxdy|\\
\le
&\|\p_y \eta_\rho Z_\t^{\a_1} \ps\|_{L^2_{l-1}(\O)}\| \vr_m\|_{L^2_{l}(\O)}
+\|\p_y \eta_\rho Z_\t^{\a_1} \ps\|_{L^2_{l}(\O)}\|\p_y \vr_m\|_{L^2_{l}(\O)}.
\end{aligned}
\end{equation}
Thanks to the relation $\p_y \eta_\rho=\frac{1}{\h+1}\{\p_y^2 \vr-\frac{\p_y \vr \p_y \h}{\h+1}\}$, we get
\begin{equation}\label{3311}
\begin{aligned}
\|\p_y \eta_\rho Z_\t^{\a_1} \ps\|_{L^2_{l}(\O)}
\le
& (\|Z_2 \p_y  \vr\|_{L^\infty_1(\O)}+\d^{-1}\|Z_2 \vr\|_{L^\infty_1(\O)}\| \p_y \h \|_{L^\infty_0(\O) })
\|\frac{1}{\varphi(y)} \frac{Z_\t^{\a_1} \ps}{\h+1}\|_{L^2_{l-1}(\O)} \\
\le
& C_l \d^{-2}(\|Z_2 \p_y  \vr\|_{L^\infty_1(\O)}+\|Z_2 \vr\|_{L^\infty_1(\O)}\| \p_y \h \|_{L^\infty_0(\O)})
\|\h_m\|_{L^2_{l}(\O)},
\end{aligned}
\end{equation}
where, in the last inequality, we have used the  following estimate
\begin{equation*}\label{3312}
\|\frac{1}{\varphi(y)} \frac{Z_\t^{\a_1} \ps}{\h+1}\|_{L^2_{l-1}(\O)}
\le C_l \d^{-1}\|\h_m\|_{L^2_l(\O)}.
\end{equation*}
Combining \eqref{3310} with \eqref{3311}, we conclude that
\begin{equation}\label{3313}
|\es \int_\O \p_y \eta_\rho Z_\t^{\a_1} \ps \cdot (2l \yb \vr_m+\ya \p_y \vr_m) dxdy|
\le \frac{1}{8} \es \|\p_y \vr_m\|_{L^2_l(\O)}^2
  +C_l \d^{-4}Q^2(t) \|(\vr_m, \h_m)\|_{L^2_{l}(\O)}^2.
\end{equation}
Using the Cauchy inequality and the estimate \eqref{b12}, we show
\begin{equation*}\label{3314}
\begin{aligned}
&|\es \int_\O  \eta_\rho Z_\t^{\a_1} \h \cdot(2l \yb \vr_m+\ya \p_y \vr_m) dxdy\\
\le
& \frac{1}{8} \es \|\p_y \vr_m\|_{L^2_l(\O)}^2
  +C_l \d^{-4}(1+\|\p_y \vr \|_{L^\infty_0(\O)}^4
            +\|\p_y \h \|_{L^\infty_1(\O)}^4)\|(\vr_m, \h_m)\|_{L^2_{l}(\O)}^2.
\end{aligned}
\end{equation*}
This and the inequality \eqref{3313} give
\begin{equation}\label{3315}
|\es \int_\O \p_y^2 (\eta_\rho Z_\t^{\a_1} \ps)\cdot \ya \vr_m dxdy|
\le
  \frac{1}{4} \es \|\p_y \vr_m\|_{L^2_l(\O)}^2
  +C_l \d^{-4}(1+Q^2(t))\|(\vr_m,\h_m)\|_{L^2_l(\O)}^2.
\end{equation}
The integration by part with respect to $x$ variable yields immediately
\begin{equation*}\label{3316}
\es \int_\O Z_\t^{\a_1} \ps   \p_x^2  \eta_\rho \cdot \ya \vr_m dxdy
=
\es \int_\O \ya  \p_x   \eta_\rho (Z_\t^{\a_1} \ps \p_x \vr_m+\p_x  Z_\t^{\a_1} \ps \vr_m)dxdy.
\end{equation*}
By virtue of the Holder inequality, we get
\begin{equation}\label{3317}
|\es \int_\O \ya  \p_x   \eta_\rho  Z_\t^{\a_1} \ps \p_x \vr_m dxdy|
\le \es \|\p_x   \eta_\rho  Z_\t^{\a_1} \ps\|_{L^2_{l}(\O)}\|\p_x \vr_m\|_{L^2_{l}(\O)}.
\end{equation}
Due to the fact $\p_x \eta_\rho=\frac{1}{\h+1}\{\p_{xy}^2 \vr-\frac{\p_y \vr \p_x \h}{\h+1}\}$, it follows
\begin{equation*}\label{3318}
\begin{aligned}
\|\p_x   \eta_\rho  Z_\t^{\a_1} \ps\|_{L^2_{l}(\O)}
&\le(\|\p_{xy}^2 \vr\|_{L^\infty_1(\O)}+\d^{-1}\|\p_y \vr\|_{L^\infty_1(\O)}\|\p_x \h\|_{L^\infty_0(\O)})
\|\frac{Z_\t^{\a_1} \ps}{\h+1}\|_{L^2_{l-1}(\O)}\\
&\le C_l \d^{-2}(\|\p_{xy}^2 \vr\|_{L^\infty_1(\O)}+\|\p_y \vr\|_{L^\infty_1(\O)}\|\p_x \h\|_{L^\infty_0(\O)})
\|\h_m\|_{L^2_{l}(\O)},
\end{aligned}
\end{equation*}
where we have used the estimate \eqref{b11} in the last inequality.
This and inequality \eqref{3317} yield directly
\begin{equation}\label{3319}
\begin{aligned}
|\es \int_\O \ya  \p_x   \eta_\rho  Z_\t^{\a_1} \ps \p_x \vr_m dxdy|
\le
 \frac{1}{8}\es \|\p_x \vr_m\|_{L^2_{l}(\O)}^2+C_l \d^{-4}(1+Q^2(t))\| \h_m\|_{L^2_{l}(\O)}^2.
\end{aligned}
\end{equation}
Similarly, it is easy to justify that
\begin{equation*}\label{3319a}
|\es \int_\O \ya  \p_x   \eta_\rho \p_x  Z_\t^{\a_1} \ps \vr_m dxdy|
\le \d_1 \es \|\p_x \h_m\|_{L^2_{l}(\O)}^2+C_l \d^{-4}(1+Q^2(t))\|(\vr_m, \h_m)\|_{L^2_{l}(\O)}^2.
\end{equation*}
which, along with \eqref{3319}, yields directly
\begin{equation}\label{3320}
\begin{aligned}
&|\es \int_\O Z_\t^{\a_1} \ps   \p_x^2  \eta_\rho \cdot \ya \vr_m dxdy|\\
&\le
  \es (\frac{1}{8}\|\p_x \vr_m\|_{L^2_l(\O)}^2+\d_1 \|\p_x \h_m\|_{L^2_l(\O)}^2)
  +C_l \d^{-4}(1+Q^2(t))\|(\vr_m, \h_m)\|_{L^2_l(\O)}^2.
\end{aligned}
\end{equation}
Using the H\"{o}lder inequality and estimate \eqref{b11}, we get
\begin{equation*}\label{3321}
|\int_\O Z_\t^{\a_1} \ps (\p_t+(\u+1-e^{-y})\p_x+\v \p_y)\eta_\rho \cdot \ya \vr_m dxdy|
\le  C_l \d^{-2}(1+Q(t)) \|(\vr_m, \h_m)\|_{L^2_l(\O)}^2.
\end{equation*}
The application of H\"{o}lder and Cauchy inequalities gives directly
\begin{equation*}\label{3322}
|\int_\O (f_\rho-\eta_\rho f_\psi) \cdot \ya \vr_m dxdy|
\le C(\|f_\rho\|_{L^2_l(\O)}^2+ \| f_\psi\|_{L^2_{l-1}(\O)}^2)
+C(1+\|\eta_\rho\|_{L^\infty_1(\O)}^2)\|\vr_m\|_{L^2_l(\O)}^2.
\end{equation*}
Integrating by part and applying the Cauchy inequality, it follows
\begin{equation*}
\begin{aligned}
&|\int_\O(-\es Z_\t^{\a_1} (\p_x r_1+\p_y r_2)+\es \eta_\rho \p_y^{-1} Z_\t^{\a_1} \p_x r_h)\cdot \ya \vr_m dxdy|\\
\le
&\frac{\es}{8}\|(\p_x \vr_m, \p_y \vr_m)\|_{L^2_l(\O)}^2
  +\|Z_\t^{\a_1}(r_1, r_2, r_h)\|_{L^2_l(\O)}^4
  +C_l\d^{-4}(1+Q^2(t))(1+\|\vr_m\|_{L^2_l(\O)}^2).
\end{aligned}
\end{equation*}
Combining the above estimates of terms for the righthand side of \eqref{333}, and
integrating the resulting inequality over $[0, t]$, we get
\begin{equation*}\label{3323}
\begin{aligned}
&\|\vr_m(t)\|_{L^2_l(\O)}^2+\frac{1}{2}\es \int_0^t (\|\p_x \vr_m\|_{L^2_l(\O)}^2+\|\p_y \vr_m\|_{L^2_l(\O)}^2)d\tau\\
\le
&\|\vr_m(0)\|_{L^2_l(\O)}^2+\d_1 \int_0^t(\es \|\p_x  \h_m\|_{L^2_l(\O)}^2+\k \|\p_y \h_m\|_{L^2_l(\O)}^2)d\t
 +\int_0^t \|Z_\t^{\a_1}(r_1, r_2, r_h)\|_{L^2_l(\O)}^4 d\d\\
&+C\int_0^t (\|f_\rho\|_{L^2_l(\O)}^2+\|f_\psi\|_{L^2_{l-1}(\O)}^2) d\t
+C_{\k,l}\d^{-4}(1+Q^2(t))\int_0^t(1+ \|(\vr_m, \h_m)\|_{L^2_l(\O)}^2)d\t.
\end{aligned}
\end{equation*}
On the other hand, we applying the Moser type inequality \eqref{ineq-moser} to get
\begin{equation*}\label{3324}
\int_0^t(\|f_\rho\|_{L^2_l(\O)}^2+\| f_\psi\|_{L^2_{l-1}(\O)}^2) d\tau
\le C_{m,l} Q(t)\int_0^t (\|(\vr, \u, \h)\|_{\H^{m}_l}^2+\|\p_y \vr\|_{\H^{m-1}_l}^2) d\tau,
\end{equation*}
which along with the estimate \eqref{b22} completes the proof of lemma.
\end{proof}

Next, we establish the estimate for the quantities $\u_m$ and $\h_m$.
Indeed, it is easy to check that
\begin{equation*}\label{3b11}
\begin{aligned}
&-(\h+1)\p_x Z_\t^{\a_1} \h-\g \p_y Z_\t^{\a_1} \h- Z_\t^{\a_1} \g \p_y \h\\
=
&-(\h+1)\p_x \h_m-\g \p_y \h_m-(\h+1)\p_x \eta_h Z_\t^{\a_1} \ps
-\g \p_y \eta_h Z_\t^{\a_1} \ps-\g \eta_h Z_\t^{\a_1} \h,
\end{aligned}
\end{equation*}
and
\begin{equation*}\label{3b12}
\begin{aligned}
&-(\h+1)\p_x Z_\t^{\a_1}\u-\g \p_y Z_\t^{\a_1} \u-Z_\t^{\a_1}\g \p_y(\u-e^{-y})\\
=
&-(\h+1)\p_x \u_m-\g \p_y \u_m-(\h+1)\p_x \eta_u Z_\t^{\a_1} \ps
-\g \p_y \eta_u Z_\t^{\a_1} \ps-\g \eta_u Z_\t^{\a_1} \h,
\end{aligned}
\end{equation*}
which were first observed in \cite{Liu-Xie-Yang}.
Then, we will have the following estimates:

\begin{lemma}\label{lemma34}
Let $(\vr, \u, \v, \h, \g)$ be sufficiently smooth solution, defined on $[0, T^\es]$,
to the equations \eqref{eq5}.
Under the assumption of conditions \eqref{a2} and \eqref{a1}, it holds on
\begin{equation*}\label{341}
\begin{aligned}
&\underset{0\le \t \le t}{\sup}\|(\u_m, \h_m)(\t)\|_{L^2_l(\O)}^2
+\int_0^t \|\sqrt{\es}\p_x (\u_m, \h_m)\|_{L^2_l(\O)}^2 d\t
+\int_0^t \|\p_y (\sqrt{\mu} \u_m, \sqrt{\k} \h_m)\|_{L^2_l(\O)}^2 d\t\\
\le
&C\|(\u_m, \h_m)(0)\|_{L^2_l(\O)}^2
+C\int_0^t \|(r_u, r_h)\|_{\H^{m}_{l, tan}}^4d\t
 +C_{\mu, \k, m, l} \d^{-6}(1+Q^3(t))\int_0^t \mathcal{E}_{m,l}(\t) d\t.\\
& +C_{\mu, \k, m, l}
   \d^{-6}(1+Q^3(t))\int_0^t (\|(\vr_m, \u_m, \h_m)\|_{L^2_l(\O)}^2+\|(\p_y \u, \p_y \h)\|_{\H^{m-1}_l}^2)d\t.
\end{aligned}
\end{equation*}
\end{lemma}

\begin{proof}
Multiplying the equation \eqref{equm} by $\ya \u_m$, integrating over $\O$ and
integrating by part, we find
\begin{equation}\label{342}
\begin{aligned}
&\frac{1}{2}\frac{d}{dt}\int_\O \ya \r |\u_m|^2 dxdy
+\es \int_\O \ya |\p_x \u_m|^2 dxdy+\mu \int_\O \ya |\p_y \u_m|^2 dxdy\\
&+\int_\O \ya (\h+1)\p_x \u_m \cdot \h_m dxdy
 +\int_\O \ya \g \p_y \u_m \cdot \h_m dxdy=\sum_{i=1}^9 I_i,
\end{aligned}
\end{equation}
where $I_i(i=1,...,9)$ are defined by
\begin{equation*}
\begin{aligned}
&I_1=\frac{1}{2}\int_\O \ya |\u_m|^2(\p_t \r+(\u+1-e^{-y})\p_x \r+\v \p_y \r)dxdy,\\
&I_2=l \int_\O \yb \r \v |\u_m|^2 dxdy, \quad I_3=-2 l \int_\O \yb (\mu\p_y \u_m +\g \h_m)\cdot \u_m dxdy,\\
&I_4=\int_\O (-(\h+1)\p_x \eta_h Z_\t^{\a_1} \ps
-\g \p_y \eta_h Z_\t^{\a_1} \ps-\g \eta_h Z_\t^{\a_1} \h )\cdot \ya \u_m dxdy,\\
&I_5=\int_\O[-Z_\t^{\a_1} \ps (\r \p_t+\r(\u+1-e^{-y})\p_x+\r\v \p_y)\eta_u]\cdot \ya \u_m dxdy,\\
&I_6=\int_\O(f_u-\r \eta_u f_\psi
-\k \r \eta_u \p_y^2 Z_\t^{\a_1} \ps+2\es \p_x \eta_u \p_x Z_\t^{\a_1} \ps)\cdot \ya \u_m dxdy,\\
&I_7=\int_\O(\es \p_x^2 \eta_u Z_\t^{\a_1} \ps
+\mu \p_y^2(  \eta_u Z_\t^{\a_1} \ps )\cdot \ya \u_m dxdy, \quad
I_8=\es\int_\O (1-\r )\eta_u \p_x^2 Z_\t^{\a_1}\ps \cdot \ya \u_m dxdy,\\
&I_9=\int_\O(-\es Z_\t^{\a_1} \p_x r_u+\es \rho^\es \eta_\rho \p_y^{-1} Z_\t^{\a_1} \p_x r_h)\cdot \ya \u_m dxdy.
\end{aligned}
\end{equation*}
By routine checking, we may show that
\begin{equation*}\label{343}
|I_1|+|I_2|\le C_l (1+\|(\u, \v, \p_t \vr, \p_x \vr, \p_y \vr) \|_{L^\infty_0(\O)}^2)\|\u_m\|_{L^2_l(\O)}^2.
\end{equation*}
By virtue of the Holder and Cauchy inequalities, we find
\begin{equation*}\label{344}
|I_3|\le \frac{1}{8} \mu \|\p_y \u_m\|_{L^2_l(\O)}^2
+C_{\mu, l}(1+\|\g\|_{L^\infty_{-1}(\O)}^2)(\| \u_m\|_{L^2_l(\O)}^2+\|\h_m\|_{L^2_l(\O)}^2).
\end{equation*}
Deal with the term $I_4$.
By virtue of  $\h+1\ge \d$ and estimate \eqref{b12}, we apply the H\"{o}lder inequality to get
\begin{equation}\label{345}
\begin{aligned}
&|\int_\O \g \eta_h Z_\t^{\a_1} \h  \cdot \ya \u_m dxdy|\\
\le
&\d^{-1}\|\g\|_{L^\infty_{-1}(\O)}\|\p_y \h\|_{L^\infty_1(\O)}
\|Z_\t^{\a_1} \h\|_{L^2_l(\O)}\|\u_m\|_{L^2_l(\O)}\\
\le
&C_l\d^{-2}(\|\g\|_{L^\infty_{-1}(\O)}^2+\|\p_y \h\|_{L^\infty_1(\O)}^2)
(\|\u_m\|_{L^2_l(\O)}^2+\|\h_m\|_{L^2_l(\O)}^2).
\end{aligned}
\end{equation}
Due to the fact
$(\h+1)\p_x \eta_h=\p_{xy}\h-\frac{\p_y \h \p_x \h}{\h+1}$,
we apply the estimate \eqref{b11} and H\"{o}lder inequality to get
\begin{equation}\label{346}
\begin{aligned}
&|\int_\O (\h+1)\p_x \eta_h Z_{\t}^{\a_1}\ps\cdot \ya \u_m dxdy|\\
\le
&(1+\|\h\|_{L^\infty_0(\O)})\|(\h+1)\p_x \eta_h\|_{L^\infty_1(\O)}
\|\frac{ Z_{\t}^{\a_1}\ps}{\h+1}\|_{L^2_{l-1}(\O)}\|\u_m\|_{L^2_l(\O)}\\
\le
&C_l \d^{-2}(1+Q^2(t))\|(\u_m, \h_m)\|_{L^2_l(\O)}^2.
\end{aligned}
\end{equation}
Similarly, we also have
\begin{equation*}\label{347}
|\int_\O \g \p_y \eta_h Z_\t^{\a_1} \ps \cdot \ya \u_m dxdy|
\le C_l \d^{-2}(1+Q^2(t))\|(\u_m, \h_m)\|_{L^2_l(\O)}^2.
\end{equation*}
This, along with inequalities \eqref{345} and \eqref{346}, yields directly
\begin{equation*}\label{348}
|I_4|\le C_l \d^{-2}(1+Q^2(t)) \|(\u_m, \h_m)\|_{L^2_l(\O)}^2,
\end{equation*}
Similarly, we also get that
\begin{equation*}\label{349}
|I_5|\le C_l \d^{-2}(1+Q^2(t)) \|(\u_m, \h_m)\|_{L^2_l(\O)}^2.
\end{equation*}
Deal with the term $I_6$.
By virtue of the H\"{o}lder and Cauchy inequalities, we get
\begin{equation}\label{3410}
|\int_\O (f_u-\rho^\es \eta_u f_\psi) \cdot \ya \u_m dxdy|
\le  C(\|f_u\|_{L^2_l(\O)}^2+ \| f_\psi\|_{L^2_{l-1}(\O)}^2)
+C(1+\|\eta_u\|_{L^\infty_1(\O)}^2)\|\u_m\|_{L^2_l(\O)}^2.
\end{equation}
Similar to the estimate \eqref{345}, it is easy to check that
\begin{equation}\label{3411}
\begin{aligned}
&|\k \int_\O \r \eta_u \p_y^2 Z_\t^{\a_1} \ps \cdot \ya \u_m dxdy|\\
\le
&C\k\d^{-1}\|\p_y \u\|_{L^\infty_0(\O)}\|\p_y Z_\t^{\a_1} \h\|_{L^2_l(\O)}\|\u_m\|_{L^2_l(\O)}\\
\le
&\d_2 \k \|\p_y \h_m\|_{L^2_l(\O)}^2
  +C_{\k, l}\d^{-2}(\| Z_2 \p_y \h \|_{L^\infty_1(\O)}^2+\|(\p_y \u, \p_y \h)\|_{L^\infty_0(\O)}^2)
  \|(\u_m, \h_m)\|_{L^2_l(\O)}^2,
\end{aligned}
\end{equation}
where we have used the estimate \eqref{b14} in the last inequality.
Similarly, we also have
\begin{equation*}\label{3412}
\begin{aligned}
&|2\es \int_\O  \p_x \eta_u \p_x Z_\t^{\a_1} \ps  \cdot \ya \u_m dxdy|\\
\le
&2\es (\|\p_{xy}^2 \u\|_{L^\infty_1(\O)}+\|\p_y \u\|_{L^\infty_1(\O)}\|\p_x \h\|_{L^\infty_0(\O) })
\|\frac{\p_x Z_\t^{\a_1} \ps}{\h+1}\|_{L^2_{l-1}(\O)}\|\u_m\|_{L^2_l(\O)}\\
\le
&\d_1 \es \|\p_x \h_m\|_{L^2_l(\O)}^2
  +C_l \d^{-4}(1+\|(\p_{xy}\u, \p_y \u)\|_{L^\infty_1(\O)}^4+\|\p_x \h\|_{L^\infty_0(\O)}^4)\|(\u_m, \h_m)\|_{L^2_l(\O)}^2.
\end{aligned}
\end{equation*}
This, along with inequalities \eqref{3410} and  \eqref{3411}, yields directly
\begin{equation*}\label{3413}
\begin{aligned}
|I_6|
\le
&\d_2 \k \|\p_x \h_m\|_{L^2_l(\O)}^2+\d_2 \es \|\p_y \h_m\|_{L^2_l(\O)}^2
          + C(\|f_u\|_{L^2_l(\O)}^2+ \| f_\psi\|_{L^2_{l-1}(\O)}^2)\\
&+C_{\k, l}\d^{-4}(1+Q^2(t))\|(\u_m, \h_m)\|_{L^2_l(\O)}^2.
\end{aligned}
\end{equation*}
Now, we give the estimate for the term $I_7$.
Similar to the estimates \eqref{3315} and \eqref{3320}, we can obtain
\begin{equation*}\label{3414}
|\es \int_\O \p_x^2 \eta_u Z_\t^{\a_1} \ps \cdot \ya \u_m dxdy|
\le \frac{1}{8} \es \|\p_x \u_m\|_{L^2_l(\O)}^2
  +C_l \d^{-4}(1+Q^2(t))\|(\u_m, \h_m)\|_{L^2_{l}(\O)}^2,
\end{equation*}
and
\begin{equation*}\label{3415}
|\es \int_\O \p_y^2 (\eta_u Z_\t^{\a_1} \ps)\cdot \ya \u_m dxdy|
\le \frac{1}{8} \mu \|\p_y \u_m\|_{L^2_l(\O)}^2+C_{\mu, l}\d^{-4}(1+Q^2(t))\|(\u_m, \h_m)\|_{L^2_{l}(\O)}^2,
\end{equation*}
and hence, it follows
\begin{equation*}\label{3416}
|I_7|\le
\frac{1}{8}(\es\|\p_x \u_m\|_{L^2_l(\O)}^2+\mu\|\p_y \u_m\|_{L^2_l(\O)}^2)
+C_{\mu, l}\d^{-4} (1+Q^2(t))\|(\u_m, \h_m)\|_{L^2_l(\O)}^2.
\end{equation*}
Finally, we deal with the term $I_8$. Indeed, the integration by part with respect to $x$ yields directly
\begin{equation}\label{3417}
\begin{aligned}
I_8
=&\es \int_\O \ya \p_x  \vr \eta_u \u_m \cdot \p_x Z_\t^{\a_1}\ps  dxdy\\
&+\es \int_\O \ya   \vr \eta_u \p_x \u_m \cdot \p_x Z_\t^{\a_1}\ps  dxdy\\
&+\es \int_\O \ya   \vr \p_x \eta_u  \u_m \cdot \p_x Z_\t^{\a_1}\ps  dxdy\\
=&I_{81}+I_{82}+I_{83}.
\end{aligned}
\end{equation}
By virtue of the estimate \eqref{b13}, Holder and Cauchy equalities, we find
\begin{equation}\label{3418}
\begin{aligned}
I_{81}
&\le \es \| \p_x \vr \|_{L^\infty_0(\O)}\| \p_y (\u-e^{-y}) \|_{L^\infty_1(\O)}
      \|\u_m\|_{L^2_l(\O)} \|\frac{\p_x Z^{\a_1}_\t \ps}{h+1}\|_{L^2_{l-1}(\O)}\\
&\le \d_2 \es \|\p_x \h_m\|_{L^2_l(\O)}^2
     +C_l\d^{-4}(1+\|\p_y \u\|_{L^\infty_1(\O)}^4+\|(\p_x \vr, \p_x \h)\|_{L^\infty_0(\O)}^4)
      \|(\u_m, \h_m)\|_{L^2_l(\O)}^2.
\end{aligned}
\end{equation}
Similar to the estimate \eqref{3412}, it is easy to justify
\begin{equation}\label{3419}
I_{83}\le
\d_2 \es \|\p_x \h_m\|_{L^2_l(\O)}^2+C_l \d^{-4}(1+\|(\p_{xy}^2 \u, \p_y \u)\|_{L^\infty_1(\O)}^4
     +\|\p_x \h\|_{L^\infty_0(\O)}^4)\|(\u_m,\h_m)\|_{L^2_l(\O)}^2.
\end{equation}
Using the H\"{o}lder inequality and the estimate \eqref{b13}, it follows
\begin{equation*}\label{3420}
\begin{aligned}
I_{82}
\le
& \es \|\vr \|_{L^\infty_0(\O)}\|\p_y (\u-e^{-y})\|_{L^\infty_1(\O)}
     \|\p_x \u_m\|_{L^2_l(\O)} \|\frac{\p_x Z^{\a_1}_\t \ps}{\h+1}\|_{L^2_{l-1}(\O)}\\
\le
& \frac{4\es\d^{-1}}{2l-1} \|\vr \|_{L^\infty_0(\O)}\| \p_x \h\|_{L^\infty_0(\O)} \|\p_y (\u-e^{-y})\|_{L^\infty_1(\O)}
     \|\p_x \u_m\|_{L^2_l(\O)} \|\h_m\|_{L^2_l(\O)}^2\\
     &+ \frac{2\es\d^{-1}}{2l-1} \|\vr \|_{L^\infty_0(\O)}\|\p_y (\u-e^{-y})\|_{L^\infty_1(\O)}
     \|\p_x \u_m\|_{L^2_l(\O)} \|\p_x \h_m\|_{L^2_l(\O)}\\
\le
& 2\es\| \p_x \h\|_{L^\infty_0(\O)}
     \|\p_x \u_m\|_{L^2_l(\O)} \|\h_m\|_{L^2_l(\O)}+ \es\|\p_x \u_m\|_{L^2_l(\O)} \|\p_x \h_m\|_{L^2_l(\O)},
\end{aligned}
\end{equation*}
where we have used the condition {$\|\p_y (\u-e^{-y})\|_{L^\infty_1(\O)}\le \d^{-1}$
and $\|\vr \|_{L^\infty_0(\O)} \le (l-\frac{1}{2})\d^2$} in the last inequality.
and hence, it follows
\begin{equation}\label{3421}
I_{82}\le
     (\frac{1}{2}+\d_2)\es\|\p_x \h_m\|_{L^2_l(\O)}^2
      +\frac{1}{2}\es \|\p_x \u_m\|_{L^2_l(\O)}^2
      +C \| \p_x \h\|_{L^\infty_0(\O)}^2 \|\h_m\|_{L^2_l(\O)}^2,
\end{equation}
Then, substituting the estimates \eqref{3418}, \eqref{3419} and \eqref{3421} into \eqref{3417}, we get
\begin{equation*}\label{3422}
|I_{8}|\le
      (\frac{1}{2}+3\d_2)\es\|\p_x \h_m\|_{L^2_l(\O)}^2
      +\frac{1}{2}\es \|\p_x \u_m\|_{L^2_l(\O)}^2
      +C_l \d^{-4}(1+Q^2(t))\|(\u_m, \h_m)\|_{L^2_l(\O)}^2.
\end{equation*}
Finally, integrating by part and applying the Cauchy inequality, we get
\begin{equation*}
|I_9|\le \frac{\es}{8}\|\p_x \u_m\|_{L^2_l(\O)}^2
+\|Z_\t^{\a_1}(r_u, r_h)\|_{L^2_l(\O)}^4+C_l \d^{-4}(1+Q^2(t))\|\u_m\|_{L^2_l(\O)}^2.
\end{equation*}
Then, substituting the estimates of $I_1$ through $I_9$
into the equality \eqref{342}, and integrating the resulting inequality over $[0, t]$, we get that
\begin{equation}\label{3423}
\begin{aligned}
& \|\sqrt{\vr}\u_m(t)\|_{L^2_l(\O)}^2  +\frac{\es}{2} \int_0^t \|\p_x \u_m\|_{L^2_l(\O)}^2 d\t
+\frac{\mu}{2}  \int_0^t \| \p_y \u_m\|_{L^2_l(\O)}^2 d\t\\
&+\int_0^t \int_\O \ya [(\h+1)\p_x \u_m \cdot \h_m+ \g \p_y \u_m \cdot \h_m] dxdy d\t\\
\le
&\|(\sqrt{\vr}\u_m)(0)\|_{L^2_l(\O)}^2
+(\frac{1}{2}+3\d_2) \es \int_0^t \|\p_x \h_m\|_{L^2_l(\O)}^2 d\t
+\d_2 \k\int_0^t \|\p_y \h_m\|_{L^2_l(\O)}^2 d\t\\
&+C\int_0^t \|Z_\t^{\a_1}(r_u, r_h)\|_{L^2_l(\O)}^4d\t
+C\int_0^t (\|f_u\|_{L^2_l(\O)}^2+ \| f_\psi\|_{L^2_{l-1}(\O)}^2) d\t\\
&
+C_{\mu, \k, l} \d^{-4}(1+Q^2(t))\int_0^t \|(\u_m, \h_m)\|_{L^2_l(\O)}^2 d\t.
\end{aligned}
\end{equation}
Applying the Moser type inequality \eqref{ineq-moser}, it is easy to justify
\begin{equation}\label{3424}
 \int_0^t\!  \|f_u\|_{L^2_l(\O)}^2 d\t
 \le C_{m, l}(1+Q^2(t))\! \int_0^t \!(\|(\vr, \u, \h)\|_{\H^m_l}^2\!+\|(\p_y \u, \p_y \h)\|_{\H^{m-1}_l}^2)d\t.
\end{equation}
Similarly, thanks to the equation \eqref{eqhm}, it is easy to obtain the estimate
\begin{equation}\label{3425}
\begin{aligned}
&\|\h_m (t)\|_{L^2_l(\O)}^2+\frac{3}{4}\es \int_0^t \|\p_x \h_m\|_{L^2_l(\O)}^2 d\t
+\frac{3}{4} \k\int_0^t \|\p_y \h_m\|_{L^2_l(\O)}^2 d\t\\
&-\int_0^t \int_\O \ya [(\h+1)\p_x \u_m \cdot \h_m+ \g \p_y \u_m \cdot \h_m] dxdy d\t\\
\le
&\|\h_m(0)\|_{L^2_l(\O)}^2
+C\int_0^t \| Z_\t^{\a_1}r_h\|_{L^2_l(\O)}^4 d\t
+C_{\k, m, l} \d^{-6}(1+Q^3(t))\int_0^t \mathcal{E}_{m,l}(\t) d\t\\
&+C_{\k, m, l} \d^{-6}(1+Q^3(t))\int_0^t (\|(\u_m, \h_m)\|_{L^2_l(\O)}^2+\|(\p_y \u, \p_y \h)\|_{\H^{m-1}_l}^2)d\t.
\end{aligned}
\end{equation}
Therefore, combining the estimates \eqref{3423}-\eqref{3425} with \eqref{b22},
and choosing $\d_2$ small enough, we complete the proof of lemma.
\end{proof}

Therefore, combining the estimates in Lemmas \ref{lemma33} and \ref{lemma34}
and choosing the constant $\d_1$ small enough, then we complete the proof of Proposition \ref{Tanential-estimate}.

\begin{remark}\label{remark-condition}
To deal with the term $\es\int_\O (1-\r )\eta_u \p_x^2 Z_\t^{\a_1}\ps \cdot \ya \u_m dxdy$
(i.e., the term $I_8$ on the right handside of equality \eqref{342}),
we require the assumption of condition \eqref{a1}.
In other words, the condition \eqref{a1} is not needed for the homogeneous flow($\r \equiv 1$)
since this difficult term will disappear.
\end{remark}

\subsection{Weighted $\H^{m-1}_l-$Estimates for Normal Derivative}

In this subsection, we shall provide an estimate for
$\|(\p_y \vr, \p_y \u, \p_y \h)\|_{\H^{m-1}_l}$,
which will be given as follows:

\begin{proposition}\label{Normal-estimate}
Let $(\vr, \u, \v, \h, \g)$ be sufficiently smooth solution, defined on $[0, T^\es]$,
to the equations \eqref{eq5}-\eqref{bc5}. Under the assumption of condition \eqref{a2}, it holds on
\begin{equation*}
\begin{aligned}
&\sup_{0\le \t \le t}\|(\p_y \vr, \p_y \u, \p_y \h)(\t)\|_{\H^{m-1}_l}^2
+\es \int_0^t \|\p_x(\p_y \vr, \p_y \u, \p_y \h) \|_{\H^{m-1}_l}^2 d\t\\
&+\int_0^t (\es \|\p_y^2 \vr\|_{\H^{m-1}_l}^2+\mu \|\p_y^2 \u\|_{\H^{m-1}_l}^2
            +\k \|\p_y^2 \h\|_{\H^{m-1}_l}^2) d\t\\
\le
&C\|(\p_y \vr_0, \p_y \u_0, \p_y \h_0)\|_{\H^{m-1}_l}^2
+C_{\mu, \k} t \|(\rho_0, u_{10}, h_{10})\|_{\widehat{\mathcal{B}}^m_l}
+C_{\mu, \k, m, l}\d^{-2}(1+Q^3(t))\int_0^t X_{m,l}(\t) d\t.
\end{aligned}
\end{equation*}
\end{proposition}
First of all, we establish the estimate for the quantity $\p_y \vr$ in $\H^{m-1}_l$ norm.
To this end, differentiating the density equation \eqref{eq5}$_1$ with respect to $y$ variable,
we get the evolution equation for $\p_y \vr$:
\begin{equation}\label{3c1}
(\p_t +(\u+1-e^{-y})\p_x +\v \p_y -\es \p_x^2-\es \p_y^2)\p_y \vr
=f_1,
\end{equation}
where the function $f_1$ is defined by
\begin{equation*}\label{3c2}
f_1:=-\es \p_y (\p_x r_1+\p_y r_2)-(\p_y \u+e^{-y})\p_x \vr+\p_x \u \p_y \vr.
\end{equation*}

\begin{lemma}\label{lemma37}
For smooth solution $(\vr, \u, \v, \h,\g)$ of the equations \eqref{eq5}-\eqref{bc5}, then it holds on
\begin{equation*}\label{371}
\begin{aligned}
&\sup_{\tau \in [0 ,t]}\|\p_y \vr(\t)\|_{\H^{m-1}_l}^2
+\es\int_0^t (\|\p_{xy} \vr\|_{\H^{m-1}_l}^2+\|\p_y^2 \vr\|_{\H^{m-1}_l}^2) d\tau\\
\le
&\|\p_y \vr_0\|_{\H^{m-1}_l}^2\!+\!\!\int_0^t\!\! \|\p_y (r_1, r_2)\|_{\H^{m-1}_l}^2 d\tau
+C_{m,l}(1+Q(t))\!\!\int_0^t\!\!(1+\|(\vr, \u)\|_{\H^m_l}^2+\|\p_y(\vr, \u)\|_{\H^{m-1}_l}^2)d\tau.
\end{aligned}
\end{equation*}
\end{lemma}

\begin{proof}
We will give the proof of the estimate \eqref{371} by induction.
First of all, multiplying \eqref{3c1} by $\ya \p_y \vr$, integrating over $\O$ and
integrating by part with respect to $x$ variable, we find
\begin{equation*}\label{373}
\begin{aligned}
&\frac{1}{2}\frac{d}{dt}\int_{\O} \ya |\p_y \vr|^2 dxdy
+\es \int_{\O} \ya |\p_x \p_y \vr|^2 dxdy\\
=
&\es \int_{\O} \p_y^3 \vr \cdot \ya \p_y \vr dxdy
+l \int_{\O} \yb \v |\p_y \vr|^2 dxdy
+\int_{\O} f_1 \cdot \ya \p_y \vr dxdy.
\end{aligned}
\end{equation*}
Integrating by part and applying the boundary condition $\p_y \vr|_{y=0}=0$, we get
\begin{equation*}\label{374}
\begin{aligned}
&\es \int_\O \p_y^3 \vr \cdot \ya \p_y \vr dxdy\\
=
&\es \int_\T \p_y^2 \vr \cdot \p_y \vr|_{y=0} dx
-\es \int \p_y (\ya \p_y \vr) \cdot \p_y^2 \vr dxdy\\
=
&-\es \int_\O \ya |\p_y^2 \vr|^2 dxdy
-2l \es \int_\O \yb \p_y \vr \cdot \p_y^2 \vr dxdy\\
\le
&-\frac{1}{2}\es \int \ya |\p_y^2 \vr|^2 dxdy
+C_l \|\p_y \vr\|_{L^2_{l-1}(\Omega)}^2,
\end{aligned}
\end{equation*}
where we have used the H\"{o}lder and Cauchy inequalities in the last inequality.
Thus we get after integrating the resulting inequality over $[0, t]$ with time variable
\begin{equation*}
\begin{aligned}
& \int_{\O} \ya |\p_y \vr|^2 dxdy
+\es \int_0^t\int_{\O} \ya (|\p_x \p_y \vr|^2+|\p_y^2 \vr|^2) dxdyd\t\\
\le
&\int_{\O} \ya |\p_y \vr_0|^2 dxdy
+2l \int_0^t \int_{\O} \yb |\v| |\p_y \vr|^2 dxdyd\t\\
&
+2\int_0^t \int_{\O}  \ya |f_1| |\p_y \vr| dxdyd\t
+C_l \int_0^t \|\p_y \vr\|_{L^2_{l-1}(\Omega)}^2 d\t,
\end{aligned}
\end{equation*}
which implies directly
\begin{equation*}\label{375}
\begin{aligned}
& \int_{\O} \ya |\p_y \vr|^2 dxdy
+\es \int_0^t(\|\p_x \p_y \vr\|_{L^2_l(\O)}^2+\|\p_y^2 \vr\|_{L^2_l(\O)}^2) d\t\\
\le
&\int_{\O} \ya |\p_y \vr_0|^2 dxdy
+\int_0^t \|\p_y(r_1, r_2)\|_{L^2_l(\O)}^2 d\t
+C_l(1+Q(t)) \int_0^t  (\|\vr\|_{\H^1_l}^2+\|\p_y \vr\|_{L^2_l(\O)}^2)d\t.
\end{aligned}
\end{equation*}
Obviously, this inequality implies the estimate \eqref{371} holds on
for $m=1$. To prove the general case, assume that \eqref{371}
is proven for $k \le m-2$, we need to prove it holds on also for $k=m-1$.
Applying the operator $\z^\a (|\a|= m-1)$ to the equation \eqref{3c1}, we get
\begin{equation}\label{376}
(\p_t+(\u+1-e^{-y})\p_x +\v \p_y)\z^\a \p_y \vr
-\es \z^\a \p_x^2 \p_y \vr-\es \z^\a \p_y^3 \vr
=\z^\a f_1+\C_{41}+\C_{42},
\end{equation}
where $\C_{4i}(i=1,2)$ are defined by
$$
\C_{41}=-[\z^\a, (\u+1-e^{-y})\p_x]\p_y \vr, \quad
\C_{42}=-[\z^\a, \v \p_y]\p_y \vr.
$$
Multiplying the equation \eqref{376} by $\ya \z^\a \p_y \vr$,
integrating over $\O \times [0, t]$,
and integrating by part with  respect to $x$ variable, we find
\begin{equation}\label{377}
\begin{aligned}
& \frac{1}{2}\int_\O \ya |\z^\a \p_y \vr|^2 dxdy
+\es \int_0^t \int_\O \ya |\p_x \z^\a \p_y \vr|^2 dxdy d\tau\\
=
&\frac{1}{2}\int_\O \ya |\z^\a \p_y \vr_0|^2 dxdy+I_{21}+I_{22}+I_{23}+I_{24}+I_{25},
\end{aligned}
\end{equation}
where the term $I_{2i}(i=1,...,5)$ are defined by
\begin{equation*}
\begin{aligned}
&I_{21}=\es \int_0^t\int_\O \z^\a \p_y^3 \vr \cdot \ya \z^\a \p_y \vr dxdyd\tau,\quad
 I_{22}=\int_0^t \int_\O \z^\a f_1 \cdot \ya \z^\a \p_y \vr dxdyd\tau,\\
&I_{23}=l \int_0^t \int_\O \yb \v |\z^\a \p_y \vr|^2 dxdyd\tau,\quad
 I_{24}=\int_0^t \int_\O  \C_{41} \cdot \ya \z^\a \p_y \vr dxdyd\tau,\\
&I_{25}=\int_0^t \int_\O  \C_{42} \cdot \ya \z^\a \p_y \vr dxdyd\tau.
\end{aligned}
\end{equation*}
Similar to the estimate \eqref{3225}, we can get
\begin{equation*}\label{378}
I_{21}
\le -\frac{1}{2}\es \int_0^t \int_\O \ya |\p_y \z^\a \p_y \vr|^2 dxdyd\tau
+C_{m,l}\int_0^t (\es\|\p_y \vr\|_{\H^{m-1}_l}^2+\es\|\p_y^2 \vr\|_{\H^{m-2}_l}^2)d\tau.
\end{equation*}
It is easy to justify
\begin{equation*}\label{379}
|I_{23}|\le C_l \|\v \|_{L^\infty_0(\O)}^2 \int_0^t \|\p_y \vr\|_{\H^{m-1}_l}^2 d\tau
\end{equation*}
Applying the Moser type inequality \eqref{ineq-moser}, we conclude
\begin{equation*}\label{3710}
\begin{aligned}
|I_{22}|
\le
&\frac{1}{4}\es \int_0^t(\|\p_x \z^\a \p_y \vr\|_{L^2_l(\O)}^2+\|\p_y \z^\a \p_y \vr\|_{L^2_l(\O)}^2)d\t
+\int_0^t \|\p_y (r_1, r_2)\|_{\H^{m-1}_l}^2 d\tau\\
&+C_{m,l}(1+Q(t))\int_0^t (1+\|(\vr, \u)\|_{\H^m_l}^2+\|(\p_y \vr, \p_y \u)\|_{\H^{m-1}_l}^2) d\tau,
\end{aligned}
\end{equation*}
and
\begin{equation*}\label{3710}
|I_{24}| \le C_m(1+Q(t))\int_0^t (1+\|\u \|_{\H^m_l}^2+\|\p_y \vr\|_{\H^{m-1}_l}^2)d\tau.
\end{equation*}
Finally, we deal with the term $I_{25}$.
It follows from the H\"{o}lder inequality that
\begin{equation}\label{3711}
|I_{25}|\le \int_0^t \|\C_{42} \|_{L^2_l(\Omega)} \|\z^\a \p_y \vr\|_{L^2_l(\Omega)} d\tau.
\end{equation}
It is easy to check that
$$
[\z^\a, \v \p_y]\vr=[\z^\a, \v]\p_y^2 \vr+\v [\z^\a, \p_y]\p_y \vr.
$$
Since the coefficient $\es$
of the quantity $\p_y^2 \vr$ in \eqref{eq5}$_1$ is sufficiently small,
it is not expected to establish a estimate which is uniform in $\es$
for $\|\p_y^2 \vr\|_{L^\infty_0(\O)}$ or $\|\p_y^2 \vr\|_{\H^{m-1}_l}$.
Hence, we first write
$$
[Z_2^{\a_2}, \p_y]\p_y \vr=\sum_{\b_2 \neq 0, \b_2+\ga_2=\a_2}
   C_{\b_2, \ga_2}Z_2^{\b_2}(\frac{1}{\varphi})Z_2^{\ga_2+1}\p_y \vr,
$$
and get
\begin{equation}\label{3712}
\begin{aligned}
\int_0^t \|\v Z_2^{\b_2}(\frac{1}{\varphi})Z_2^{\ga_2+1}Z_\tau^{\a_1 }\p_y \vr\|_{L^2_l(\O)}^2 d\tau
\le
C\|\frac{\v}{\varphi}\|_{L^\infty_0(\O)}^2 \int_0^t \|\p_y \vr\|_{\H^{m-1}_l}^2 d\tau,
\end{aligned}
\end{equation}
where we have used the estimate \eqref{3213} in the last inequality.
Similarly, we have
$$
[\z^\a, \p_y] \p_y^2 \vr=\sum_{|\b+\ga|\le m-1, |\ga|\le m-2}
C_{\b, \ga, \varphi}\z^\b (\frac{\v}{\varphi})\z^\ga (Z_2 \p_y \vr)
$$
If $\b =0$, it is easy to verify
\begin{equation*}\label{3713}
\int_0^t \|\frac{\v}{\varphi}\z^\ga (Z_2 \p_y \vr)\|_{L^2_l}^2 d\tau
\le \|\frac{\v}{\varphi}\|_{L^\infty_0(\O)}^2\int_0^t \|\p_y \vr\|_{\H^{m-1}_l}^2 d\tau.
\end{equation*}
If $\b \neq 0$, the application of Moser type inequality \eqref{ineq-moser} yields directly
\begin{equation}\label{3714}
\begin{aligned}
&\int_0^t \|\z^\b (\frac{\v}{\varphi})\z^\ga (Z_2 \p_y \vr)\|_{L^2_l(\O)}^2 d\tau\\
\le
&C\|\z^{E_i} (\frac{\v}{\varphi})\|_{L^\infty_1(\Omega)}^2
 \int_0^t \|Z_2 \p_y \vr\|_{\H^{m-2}_{l-1}}^2 d\tau\\
& +C\|Z_2 \p_y \vr\|_{L^\infty_1(\Omega)}^2
 \int_0^t \|\z^{E_i} (\frac{\v}{\varphi}) \|_{\H^{m-2}_{l-1}}^2 d\tau.
\end{aligned}
\end{equation}
Using the Hardy inequality and divergence-free condition
of velocity in \eqref{eq5}$_4$, we get
\begin{equation*}\label{3715}
\begin{aligned}
\|\z^{E_i} (\frac{\v}{\varphi}) \|_{\H^{m-2}_{l-1}}^2
\le \|\frac{\v}{y} \|_{\H^{m-1}_{l}}^2
\le C_l\|\p_x \u \|_{\H^{m-1}_{l}}^2
\le C_l\| \u \|_{\H^{m}_{l}}^2,
\end{aligned}
\end{equation*}
which, together with \eqref{3714}, yields directly
\begin{equation*}\label{3716}
\begin{aligned}
\int_0^t \|\z^\b (\frac{\v}{\varphi})\z^\ga (Z_2 \p_y \vr)\|_{L^2_l(\O)}^2 d\tau
\le
C_{l} \|(\z^{E_i} (\frac{\v}{\varphi}),Z_2 \p_y \vr\|_{L^\infty_1(\Omega)}^2
 \int_0^t (\| \u \|_{\H^{m}_{l}}^2+\|\p_y \vr\|_{\H^{m-1}_l}^2) d\tau.
\end{aligned}
\end{equation*}
This and inequality \eqref{3712} give directly
\begin{equation*}\label{3717}
|I_{25}|\le C_{m,l} Q(t)\int_0^t (\| \u \|_{\H^{m}_{l}}^2+\|\p_y \vr\|_{\H^{m-1}_l}^2) d\tau.
\end{equation*}
Therefore, substituting the estimate of $I_{21}$ through $I_{25}$ into \eqref{377}
and using the induction assumption to eliminate the term $\es \int_0^t  \|\p_y^2 \vr\|_{\H^{m-2}_l}^2d\tau$,
then the proof of this lemma is completed.
\end{proof}

Next, we establish the estimate for $\|\p_y \u\|_{\H^{m-1}_l}$.
Although $\p_y \u$ does not vanish on the boundary, we can take $-\p_y^2 \u$
as the text function thanks to the coefficient $\mu>0$ in \eqref{eq5}$_2$.

\begin{lemma}\label{lemma38}
For smooth solution $(\vr, \u, \v, \h,\g)$ of the equations \eqref{eq5}-\eqref{bc5}, then it holds on
\begin{equation*}\label{381}
\begin{aligned}
&\sup_{0\le \tau \le t}\|\p_y \u(\t)\|_{\H^{m-1}_l}^2
+\int_0^t( \es\|\p_{xy} \u\|_{\H^{m-1}_l}^2+\mu \|\p_y^2 \u\|_{\H^{m-1}_l}^2) d\tau\\
&\le \|\p_y \u_0\|_{\H^{m-1}_l}^2
\!+C_\mu \!\!\int_0^t\!\! \|\p_y r_u\|_{\H^{m-1}_l}^2 d\tau
\!+\!C_{\mu, m, l}(1+Q^2(t))\!\!\int_0^t\!\! (1\!+\|(\vr, \u ,\h)\|_{\H^m_l}^2\!+\|(\p_y \u, \p_y \h)\|_{\H^{m-1}_l}^2)d\tau.
\end{aligned}
\end{equation*}
\end{lemma}

\begin{proof}
First of all, multiplying the equation \eqref{eq5}$_2$ by $-\ya \p_y^2 \u$
and integrating over $\Omega$, we find
\begin{equation*}\label{383}
\begin{aligned}
&\int_\O (-\r \p_t \u +\es \p_x^2 \u+\mu \p_y^2 \u)\cdot \ya \p_y^2 \u dxdy\\
=
&\int_\O (\es \p_x r_u+\mu e^{-y})\cdot \ya \p_y^2 \u dxdy-\int_\O f_2 \cdot \ya \p_y^2 \u dxdy,
\end{aligned}
\end{equation*}
where $f_2 $ is defined by
\begin{equation*}\label{384}
f_2:=-\r (\u+1-e^{-y})\p_x \u-\r \v \p_y \u-\r \v e^{-y}+(\h+1)\p_x \h+\g \p_y \h.
\end{equation*}
Integrating by part and applying the boundary condition $\u|_{y=0}=0$, we get
\begin{equation*}\label{385}
\begin{aligned}
&\es \int_\O \p_x^2 \u \cdot \ya  \p_y^2 \u dxdy\\
=
&\es \int_{\T} \p_x^2 \u \cdot   \p_y \u|_{y=0} dx
-\es \int_\O \p_x^2 \p_y \u \cdot \ya  \p_y \u dxdy\\
&-2l\es \int_\O \p_x^2 \u \cdot \yb  \p_y \u dxdy\\
=
&\es \int_\O \ya  |\p_x\p_y \u|^2 dxdy
+2l\es \int_\O \yb \p_x \u \cdot \p_x \p_y  \u dxdy.
\end{aligned}
\end{equation*}
Similarly, we get that
\begin{equation*}\label{386}
\begin{aligned}
-\int \r \p_t \u \cdot \ya \p_y^2 \u dxdy
=
&\frac{1}{2}\frac{d}{dt}\int \ya \r |\p_y \u|^2 dxdy
-\frac{1}{2}\int \ya \p_t \r |\p_y \u|^2 dxdy\\
&+\int \ya \p_y \r \p_t \u \cdot \p_y \u dxdy
+2l \int \yb \r \p_t \u \cdot \p_y \u dxdy.
\end{aligned}
\end{equation*}
Based on the above estimates, we can conclude that
\begin{equation*}\label{387}
\begin{aligned}
&\frac{1}{2}\frac{d}{dt}\int_\O \ya \r |\p_y \u|^2 dxdy
+\es \int_\O \ya  |\p_x\p_y \u|^2 dxdy
+\mu \int_\O \ya  |\p_y^2 \u |^2 dxdy\\
=
&-\int_\O f_2 \cdot \ya \p_y^2 \u dxdy
+\frac{1}{2}\int_\O \ya \p_t \r |\p_y \u|^2 dxdy
-\int_\O \ya \p_y \r \p_t \u \cdot \p_y \u dxdy\\
&-2l \int_\O \yb \r \p_t \u \cdot \p_y \u dxdy
-2l\es \int_\O \yb \p_x \u \cdot \p_x \p_y  \u dxdy,
\end{aligned}
\end{equation*}
which, integrating over $[0, t]$, yields directly
\begin{equation*}\label{388}
\begin{aligned}
&\sup_{\t \in [0, t]}\|\sqrt{\vr}\p_y \u(\t)\|_{L^2_l(\O)}^2
+\es \int_0^t \|\p_{xy}\u\|_{L^2_l(\O)}^2 d\t
+\mu \int_0^t \|\p_{y}^2 \u\|_{L^2_l(\O)}^2 d\t\\
\le
&\|\sqrt{\vr_0}\p_y \u_0\|_{L^2_l(\O)}^2
\!+C_\mu \!\int_0^t \!\|\p_y r_u\|_{L^2_l(\O)}^2 d\t
\!+\!C_{\mu, l}(1+Q(t))\!\int_0^t\!(1+\|(\u, \h)\|_{\H^1_l(\O)}^2\!+\!\|(\p_y \u, \p_y \h)\|_{L^2_l(\O)}^2)d\t.
\end{aligned}
\end{equation*}
This implies the estimate \eqref{381} holds on for $m=1$. To prove the general case, let us assume that \eqref{381}
is proven for $k \le m-2$, we need to prove it also holds on for $k=m-1$.
Applying $\z^\a(|\a|=m-1)$ operator to the second equation of \eqref{eq5}, multiplying the resulting equation
by $-\ya \p_y \z^\a \p_y\u $ and integrating over $\Omega$, we find
\begin{equation}\label{389}
\begin{aligned}
&\int_\O(- \r \p_t \z^\a \u +\es \p_x^2 \z^\a \u+\mu \z^\a \p_y^2 \u)\cdot \ya \p_y \z^\a \p_y \u dxdy\\
&=\int_\O [\z^\a, \rho]\p_t\u \cdot \ya \p_y \z^\a \p_y \u dxdy
+\int_\O \z^\a (-f_2+\es \p_x r_u+\mu e^{-y})\cdot \ya \p_y \z^\a \p_y \u dxdy.
\end{aligned}
\end{equation}
In view of the boundary condition $\u|_{y=0}=0$ and the definition of $\varphi(y)$,
we can justify that $\z^\a \u|_{y=0}=0$.
Then, integrating by part and applying the fact
$\z^\a \u|_{y=0}=0$, one arrives at
\begin{equation*}\label{3810}
\begin{aligned}
&\es \int_\O \p_x^2 \z^\a \u \cdot \ya \p_y \z^\a \p_y \u dxdy\\
=
&\es \int_{\mathbb{T}} \p_x^2 \z^\a \u \cdot \z^\a \p_y \u|_{y=0} dx
-\es \int_\O \ya \p_y \p_x^2 \z^\a \u \cdot \z^\a \p_y \u dxdy\\
&-2l \es \int_\O \yb \p_x^2 \z^\a \u \cdot \z^\a \p_y \u dxdy\\
=
&
\es \int_\O \ya |\p_x \z^\a \p_y \u|^2dxdy
+\es \int_\O \ya [\z^\a, \p_y]\p_x^2 \u \cdot \z^\a \p_y \u dxdy\\
&+2l \es \int_\O \yb \p_x \z^\a \u \cdot \p_x \z^\a \p_y \u dxdy\\
\ge
&\frac{1}{2}\es \int_\O \ya |\p_x \z^\a \p_y \u|^2dxdy
-C(\es \| \p_x \p_y \u\|_{\H^{m-2}_l}^2+ \| \z^\a \p_x \u\|_{L^2_l(\Omega)}^2).
\end{aligned}
\end{equation*}
By virtue of the Cauchy-Schwarz inequality, it is easy to justify that
\begin{equation*}\label{3811}
\begin{aligned}
&\mu \int_\O \z^\a \p_y^2 \u \cdot \ya \p_y \z^\a \p_y \u dxdy\\
=
&\mu \int_\O \ya |\p_y \z^\a \p_y \u|^2 dxdy
+\mu \int_\O \ya [\z^\a, \p_y]\p_y \u \cdot \p_y \z^\a \p_y \u dxdy\\
\ge
&\frac{1}{2}\mu \int_\O \ya |\p_y \z^\a \p_y \u|^2 dxdy
-C\|[\z^\a, \p_y]\p_y \u\|_{L^2_l(\Omega)}^2\\
\ge
&\frac{1}{2}\mu \int_\O \ya |\p_y \z^\a \p_y \u|^2 dxdy
-C\|\p_y^2 \u\|_{\H^{m-2}_l}^2.
\end{aligned}
\end{equation*}
Integrating by part and applying the boundary condition $\z^\a \u|_{y=0}=0$, we get
\begin{equation*}\label{3812}
-\int_\O \r \p_t \z^\a \u \cdot \ya \p_y \z^\a \p_y \u dxdy
=
\frac{1}{2}\frac{d}{dt}\int_\O \ya \r |\z^\a \p_y \u|^2 dxdy
+II.
\end{equation*}
where $II$ is defined by
\begin{equation*}\label{3813}
\begin{aligned}
II=
&-\frac{1}{2} \int_\O \ya \p_t \r  |\z^\a \p_y \u|^2 dxdy
-\int_\O \ya \r  [\z^\a, \p_y]\p_t \u \cdot  \z^\a \p_y \u dxdy\\
&+\int_\O \ya \p_y \r \p_t \z^\a \u \cdot  \z^\a \p_y \u dx dy
+2l \int_\O \yb \r \p_t \z^\a \u \cdot  \z^\a \p_y \u dx dy,
\end{aligned}
\end{equation*}
which can be estimated as follows
\begin{equation*}\label{3814}
|II|\le C_l \|(\p_t \r, \p_y \r)\|_{L^\infty_0(\O)}(\|\u\|_{\H^{m}_l}^2+\|\p_y \u\|_{\H^{m-1}_l}^2).
\end{equation*}
Using the H\"{o}lder and Cauchy inequalities, it follows
\begin{equation*}\label{3815}
\begin{aligned}
&| \int_\O [\z^\a, \r ]\p_t\u \cdot \ya \p_y \z^\a \p_y \u dxdy|\\
&\le
\frac{\mu}{4}\|\ya \p_y \z^\a \p_y \u \|_{L^2_l(\Omega)}^2
+C_\mu\|[\z^\a, \r ]\p_t\u\|_{L^2_l(\Omega)}^2,
\end{aligned}
\end{equation*}
and
\begin{equation*}\label{3816}
\begin{aligned}
&|\int_\O \z^\a (f_2+\es \p_x r_u+\mu e^{-y})\cdot \ya \p_y \z^\a \p_y \u dxdy|\\
\le
&\frac{\mu}{4}\|\ya \p_y \z^\a \p_y \u \|_{L^2_l(\Omega)}^2
+C_\mu( 1+\|\z^\a r_u\|_{L^2_l(\Omega)}^2+\|\z^\a f_2\|_{L^2_l(\Omega)}^2).
\end{aligned}
\end{equation*}
Substituting the above estimates into \eqref{389},
and integrating the inequality over $[0, t]$, we get
\begin{equation*}\label{3817}
\begin{aligned}
&\int_\O \ya \r |\z^\a \p_y \u|^2 dxdy
+\int_0^t \int_\O \ya (\es|\p_x \z^\a \p_y \u|^2+\mu |\p_y \z^\a \p_y \u|^2)dxdy d\tau\\
&\le \int_\O \ya \r_0 |\z^\a \p_y \u_0|^2 dxdy
+C \es \int_0^t  \| \p_x \p_y \u\|_{\H^{m-2}_l}^2 d\tau
+C \mu \int_0^t \| \p_y^2 \u\|_{\H^{m-2}_l}^2 d\tau\\
&
+C_\mu \int_0^t \|[\z^\a, \r ]\p_t\u \|_{L^2_l(\Omega)}^2 d\tau
+C_\mu \int_0^t ( 1+\|\z^\a f_2\|_{L^2_l(\Omega)}^2+\|\z^\a r_u\|_{L^2_l(\Omega)}^2)d\tau\\
&+C_l (1+\|(\p_t \r, \p_y \r)\|_{L^\infty_0(\O)})
\int_0^t (\|\u\|_{\H^{m}_l}^2+\|\partial_y \u\|_{\H^{m-1}_l}^2)d\tau.
\end{aligned}
\end{equation*}
By virtue of the Cauchy and Morse type inequality \eqref{ineq-moser}, we get
\begin{equation*}\label{3818}
\begin{aligned}
&|\int_0^t \int_\O [\z^\a, \r ]\p_t\u \cdot \ya \p_y \z^\a \p_y \u dxdy d\tau|\\
\le
& \frac{\mu}{4}\int_0^t \|\p_y \z^\a \p_y \u\|_{L^2_l(\Omega)}^2d\tau
+C_\mu \int_0^t \|[\z^\a, \r ]\p_t\u\|_{L^2_l(\Omega)}^2d\tau\\
\le
& \frac{\mu}{4}\int_0^t \|\p_y \z^\a \p_y \u\|_{L^2_l(\Omega)}^2d\tau
+C_\mu \|(\z^{E_i}\vr, \p_t \u)\|_{L^\infty_0(\O)}^2
\int_0^t \|(\vr, \u)\|_{\H^m_l}^2 d\tau.
\end{aligned}
\end{equation*}
Similarly, by routine checking, we may show
\begin{equation*}\label{3819}
\int_0^t \|\z^\a f_2\|_{L^2_l(\Omega)}^2 d\tau
\le C(1+Q^2(t))\int_0^t (\|(\vr, \u ,\h)\|_{\H^m_l}^2
   +\|(\p_y \u, \p_y \h)\|_{\H^{m-1}_l}^2)d\tau.
\end{equation*}
Thus, it is easy to justify the following estimate for $|\a|=m-1$
\begin{equation}\label{3820}
\begin{aligned}
&\sup_{0\le \tau \le t}\int_\O \ya \r |\z^\a \p_y \u|^2 dxdy
+\int_0^t \int_\O \ya (\es|\p_x \z^\a \p_y \u|^2+\mu |\p_y \z^\a \p_y \u|^2)dxdy d\tau\\
&\le \int_\O \ya \r_0 |\z^\a \p_y \u_0|^2 dxdy
+C \es \int_0^t  \| \p_x \p_y \u\|_{\H^{m-2}_l}^2 d\tau
+C \mu \int_0^t \| \p_y^2 \u\|_{\H^{m-2}_l}^2 d\tau\\
&
+C_\mu \int_0^t \|\p_y r_u\|_{\H^{m-1}_l}^2 d\tau
+C_{\mu, m, l}(1+Q^2(t))\int_0^t (\|(\vr, \u ,\h)\|_{\H^m_l}^2
   +\|(\p_y \u, \p_y \h)\|_{\H^{m-1}_l}^2)d\tau.
\end{aligned}
\end{equation}
Since the terms $\es \int_0^t  \| \p_x \p_y \u\|_{\H^{m-2}_l}^2 d\tau$
and $\mu \int_0^t \| \p_y^2 \u\|_{\H^{m-2}_l}^2 d\tau$ in \eqref{3820}
can be obtained by induction, we complete the proof of this lemma.
\end{proof}

Similarly, we can obtain the following estimates for the quantity $\|\p_y \h\|_{\H^{m-1}_l}$.

\begin{lemma}\label{lemma39}
For smooth solution $(\vr, \u, \v, \h,\g)$ of the equations \eqref{eq5}-\eqref{bc5}, then it holds on
\begin{equation*}\label{391}
\begin{aligned}
&\sup_{0 \le \tau \le t}\|\p_y \h\|_{\H^{m-1}_l}^2
+\es \int_0^t \|\p_{xy} \h\|_{\H^{m-1}_l}^2 d\tau
+\k \int_0^t \|\p_y^2 \h\|_{\H^{m-1}_l}^2 d\tau\\
&\le \|\p_y \h_0\|_{\H^{m-1}_l}^2
 +C_\k \int_0^t \|\p_y r_h\|_{\H^{m-1}_l}^2 d\tau
 +C_{\k, m, l}(1+Q(t))\int_0^t (\|(\u, \h)\|_{\H^m_l}^2
  +\|(\p_y \u, \p_y \h)\|_{\H^{m-1}_l}^2)d\tau.
\end{aligned}
\end{equation*}
\end{lemma}

Finally, we give the proof for the estimate in  Proposition \ref{Normal-estimate}.
Indeed, we recall the estimate(see \eqref{b22} in appendix \ref{appendixB}) as follows
\begin{equation*}
\|Z_\t^{\a_1}(\vr, \u, \h)\|_{L^2_l(\O)}^2
\le C_l \d^{-2}(1+\|\p_y (\vr, \u, \h)\|_{L^\infty_1(\O)}^2)\|(\vr_m, \u_m, \h_m)\|_{L^2_l(\O)}^2,
\end{equation*}
which, together with the estimates in Lemmas \ref{lemma37}, \ref{lemma38} and \ref{lemma39},
completes the proof of estimate in  Proposition \ref{Normal-estimate}.

\subsection{$L^\infty-$Estimates}

To close the estimate, we need to control the $L^\infty-$norm of $(\r, \u, \v, \h, \g)$ in $Q(t)$.
Then, we have

\begin{proposition}\label{Infinity-estimate}
Let $(\vr, \u, \v, \h, \g)$ be sufficiently smooth solution, defined on $[0, T^\es]$,
to the equations \eqref{eq5}-\eqref{bc5}, then we have the following estimates:
\begin{equation*}
Q(t)\le C(1+\|(\rho_0, u_{10}, h_{10})\|_{\overline{\mathcal{B}}^m_l}
           +t \|(\rho_0, u_{10}, h_{10})\|_{\widehat{\mathcal{B}}^m_l}+X_{m,l}^3(t)),
\end{equation*}
and
\begin{equation*}
\begin{aligned}
\|\p_y \vr(t)\|_{\H^{1,\infty}_1}^2
\le
&C(\|\p_y \vr_0\|_{\H^{1,\infty}_1}^2+\|(\vr_0, \u_0, \h_0)\|_{\H^3_0}^2)
 +\!C t \|(\rho_0, u_{10}, h_{10})\|_{\widehat{\mathcal{B}}^m_l}^2
 +\! C(1+Q(t))\!\int_0^t\! X_{m,l}^6(\t) d\t,
\end{aligned}
\end{equation*}
for $m \ge 5, l \ge 2$.
\end{proposition}

We point out that the Proposition will be proved in Lemmas \ref{Lemma310}
and \ref{Lemma311}.
First of all, due to the coefficients $\mu>0$ and $\k>0$,
we can apply the Sobolev inequality, and equations \eqref{eq5} to establish
the estimates as follows.\\

\begin{lemma}\label{Lemma310}
Let $(\vr, \u, \v, \h, \g)$ be sufficiently smooth solution, defined on $[0, T^\es]$,
to the equations \eqref{eq5}-\eqref{bc5}, then we have the following estimates:
\begin{equation}\label{3101}
\|Z_\t \vr(t)\|_{L^\infty_0(\O)} +\|(\u, \h)(t)\|_{\H^{1,\infty}_{0, tan}}
\le C(\e_{3,0}^{\frac{1}{2}}(t)+\|\p_y(\vr, \u, \h)(t)\|_{\H^2_0}),
\end{equation}
\begin{equation}\label{3102}
\|(\v, \g)(t)\|_{\H^{1, \infty}_{1, tan}} \le C \e_{4,2}^\frac{1}{2}(t),
\end{equation}
\begin{equation}\label{3104}
\|(\frac{\v}{\varphi})(t)\|_{\H^{1,\infty}_1}
\le C(\e_{4,2}^{\frac{1}{2}}(t)+\|\p_y \u(t)\|_{\H^3_2}),
\end{equation}
\begin{equation}\label{3103}
\begin{aligned}
\|(\p_y \u, \p_y \h)(t)\|_{\H^{1,\infty}_1}
\le
&C(1+\|(\rho_0, u_{10}, h_{10})\|_{\overline{\mathcal{B}}^m_l}^{\frac{1}{2}}
         +t \|(\rho_0, u_{10}, h_{10})\|_{\widehat{\mathcal{B}}^m_l}^{\frac{1}{2}})\\
&        +C(\e_{5,1}^{\frac{3}{2}}(t)+\|\p_y(\vr, \u, \h)(t)\|_{\H^3_1}^3), \quad m \ge 5,\ l \ge 1.
\end{aligned}
\end{equation}

\end{lemma}

\begin{proof}
By virtue of the Sobolev inequality \eqref{sobolev}
and the definition of $\mathcal{E}_{m, l}(t)$(see \eqref{eml}), then we get
\begin{equation*}\label{3105}
\|\u\|_{L^\infty_0(\O)}
\le C(\|\u \|_{L^2_0(\O)} +\|\p_x \u \|_{L^2_0(\O)}
+\|\p_y \u \|_{L^2_0(\O)} +\|\p_{xy} \u\|_{L^2_0(\O)})
\le C(\e_{2,0}^{\frac{1}{2}}(t) +\|\p_y \u(t)\|_{\H^1_0}).
\end{equation*}
Similarly, it is easy to justify
\begin{equation*}\label{3106}
\|\h\|_{L^\infty_0(\O)} \le C(\e_{2,0}^{\frac{1}{2}}(t)+\|\p_y \h\|_{\H^1_0}),
\quad
\|Z_\t^{e_i}(\vr, \u, \h)\|_{L^\infty_0(\O)}
\le C(\e_{3,0}^{\frac{1}{2}}(t)+\|\p_y(\vr, \u, \h)\|_{\H^2_0}).
\end{equation*}
Thus we obtain the estimate \eqref{3101}.
Using the Hardy inequality and Sobolev inequality \eqref{sobolev}, we get
\begin{equation}\label{3107}
\begin{aligned}
\|\v\|_{L^\infty_1(\O)}
&\le C(\|\v \|_{L^2_1(\O)}+\|\p_x \v \|_{L^2_1(\O)}
+\|\p_y \v \|_{L^2_1(\O)}+\|\p_{xy} \v\|_{L^2_1(\O)})\\
&\le C(\|\p_y \v \|_{L^2_2(\O)}+\|\p_{xy} \v \|_{L^2_2(\O)}
+\|\p_{x} \u \|_{L^2_1(\O)}+\|\p_{xx} \u\|_{L^2_1(\O)})\\
&\le C\e_{3,2}^\frac{1}{2}(t),
\end{aligned}
\end{equation}
where we have used the divergence-free condition in the last inequality.
Similarly, we obtain
\begin{equation}\label{3108}
\|\g\|_{L^\infty_1(\O)}^2  \le C \e_{3,2}^\frac{1}{2}(t),
\quad
\|Z_\t^{e_i}\v\|_{L^\infty_1(\O)}
+\|Z_\t^{e_i}\g\|_{L^\infty_1(\O)}
\le C\e_{4,2}^\frac{1}{2}(t).
\end{equation}
Then, the combination of estimates \eqref{3107} and \eqref{3108} yields the estimate \eqref{3102}.
By virtue of the Sobolev inequality \eqref{sobolev}, we get
\begin{equation*}\label{3109}
\|\p_y \u\|_{L^\infty_1(\O)}
\le C(\|\p_y \u\|_{L^2_1(\O)}+\|\p_{xy} \u\|_{L^2_1(\O)}
         +\|\p_{yy} \u\|_{L^2_1(\O)}+\|\p_{xyy} \u\|_{L^2_1(\O)}).
\end{equation*}
In view of the equation \eqref{eq5}$_2$ and the estimate \eqref{3108}
with the weight $0$ instead of $1$, we find
\begin{equation*}\label{31010}
\begin{aligned}
\|\p_y^2 \u\|_{L^2_1(\O)} \le
& C(1+\|\p_x r_u\|_{L^2_1(\O)} +\|\es \p_x^2 \u\|_{L^2_1(\O)})
+C(1+\|(\u, \h)\|_{L^\infty_0(\O)} )\|(\p_x \u, \p_x \h)\|_{L^2_1(\O)} \\
&+C(\|\p_t \u\|_{L^2_1(\O)} +\|(\v, \g)\|_{L^\infty_0(\O)}) (1+\|(\p_y \u, \p_y \h)\|_{L^2_1(\O)})\\
\le
&C(1+ \|\p_x r_u\|_{L^2_1(\O)} +\e_{3,1}(t)+\|\p_y(\u, \h)\|_{\H^1_1}^2).
\end{aligned}
\end{equation*}
Similarly, it is easy to justify
\begin{equation*}\label{31011}
\|\p_{xyy} \u\|_{L^2_1(\O)}
\le C(1+ \|\p_x^2 r_u\|_{L^2_1(\O)}+\e_{4,1}^{\frac{3}{2}}(t)+\|\p_y(\vr, \u, \h)\|_{\H^2_1}^3).
\end{equation*}
Thus we can conclude the estimate
\begin{equation}\label{31012-1}
\|\p_y \u\|_{L^\infty_1(\O)} \le C(1+\|(\p_x  r_u, \p_x^2 r_u)\|_{L^2_1(\O)}
+\e_{4,1}^{\frac{3}{2}}(t)+\|\p_y(\vr, \u, \h)\|_{\H^2_1}^3).
\end{equation}
By virtue of the definition of $r_u$ in \eqref{rdef}, we get for $m\ge 4, l \ge 1$ that
\begin{equation*}
\|(\p_x r_u, \p_x^2 r_u)\|_{L^2_1(\O)}
\le \|(\rho_0, u_{10}, h_{10})\|_{\mathcal{\mathcal{B}}^m_l}^{\frac{1}{2}}
+C t \|(\rho_0, u_{10}, h_{10})\|_{\widehat{\mathcal{B}}^m_l}^{\frac{1}{2}},
\end{equation*}
which, together with the estimate \eqref{31012-1}, yields directly
\begin{equation}\label{31012a}
\|\p_y \u\|_{L^\infty_1(\O)}
 \le C(1+\|(\rho_0, u_{10}, h_{10})\|_{\overline{\mathcal{B}}^m_l}^{\frac{1}{2}}
 +C t \|(\rho_0, u_{10}, h_{10})\|_{\widehat{\mathcal{B}}^m_l}^{\frac{1}{2}}
 +\e_{4,1}^{\frac{3}{2}}(t)+\|\p_y(\vr, \u, \h)\|_{\H^2_1}^3).
\end{equation}
Similarly, we get for $m \ge 4, l \ge 1$
\begin{equation}\label{31012-a}
\begin{aligned}
\|\p_y \h\|_{L^\infty_1(\O)} \le
 C(1+\|(\rho_0, u_{10}, h_{10})\|_{\overline{\mathcal{B}}^m_l}^{\frac{1}{2}}
 +C t \|(\rho_0, u_{10}, h_{10})\|_{\widehat{\mathcal{B}}^m_l}^{\frac{1}{2}}
 +\e_{4,1}(t)+\|\p_y(\u, \h)\|_{\H^2_1}^2),
\end{aligned}
\end{equation}
and for $m \ge 5, l \ge 1$
\begin{equation}\label{31012}
\begin{aligned}
&\!\|\z^{E_i} \p_y \u\|_{L^\infty_1(\O)}
  \le\! C(1+\|(\rho_0, u_{10}, h_{10})\|_{\overline{\mathcal{B}}^m_l}^{\frac{1}{2}}
       \!+C t \|(\rho_0, u_{10}, h_{10})\|_{\widehat{\mathcal{B}}^m_l}^{\frac{1}{2}}
       \!+\e_{5,1}^{\frac{3}{2}}(t)\!+\!\|\p_y(\vr, \u, \h)\|_{\H^3_1}^3),\\
& \!\|\z^{E_i} \p_y \h\|_{L^\infty_1(\O)}
  \!\le C(1+\|(\rho_0, u_{10}, h_{10})\|_{\overline{\mathcal{B}}^m_l}^{\frac{1}{2}}
  \!+C t \|(\rho_0, u_{10}, h_{10})\|_{\widehat{\mathcal{B}}^m_l}^{\frac{1}{2}}
  \!+\e_{5,1}(t)\!+\|\p_y(\u, \h)\|_{\H^3_1}^2).
\end{aligned}
\end{equation}
Then, the combination of estimates \eqref{31012a}, \eqref{31012-a} and \eqref{31012} yields \eqref{3103}.

Finally, we give the estimate for the quantity $\|\frac{\v}{\varphi}\|_{\H^{1,\infty}_1}$.
Since vertical velocity $\v$ vanishes on the boundary(i.e., $\v|_{y=0}=0$), we get
$
\la y \ra^2|\v|\le C y(\|\v\|_{L^\infty_1(\O)}+\|\p_y \v\|_{L^\infty_2(\O)}).
$
Using Sobolev inequality \eqref{sobolev} and divergence-free condition \eqref{eq5}$_4$, it follows
\begin{equation}\label{31013}
\|\frac{\v}{\varphi}\|_{L^\infty_1(\O)}
\le C(\|\v\|_{L^\infty_1(\O)} +\|\p_y \v\|_{L^\infty_2(\O)})
\le C(\e_{3,2}^{\frac{1}{2}}(t)+\|\p_y \u\|_{\H^2_2}).
\end{equation}
Similarly, we also get that
\begin{equation}\label{31014}
\|Z_\t^{e_i}(\frac{\v}{\varphi})\|_{L^\infty_1(\O)}^2 \le C(\e_{4,2}^{\frac{1}{2}}(t)+\|\p_y \u\|_{\H^3_2}).
\end{equation}
By virtue of the fact $\p_y(\frac{1}{\varphi})=-\frac{1}{y^2}$,
we get after using the divergence-free condition \eqref{eq5}$_4$
\begin{equation*}\label{31015}
\begin{aligned}
\|Z_2 (\frac{\v}{\varphi})\|_{L^\infty_1(\O)}
&\le C(\|y \p_y(\frac{1}{\varphi})\v\|_{L^\infty_0(\O)}+\|\p_y \v\|_{L^\infty_1(\O)})\\
&\le C(\|\frac{\v}{y}\|_{L^\infty_0(\O)}+\|\p_x \u\|_{L^\infty_1(\O)})
 \le C\|\p_x \u\|_{L^\infty_1(\O)},
\end{aligned}
\end{equation*}
where we have used the fact $|\v|\le y \|\p_y \v\|_{L^\infty_0(\O)}$ in the last inequality.
Using the above inequality and Sobolev inequality \eqref{sobolev}, we conclude
\begin{equation*}
\|\la y \ra Z_2 (\frac{\v}{\varphi})\|_{L^\infty_0(\O)} \le C(\e_{3, 1}^{\frac{1}{2}}(t)+\|\p_y \u\|_{\H^2_1}),
\end{equation*}
which, together with the estimates \eqref{31013} and \eqref{31014}, yields directly
\begin{equation*}
\|\frac{\v}{\varphi}\|_{\H^{1,\infty}_1}\le C(\e_{4,2}^{\frac{1}{2}}(t)+\|\p_y \u\|_{\H^3_2}).
\end{equation*}
Therefore, we complete the proof of Lemma \ref{Lemma310}.
\end{proof}

By virtue of the estimates \eqref{3101}-\eqref{3104} in Lemma \ref{Lemma310},
then $Q(t)$ can be controlled as follows:
\begin{equation*}
Q(t)\le C(1+\|(\rho_0, u_{10}, h_{10})\|_{\overline{\mathcal{B}}^m_l}
          + t \|(\rho_0, u_{10}, h_{10})\|_{\widehat{\mathcal{B}}^m_l}+X_{m,l}^3(t))
\end{equation*}
for $m \ge 5, l \ge 1$.
To close the estimate, we still need to establish the estimate for the quantity $\|\p_y \vr(t)\|_{\H^{1,\infty}_1}^2$.
Since the quantity $y\p_y^2 \vr$ does not communicate with the diffusive term, this prevents
us to apply the maximum principle of transport-diffusion equation.
We should point out that Masmoudi and Rousset \cite{Masmoudi-Rousset} have applied
some estimates of one dimensional Fokker-Planck type equation to achieve this target.
However, we can only apply the $L^\infty-$estimate of heat equation(cf. \ref{eheat} in Lemma \ref{A-heat})
to achieve this goal since $\p_y \v$ vanishes on the boundary due to $\u|_{y=0}=0$ and divergence-free condition.

\begin{lemma}\label{Lemma311}
Let $(\vr, \u, \v, \h, \g)$ be sufficiently smooth solution, defined on $[0, T^\es]$,
to the equations \eqref{eq5}. Then, it holds on
\begin{equation*}\label{31101}
\begin{aligned}
\|\p_y \vr(t)\|_{\H^{1,\infty}_1}^2
\le
&C(\|\p_y \vr_0\|_{\H^{1,\infty}_1}^2+\|(\vr_0, \u_0, \h_0)\|_{\H^3_0}^2)
 +C t \|(\rho_0, u_{10}, h_{10})\|_{\widehat{\mathcal{B}}^m_l}^2
 +C(1+Q(t))\int_0^t X_{m,l}^6(\t) d\t,
\end{aligned}
\end{equation*}
for $m \ge 5, l \ge 2$.
\end{lemma}

\begin{proof}
By virtue of the Sobolev inequality \eqref{sobolev}, it is easy to justify that
\begin{equation}\label{31102}
\begin{aligned}
\|\p_y \vr\|_{\H^{1,\infty}_1}^2
\le
&\|\p_y \vr\|_{L^\infty_0(\O)}^2+\|Z_2 \vr\|_{L^\infty_1(\O)}^2
+\|Z_\t^{e_i}\p_y \vr\|_{L^\infty_0(\O)}^2+\|y\p_{yy}\vr\|_{L^\infty_0(\O)}^2\\
\le
&C(\|(\p_y \vr, Z_\t^{e_i}\p_y \vr, y\p_{yy}\vr)\|_{L^\infty_0(\O)}^2
   +\e_{3,1}(t)+\|\p_y \vr(t)\|_{\H^3_1}^2).
\end{aligned}
\end{equation}

First of all, since the quantity $\p_y \vr$ satisfies the evolution equation
\eqref{3c1}, we may apply the maximum principle of transport-diffusion equation
\eqref{3c1} to get
\begin{equation*}\label{31103}
\|\p_y \vr\|_{L^\infty_0(\O)}
\le \|\p_y \vr_0\|_{L^\infty_0(\O)}+\int_0^t \|f_1\|_{L^\infty_0(\O)}d\t,
\end{equation*}
where $f_1$ is defined in \eqref{3c2}.
Thus we apply the Cauchy-Schwarz inequality to get
\begin{equation}\label{31104}
\begin{aligned}
\|\p_y \vr\|_{L^\infty_0(\O)}^2
\le
& 2\|\p_y \vr_0\|_{L^\infty_0(\O)}^2+2t \int_0^t \|f_1\|_{L^\infty_0(\O)}^2d\t\\
\le
&2\|\p_y \vr_0\|_{L^\infty_0(\O)}^2+2 t\int_0^t \|\p_y(\p_x r_1, \p_y r_2)\|_{L^\infty_0(\O)}^2 d\t\\
&+Ct\int_0^t (1+\e_{4,0}^2(\t)+\|\p_y(\vr, \u, \h)\|_{\H^2_0}^4+\|\p_y \vr\|_{L^\infty_0(\O)}^4) d\t.
\end{aligned}
\end{equation}

Next, Applying the tangential differential derivatives $Z_\t^{e_i}(i=1,2, e_1=(1, 0), e_2=(0, 1))$
on the evolution equation \eqref{3c1}, we get the evolution for $Z_\t^{e_i}\p_y \vr$:
\begin{equation*}\label{31105}
(\p_t +(\u+1-e^{-y})\p_x +\v \p_y -\es \p_x^2-\es \p_y^2)Z_\t^{e_i}\p_y \vr
=Z_\t^{e_i}f_1-Z_\t^{e_i}\u \p_{xy}\vr-Z_\t^{e_i} \v \p_y^2 \vr,
\end{equation*}
and hence it follows from the maximum principle and Cauchy-Schwarz inequality that
\begin{equation*}\label{31106}
\|Z_\t^{e_i}\p_y \vr\|_{L^\infty_0(\O)}^2 \le 2\|Z_\t^{e_i}\p_y \vr_0\|_{L^\infty_0(\O)}^2
+2t\int_0^t \|(Z_\t^{e_i} f_1, Z_\t^{e_i}\u \p_{xy}\vr, Z_\t^{e_i} \v \p_y^2 \vr)\|_{L^\infty_0(\O)}^2 d\t.
\end{equation*}
Since the vertical velocity $\v$ vanishes on the boundary(i.e., $\v|_{y=0}=0$), we conclude
\begin{equation*}\label{31107}
\begin{aligned}
&\|Z_\t^{e_i}\u \p_{xy}\vr\|_{L^\infty_0(\O)}^2+\|Z_\t^{e_i} \v \p_y^2 \vr\|_{L^\infty_0(\O)}^2\\
\le
&\|Z_\t^{e_i}\u\|_{L^\infty_0(\O)}^2 \|\p_{xy}\vr\|_{L^\infty_0(\O)}^2
    +\|\frac{Z_\t^{e_i} \v}{\varphi}\|_{L^\infty_0(\O)}^2\|Z_2 \p_y \vr\|_{L^\infty_0(\O)}^2\\
\le
&C(\e_{4,1}^2(t)+\|\p_y \u\|_{\H^3_1}^4+\|\p_y \vr\|_{\H^{1,\infty}_0}^4).
\end{aligned}
\end{equation*}
Thus we obtain the estimate
\begin{equation}\label{31108}
\begin{aligned}
\|Z_\t^{e_i}\p_y \vr\|_{L^\infty_0(\O)}^2
\le
&2\|Z_\t^{e_i}\p_y \vr_0\|_{L^\infty_0(\O)}^2
 +2t\int_0^t(\|Z_\t^{e_i} \p_y(\p_x r_1, \p_y r_2)\|_{L^\infty_0(\O)}^2+\|r_u\|_{\H^3_0}^4)d\t\\
&+C t\int_0^t (1+\e_{5,1}^2(\t)+\|\p_y(\vr, \u, \h)\|_{\H^3_1}^4+\|\p_y \vr\|_{\H^{1,\infty}_0}^4)d\t.
\end{aligned}
\end{equation}

Finally, we deal with the term $\|y \p_y^2 \vr\|_{L^\infty_0(\O)}$.
The main difficulty is the estimate of $y \p_y^2 \vr$,
since the communicator of this quantity with the Laplacian involves
two derivatives in the normal variable.
Let $\chi(y)$ be a smooth compactly supported function which takes
the value one in the vicinity of $0$ and is supported in $[0, 1]$,
and hence, we get
\begin{equation*}
\p_y \vr=\chi(y)\p_y \vr+(1-\chi(y))\p_y \vr \triangleq \varrho^b+\varrho^{int},
\end{equation*}
where $\varrho^b$ is compactly supported in $y$ and $\varrho^{int}$ is supported away from the boundary.

Since $H^m_{co}$ norm is equivalent to the usual $H^m$ norm if
the function is support away from the boundary, we apply
the Sobolev inequality \eqref{sobolev} to get
\begin{equation}\label{31109}
\|y \p_y \varrho^{int}\|_{L^\infty_0(\O)}\le C\|\p_y \vr\|_{\H^3_1}.
\end{equation}
On the other hand, due to the equation \eqref{3c1}, we can get the evolution equation for $\varrho^b$:
\begin{equation}\label{311010}
(\p_t-\es \p_y^2)\varrho^b=\es \chi \p_x^2\vr  +\chi f_1+R_1+R_2
\end{equation}
where $R_i(i=1,2)$ are defined by
\begin{equation*}
R_1=-\es \chi'' \p_y \vr-2\es \chi' \p_y^2 \vr,
\quad
R_2=-\chi(\u+1-e^{-y})\p_{xy} \vr-\chi \v \p_y^2 \vr.
\end{equation*}
Applying the estimate \eqref{eheat} to the equation \eqref{311010}, it follows
\begin{equation}\label{311011}
\begin{aligned}
\|y\p_y \varrho^b\|_{L^\infty_0(\O)}^2
\le
& C(\|\varrho^b_0\|_{L^\infty_0(\O)}^2+\|y\p_y \varrho^b_0\|_{L^\infty_0(\O)}^2)
    +C \es^2\int_0^t \|(\chi \p_x^2\vr, y \p_y(\chi\p_x^2\vr))\|_{L^\infty_0(\O)}^2 d\t\\
&   +C\int_0^t \|(\chi f_1, y\p_y(\chi f_1), R_1, y\p_y R_1, R_2, y\p_y R_2) \|_{L^\infty_0(\O)}^2 d\t.
\end{aligned}
\end{equation}
In view of the definition of $\chi$ and the Sobolev inequality, we find
\begin{equation}\label{311012}
|\es^2\int_0^t \|(\chi \p_x^2\vr, y \p_y(\chi\p_x^2\vr))\|_{L^\infty_0(\O)}^2 d\t|
\le C\es^2 \int_0^t \|\p_y \vr\|_{\H^4_0}^2 d\t+C\es^2 \int_0^t \|\vr\|_{\H^4_0}^2 d\t,
\end{equation}
and
\begin{equation}\label{311014}
\|R_1\|_{L^\infty_0(\O)}^2+\|y\p_y R_1\|_{L^\infty_0(\O)}^2\le C \|\vr\|_{\H^5_0}^2.
\end{equation}
Using the $L^\infty-$estimates in Lemma \ref{Lemma310}, we conclude
\begin{equation}\label{311013}
\begin{aligned}
\|(\chi f_1,y\p_y(\chi f_1))\|_{L^\infty_0(\O)}^2
\le
&\|(\p_{xy} r_1, \p_{y}^2 r_2, Z_2 \p_{xy} r_1, Z_2 \p_{y}^2 r_2)\|_{L^\infty_0(\O)}^2
 +\|r_u\|_{\H^3_0}^4\\
&+C(1+\e_{5,0}^2+\|\p_y(\vr, \u, \h)\|_{\H^3_0}^4+\|\p_y \vr\|_{\H^{1,\infty}_0}^4).
\end{aligned}
\end{equation}
By virtue of $\v|_{y=0}=0$, we may apply the Taylor formula
and divergence-free condition \eqref{eq5}$_4$ to get
\begin{equation*}\label{311015}
\|\chi \v \p_{y}^2 \vr\|_{L^\infty_0(\O)}
\le  \|\p_y  \v\|_{L^\infty_0(\O)} \|\chi y \p_{y}^2 \vr\|_{L^\infty_0(\O)}
\le C \|\p_x  \u\|_{L^\infty_0(\O)} \|Z_2 \p_y \vr\|_{L^\infty_0(\O)},
\end{equation*}
and along with the Sobolev inequality \eqref{sobolev} yields directly
\begin{equation}\label{311016}
\begin{aligned}
\|R_2\|_{L^\infty_0(\O)}^2
&\le C(1+\|\u\|_{L^\infty_0(\O)}^2)\|\p_{xy} \vr\|_{L^\infty_0(\O)}^2
+ C \|\p_x  \u\|_{L^\infty_0(\O)}^2 \|Z_2 \p_y \vr\|_{L^\infty_0(\O)}^2\\
&\le C(1+\e_{3,0}^2(t)+\|\p_y \u\|_{\H^2_0}^4+\|\p_y \vr\|_{\H^{1,\infty}_0}^4).
\end{aligned}
\end{equation}
By routine checking, we may check that
\begin{equation*}\label{311017}
\begin{aligned}
y\p_y R_2
&=\chi' y [(\u+1-e^{-y})\p_{xy} \vr+ \v \p_y^2 \vr]
+\chi y[\p_y(\u+1-e^{-y})\p_{xy} \vr+ \p_y \v \p_y^2 \vr]\\
&\quad+\chi (\u+1-e^{-y})y\p_{xyy} \vr+ \chi \v y \p_y^3 \vr.
\end{aligned}
\end{equation*}
By virtue of the definition $\chi$, it follows
\begin{equation}\label{311018}
\begin{aligned}
\|\chi' y [(\u+1-e^{-y})\p_{xy} \vr+ \v \p_y^2 \vr]\|_{L^\infty_0(\O)}^2
\le
C(1+\e_{3,1}^2(t)+\|\p_y \u\|_{\H^1_0}^4+\|\p_y \vr\|_{\H^{1,\infty}_0}^4),
\end{aligned}
\end{equation}
and
\begin{equation}\label{311019}
\begin{aligned}
&\|\chi y[\p_y(\u+1-e^{-y})\p_{xy} \vr+ \p_y \v \p_y^2 \vr]\|_{L^\infty_0(\O)}^2\\
&\le C(1+\es^4\|r_u\|_{\H^{2}_0}^4+\e^6_{4,0}(t)+\|\p_y(\vr, \u, \h)\|_{\H^2_0}^{12}
+\|\p_y \vr\|_{\H^{1,\infty}_0}^4).
\end{aligned}
\end{equation}
Since the velocity $\u$ vanishes on the boundary, the application of Taylor formula yields immediately
\begin{equation}\label{311020}
\begin{aligned}
\|\chi (\u+1-e^{-y})y\p_{xyy} \vr\|_{L^\infty_0(\O)}^2
&\le (1+\|\p_y  \u\|_{L^\infty_0(\O)}^2)\|\chi y^2 \p_{xyy} \vr\|_{L^\infty_0(\O)}^2\\
&\le C(1+\|\p_y  \u\|_{L^\infty_0(\O)}^2)\|\varphi (y)Z_2 \p_{xy} \vr\|_{L^\infty_0(\O)}^2,
\end{aligned}
\end{equation}
where we have used the fact that $y$ is equivalent to $\frac{y}{1+y}$
if $y\in [0, c_0]$.
Using the fact $\u|_{y=0}=0$ and $\p_x \u+\p_y \v=0$, we have $\p_y \v|_{y=0}=0$,
and hence, the Taylor formula implies for $\xi \in [0, y]$
\begin{equation*}\label{311021}
\v(t,x,y)=\v(t,x,0)+y\p_y \v(t,x,0)+\frac{1}{2}y^2 \p_y^2 \v(t,x, \xi)
=\frac{1}{2}y^2 \p_y^2 \v(t,x, \xi),
\end{equation*}
where we have used the fact $\v|_{y=0}=\p_y\v|_{y=0}=0$. Thus, it follows
\begin{equation}\label{311022}
\|\chi \v y \p_y^3 \vr\|_{L^\infty_0(\O)}^2
\le C\|\p_y^2 \v\|_{L^\infty_0(\O)}^2\|\chi y^3 \p_y^3 \vr\|_{L^\infty_0(\O)}^2
\le C\|\p_{xy}\u\|_{L^\infty_0(\O)}^2\|\varphi(y)Z_2^2 \p_y \vr\|_{L^\infty_0(\O)}^2.
\end{equation}
Then the combination of estimates \eqref{311020}, \eqref{311022}
and Sobolev inequality \eqref{sobolev} yields directly
\begin{equation*}
\|\chi (\u+1-e^{-y})y\p_{xyy} \vr+ \chi \v y \p_y^3 \vr\|_{L^\infty_0(\O)}^2
\le C(1+\|r_u\|_{\H^2_0}^4+\e_{5,0}^6(t)+\|\p_y(\vr, \u, \h)\|_{\H^4_0}^{12}),
\end{equation*}
which, together with the estimates \eqref{311018} and \eqref{311019}, yields directly
\begin{equation}\label{311023}
\|y\p_y R_2\|_{L^\infty_0(\O)}^2
\le C(1+\|r_u\|_{\H^3_0}^4+\e_{5,1}^6(t)+\|\p_y(\vr, \u, \h)\|_{\H^4_0}^{12}+\|\p_y \vr\|_{\H^{1,\infty}_0}^4).
\end{equation}
Then, we can get from the estimates \eqref{311012}, \eqref{311013}, \eqref{311014}, \eqref{311016}
and \eqref{311023} that
\begin{equation*}
\begin{aligned}
\|y\p_y \varrho^b\|_{L^\infty_0(\O)}^2
\le
&C(\|\varrho^b_0\|_{L^\infty_0(\O)}^2+\|y\p_y \varrho^b_0\|_{L^\infty_0(\O)}^2)
 +C\int_0^t \|(\p_{xy} r_1, \p_{y}^2 r_2, Z_2 \p_{xy} r_1, Z_2 \p_{y}^2 r_2)\|_{L^\infty_0(\O)}^2 d\t\\
&+C\int_0^t \|r_u\|_{\H^3_0}^4 d\t
 +C\int_0^t(1+\e_{5,1}^6(\t)+\|\p_y(\vr, \u, \h)\|_{\H^4_0}^{12}+\|\p_y \vr\|_{\H^{1,\infty}_0}^4)d\t,
\end{aligned}
\end{equation*}
which, together with the estimate \eqref{31109}, yields directly
\begin{equation}\label{311024}
\begin{aligned}
\|y\p_y^2 \rho\|_{L^\infty_0(\O)}^2
\le
&C(\|\p_y \vr_0\|_{L^\infty_0(\O)}^2+ \|Z_2 \p_y \vr_0\|_{L^\infty_0(\O)}^2)+C\|\p_y \vr\|_{\H^3_1}^2\\
&+C\int_0^t (\|r_u\|_{\H^3_0}^4+
             \|(\p_{xy} r_1, \p_{y}^2 r_2, Z_2 \p_{xy} r_1, Z_2 \p_{y}^2 r_2)\|_{L^\infty_0(\O)}^2)d\t\\
&+C\int_0^t(1+\e_{5,1}^6(\t)+\|\p_y(\vr, \u, \h)\|_{\H^4_0}^{12}+\|\p_y \vr\|_{\H^{1,\infty}_0}^4)d\t.
\end{aligned}
\end{equation}
Therefore, substituting the estimates \eqref{31104}, \eqref{31108} and \eqref{311024} into \eqref{31102},
we complete the proof of lemma.
\end{proof}

\subsection{Proof of Theorem \ref{theo a priori}}

Based on the estimates obtained so far, we can complete the proof of Theorem \ref{theo a priori} in this subsection.
First of all, we give the proof for the estimate \eqref{3a1}.
For two parameters $R$ and $\d$, which will be defined later, we define
\begin{equation*}\label{3d1}
\begin{aligned}
T^\es_* :=&\sup\left\{T\in [0, 1]\ |\ \Theta_{m,l}(t)\le R, \
\|\p_y (\u-e^{-y})(t)\|_{L^\infty_1(\O)}\le \d^{-1}, \right.\\
&\quad \quad \quad \quad \left.
\|\vr(t)\|_{L^\infty_0(\O)}\le \frac{2l-1}{2}\d^2, \
h^\es(t,x,y)+1\ge \d, \ \forall t \ \in [0, T], \ (x, y)\in \O\right\}.
\end{aligned}
\end{equation*}
Now, we write
\begin{equation}\label{3d2}
\begin{aligned}
&\n_{m,l}(t)
:=\sup_{0\le s \le t}\{1+\e_{m,l}(s)+\|(\vr_m,\u_m, \h_m)(s)\|_{L^2_l(\O)}^2
   +\|(\p_y \vr, \p_y \u, \p_y \h)(s)\|_{\H^{m-1}_l}^2+\|\p_y \vr(s)\|_{\H^{1,\infty}_1}^2\}\\
&\quad \quad \
 +\!\int_0^t \!\!\!\|\p_y^2(\sqrt{\es} \vr, \sqrt{\mu} \u, \sqrt{\k} \h)\|_{\H^{m-1}_l}^2d\t
 \!+\es \!\int_0^t\!\!\! \|\p_{xy}(\vr,\u,\h)\|_{\H^{m-1}_l}^2 d\t\!+\!\int_0^t\! (\D_x^{m,l}\!+\!\D_y^{m,l})(\t)d\t,
\end{aligned}
\end{equation}
where $D_x^{m,l}(t)$ and $D_y^{m,l}(t)$ are defined by
\begin{equation}\label{3d3}
\D_x^{m,l}(t)=
\sum_{\substack{0 \le |\alpha| \le m \\ |\a_1| \le m-1}}\es \|\p_x  \z^\a(\vr, \u, \h)(t)\|_{L^2_l(\O)}^2
+\es\|\p_x (\vr_m, \u_m, \h_m)(t)\|_{L^2_l}^2
\end{equation}
and
\begin{equation}\label{3d3-0}
\begin{aligned}
\D_y^{m,l}(t)=
&\sum_{\substack{0 \le |\alpha| \le m \\ |\a_1| \le m-1}}
\|\p_y(\sqrt{\es}\z^\a \vr,\sqrt{\mu}\z^\a \u,\sqrt{\k}\z^\a\h)(t)\|_{L^2_l(\O)}^2\\
&+\|\p_y(\sqrt{\es}\vr_m, \sqrt{\mu}\u_m, \sqrt{\k} \h_m)(t)\|_{L^2_l(\O)}^2.
\end{aligned}
\end{equation}

From the estimates in Propositions \ref{Lower-estimate}, \ref{Tanential-estimate},\ref{Normal-estimate},\ref{Infinity-estimate},
we may conclude for $T_1 \le T^\es_*$ that
\begin{equation}\label{3d4}
\n_{m,l}(t) \le C\d^{-2}(1+\|(\rho_0, u_{10}, h_{10})\|_{\overline{\mathcal{B}}^m_l}^6)
        +C_{\mu, \k}t\|(\rho_0, u_{10}, h_{10})\|_{\widehat{\mathcal{B}}^m_l}^6
        +C_{\mu, \k, m, l}\d^{-12} t \n_{m,l}^{12}(t), t \in [0, T_1].
\end{equation}
On the other hand, recall the almost equivalently relations
(see Lemma \ref{equi-control})
\begin{equation}\label{3d5}
\Theta_{m,l}(t)
\le C\|(\rho_0, u_{10}, h_{10})\|_{\overline{\mathcal{B}}^m_l}^4
     +C t \|(\rho_0, u_{10}, h_{10})\|_{\widehat{\mathcal{B}}^m_l}^4
     +C_l \d^{-8}\n_{m,l}^{12}(t),
\quad \forall t \in [0, T_1],
\end{equation}
and
\begin{equation}\label{3d6}
\n_{m,l}(t)
\le C\|(\rho_0, u_{10}, h_{10})\|_{\overline{\mathcal{B}}^m_l}^4
     +C t \|(\rho_0, u_{10}, h_{10})\|_{\widehat{\mathcal{B}}^m_l}^4
     +C\d^{-8} \Theta_{m,l}^{12}(t), \quad \forall t \in [0, T_1],
\end{equation}
and hence, we may deduce from the estimates \eqref{3d4}, \eqref{3d5} and \eqref{3d6} that
\begin{equation*}\label{3d7}
\begin{aligned}
\Theta_{m,l}(T_1)
\le
&C_{l} \mathcal{P}_0(\d^{-1}, \|(\rho_0, u_{10}, h_{10})\|_{\overline{\mathcal{B}}^m_l})
 +C_{\mu, \k, m, l} T_1 \mathcal{P}_1(\d^{-1}, \|(\rho_0, u_{10}, h_{10})\|_{\overline{\mathcal{B}}^m_l})\\
&+C_{\mu, \k, m, l}T_1 \mathcal{P}_2(\d^{-1}, \|(\rho_0, u_{10}, h_{10})\|_{\widehat{\mathcal{B}}^m_l})
 +C_{\mu, \k, m, l} T_1 \mathcal{P}_3(\d^{-1}, R).
\end{aligned}
\end{equation*}
Choose constant $\d=\frac{\d_0}{2}$ and
$R=4 C_l \mathcal{P}_0(\d_0^{-1}, \|(\rho_0, u_{10}, h_{10})\|_{\overline{\mathcal{B}}^m_l})$, we obtain
\begin{equation*}\label{3d8}
\begin{aligned}
\Theta_{m,l}(T_1)
\le
&C_l \mathcal{P}_0(\d^{-1}_0, \|(\rho_0, u_{10}, h_{10})\|_{\overline{\mathcal{B}}^m_l})
+C_{\mu, \k, m, l}T_1 \mathcal{P}_4(\d^{-1}_0, \|(\rho_0, u_{10}, h_{10})\|_{\overline{\mathcal{B}}^m_l})\\
&+C_{\mu, \k, m, l}T_1\mathcal{P}_5(\d^{-1}_0, \|(\rho_0, u_{10}, h_{10})\|_{\widehat{\mathcal{B}}^m_l})
+C_{\mu, \k, m, l}T_1\mathcal{P}_6(\d^{-1}_0, \|(\rho_0, u_{10}, h_{10})\|_{\overline{\mathcal{B}}^m_l}).
\end{aligned}
\end{equation*}
Choose the time
$T_1=\min\{
           \frac{\overline{C}_{0}}
                {\mathcal{P}_4(\d^{-1}_0, \|(\rho_0, u_{10}, h_{10})\|_{\overline{\mathcal{B}}^m_l})},
           \frac{\overline{C}_{0}}
                {\mathcal{P}_5(\d^{-1}_0, \|(\rho_0, u_{10}, h_{10})\|_{\widehat{\mathcal{B}}^m_l})},
           \frac{\overline{C}_{0}}
                {\mathcal{P}_6(\d^{-1}_0, \|(\rho_0, u_{10}, h_{10})\|_{\overline{\mathcal{B}}^m_l})},\}$,
and hence, it follows
\begin{equation*}\label{3d9}
\Theta_{m,l}(T_1)\le 2 C_l \mathcal{P}_0(\d^{-1}_0, \|(\rho_0, u_{10}, h_{10})\|_{\overline{\mathcal{B}}^m_l})=\frac{R}{2}.
\end{equation*}
Here the constant $\overline{C}_{0}:=\frac{\mathcal{P}_0(\d^{-1}_0, \|(\rho_0, u_{10},h_{10})\|_{\overline{\mathcal{B}}^m_l})}{3C_{\mu, \k, m, l}}$.
For any smooth function $W(t, x, y)$, it is easy to justify
\begin{equation}\label{relation}
W(t, x, y)=W(0, x, y)+\int_0^t \p_s W(s, x, y)ds,
\end{equation}
Using the relation \eqref{relation} and the Sobolev inequality \eqref{sobolev}, we get
\begin{equation*}
\begin{aligned}
\h(t, x, y)+1
&\ge \h_0(x, y)+1 -Ct\sup_{0\le s \le t}\|(\h, \p_y \h)\|_{\H^2_0} \\
&\ge 2\d_0-2C_l t \sqrt{\mathcal{P}_0(\d_0^{-1}, \|(\rho_0, u_{10}, h_{10})\|_{\overline{\mathcal{B}}^m_l})}.
\end{aligned}
\end{equation*}
Choose $T_2=\min\{T_1, \frac{\d_0}
{2 C_l \sqrt{\mathcal{P}_0(\d_0^{-1}, \|(\rho_0, u_{10}, h_{10})\|_{\overline{\mathcal{B}}^m_l})}}\}$,
it follows
\begin{equation*}\label{3d10}
\h(t,x,y)+1 \ge \d_0=2\d,\quad \text{for~all}~(t, x, y)\in [0, T_2]\times \O.
\end{equation*}

Similarly, we get from the relation \eqref{relation} and the estimate \eqref{31012}$_2$ that
\begin{equation*}
\begin{aligned}
\|\p_y(\u-e^{-y})(t)\|_{L^\infty_1(\O)}
&\le \|\p_y(\u_0-e^{-y})\|_{L^\infty_1(\O)}
+t\sup_{0\le s \le t}\|\p_y \p_s \u(s)\|_{L^\infty_1(\O)}\\
&\le (2\d_0)^{-1}
   +Ct(1+\|(\rho_0, u_{10}, h_{10})\|_{\widehat{\mathcal{B}}^m_l}
      +(\mathcal{P}_0(\d_0^{-1}, \|(\rho_0, u_{10}, h_{10})\|_{\overline{\mathcal{B}}^m_l}))^{\frac{3}{2}}).
\end{aligned}
\end{equation*}
Choosing $T_3=\min\{T_2, \frac{1}{2 \d_0 C(1+\|(\rho_0, u_{10}, h_{10})\|_{\widehat{\mathcal{B}}^m_l}
+(\mathcal{P}_0(\d_0^{-1}, \|(\rho_0, u_{10}, h_{10})\|_{\overline{\mathcal{B}}^m_l}))^{\frac{3}{2}})} \}$,
and hence
\begin{equation*}\label{3d11}
\begin{aligned}
\|\p_y(\u-e^{-y})(t)\|_{L^\infty_1(\O)}
\le \d_0^{-1}=(2\d)^{-1},\quad \text{for~all}~t\in [0, T_3].
\end{aligned}
\end{equation*}

Finally, from the relation  \eqref{relation} and the Sobolev inequality \eqref{sobolev}, we find
\begin{equation*}
\begin{aligned}
\|\vr(t)\|_{L^\infty_0(\O)}
&\le \|\vr_0\|_{L^\infty_0(\O)}+Ct\sup_{0\le s \le t}\|(\vr, \p_y \vr)(s)\|_{\H^2_0}\\
&\le \frac{2l-1}{16}\delta^2_0+2 C_l t\sqrt{\mathcal{P}_0(\d_0^{-1}, \|(\rho_0, u_{10}, h_{10})\|_{\overline{\mathcal{B}}^m_l})}.
\end{aligned}
\end{equation*}
Choose
$T_4=\min\{T_3, \frac{2l-1}{64C_l \sqrt{\mathcal{P}_0(\d_0^{-1},
\|(\rho_0, u_{10},h_{10})\|_{\overline{\mathcal{B}}^m_l})}}\d_0^2\}$, we obtain
\begin{equation*}\label{3d12}
\|\vr(t)\|_{L^\infty_0(\O)}\le \frac{3(2l-1)}{32}\delta^2_0=\frac{3(2l-1)}{8}\delta^2,
\end{equation*}
for all $t\in [0, T_4]$.
Obviously, we conclude that there exists $T_4>0$ depending only on
$\mu, \k, m, l, \d_0$ and the initial data
(hence independent of parameter $\es$) such that for $T\le \min\{T_4, T^\es\}$,
the estimates \eqref{3a1} and \eqref{3a2} hold on. Of course, it holds that $T_4\le T^\es_*$.
Indeed otherwise, our criterion about the continuation of
the solution would contradict the definition of $T^\es_*$.
Then, taking $T_a=T_4$, we obtain the estimate \eqref{3a2} and closes the a priori assumptions
\eqref{a2} and \eqref{a1}. Therefore, the proof of Theorem \ref{theo a priori} is completed.

\section{Local-in-time Existence and Uniqueness}\label{local-in-times}

In this section, we will establish the local-in-times existence and uniqueness of solutions
to the inhomogeneous incompressible MHD boundary layer equations \eqref{eq3}-\eqref{bc3}.

\subsection{Existence for the MHD Boundary Layer System}

We shall use the a priori estimates obtained thus far to prove local in time
existence result. For $m \ge 5$ and $l \ge 2$, consider initial data
such that $ \|(\rho_0, u_{10}, h_{10})\|_{\mathcal{B}^m_l}\le C_0<+\infty$.
For such initial data, we are not aware of a local well-posedness result for
the equations \eqref{eq5}-\eqref{bc5}. Since $(\rho_0, u_{10}, h_{10}) \in \mathcal{B}^{m,l}_{BL}$, there
exists a sequence of smooth approximate initial data
$(\rho_0^{\sigma}, u_{10}^{\sigma}, h_{10}^{\sigma}) \in \mathcal{B}^{m,l}_{BL, ap}$
($\sigma$ being a regularization parameter), which have enough spatial regularity
so that the time derivatives at the initial time can be defined by the equation \eqref{eq3}
and boundary compatibility condition are satisfied.
Then, it follows to get a positive time $T^{\es, \sigma}>0$($T^{\es, \sigma}$
depends on $\es, \sigma$, and the initial data) for which a solution
$(\varrho^{\es, \sigma}, u^{\es, \sigma}, h^{\es, \sigma})$ exists in Sobolev spaces $H^{4m}_l(\O)$
and $(v^{\es, \sigma}, g^{\es, \sigma})$ exists in Sobolev spaces $H^{4m}_{l-1}(\O)$ respectively.
Applying the a priori estimates given in Theorem \ref{theo a priori} to
$(\varrho^{\es, \sigma}, u^{\es, \sigma}, h^{\es, \sigma})$,
we obtain a uniform time $T_a>0$ and a constant $C_1$(independent of $\es$ and $\sigma$),
such that it holds on
\begin{equation}\label{ext1}
\ta_{m,l}(\varrho^{\es, \sigma}, u^{\es, \sigma}, h^{\es, \sigma})(t) \le C_1, \quad
\|\varrho^{\es, \sigma}(t)\|_{L^\infty_0(\O)}\le \frac{2l-1}{2}\delta^2_0,\quad
\| \p_y (u^{\es, \sigma}-e^{-y})(t)\|_{L^\infty_1(\O)}\le \delta_0^{-1},
\end{equation}
and
\begin{equation}\label{ext2}
h^{\es, \sigma}(t,x,y)+1\ge \delta_0,
\end{equation}
where $t \in [0, T_0], T_0:=\min(T_a, T^{\es, \sigma})$.
Based on the uniform estimates for $(\varrho^{\es, \sigma}, u^{\es, \sigma}, h^{\es, \sigma})$, one can pass the limit
$\es \rightarrow 0^+$ and $\sigma \rightarrow 0^+$
to get a strong solution $(\varrho, u, h)$ satisfying \eqref{eq3}
by using a strong compactness arguments.
Indeed, it follows from \eqref{ext1} that $(\varrho^{\es, \sigma}, u^{\es, \sigma}, h^{\es, \sigma})$
is bounded uniformly in $L^\infty([0, T_2]; H^m_{co})$,
while $\p_y(\varrho^{\es, \sigma}, u^{\es, \sigma}, h^{\es, \sigma})$
is bounded uniformly in $L^\infty([0, T_0]; H^{m-1}_{co})$,
and $\p_t (\varrho^{\es, \sigma}, u^{\es, \sigma}, h^{\es, \sigma})$
is bounded uniformly in $L^\infty([0, T_0]; H^{m-1}_{co})$.
Then, it follows from a strong compactness argument that $(\varrho^{\es, \sigma}, u^{\es, \sigma}, h^{\es, \sigma})$
is compact in $\mathcal{C}([0, T_0]; H^{m-1}_{co,loc})$.
Due to $\kappa>0$, it is easy to check that
$h^{\es, \sigma}$ is compact in $\mathcal{C}([0, T_0]; H^{2}_{loc})$.
In particular, there exists a sequence $\es_n, \sigma_n \rightarrow 0^+$
and $(\varrho, u, h)\in \mathcal{C}([0, T_0]; H^{m-1}_{co,loc})$ such that
$$
(\varrho^{\es_n, \sigma_n}, u^{\es_n, \sigma_n}, h^{\es_n, \sigma_n}) \rightarrow (\varrho, u, h)~~ {\rm in}~ ~
\mathcal{C}([0, T_0]; H^{m-1}_{co,loc})~ ~{\rm as} ~~\es^n, \sigma^n \rightarrow 0^+,
$$
and
$$
h^{\es_n, \sigma_n} \rightarrow  h ~~ {\rm in}~ ~
\mathcal{C}([0, T_0]; H^{2}_{loc})~ ~{\rm as} ~~\es^n, \sigma^n \rightarrow 0^+,
$$
Furthermore, we apply the Sobolev inequality to get
\begin{equation*}
\begin{aligned}
&\underset{0\le \t \le t}{\sup}\|(\p_y^{-1} \p_x u^{\es_n, \sigma_n}-\p_y^{-1} \p_x u)(\t)\|_{L^\infty_{0, co}(\O)}\\
&\le C\underset{0\le \t \le t}{\sup} \|\p_y^{-1}\p_x (u^{\es_n, \sigma_n}-u)(\t)\|_{H^{1}_{co, loc}}^{\frac{1}{2}}
        \|\p_x (u^{\es_n, \sigma_n}-u)(\t)\|_{H^{1}_{co, loc}}^{\frac{1}{2}}\\
&\le C\underset{0\le \t \le t}{\sup} \|(u^{\es_n, \sigma_n}-u)(\t)\|_{H^{2}_{co, loc}}
      \rightarrow 0, ~~{\rm as}~~\es^n, \sigma^n \rightarrow 0^+.
\end{aligned}
\end{equation*}
Hence we denote $v(t, x, y)=-\int_0^y \p_x u(t, x, \xi)d\xi$,
which satisfies the divergence-free condition $\p_x u+\p_y v=0$.
Similarly, we denote $g(t, x, y)=-\int_0^y h(t, x, \xi)d\xi$, which satisfies
\begin{equation*}
\underset{0\le \t \le t}{\sup}\|(g^{\es_n, \sigma_n}- g)(\t)\|_{L^\infty_{0, co}(\O)}
\le C\underset{0\le \t \le t}{\sup} \|(h^{\es_n, \sigma_n}-h)(\t)\|_{H^{1}_{co, loc}}
      \rightarrow 0, ~~{\rm as}~~\es^n, \sigma^n \rightarrow 0^+.
\end{equation*}
By routine checking, we may show that $(\rho, u_1, u_2, h_1, h_2):=(\varrho+1, u+1-e^{-y}, v, h+1, g)$
is a solution of the original MHD boundary layer system \eqref{eq3}.
Finally, applying the lower semicontinuity of norms to the bound \eqref{ext1},
one obtains the estimate \eqref{main-estimate} for the solution $(\rho, u_1, h_1)$.
Since $h^{\es_n, \sigma_n}$ converges uniformly to $h$, then we can get $h_1\ge \d$ from \eqref{ext2}.

\subsection{Uniqueness for the MHD Boundary Layer System}

In this subsection, we will show the uniqueness of solution to the MHD boundary
layer equations \eqref{eq3}-\eqref{bc3}.
Let $(\rho_1, u_1, v_1, h_1, g_1)$ and $(\rho_2, u_2, v_2, h_2, g_2)$ be two solutions
in the existence time $[0, T_a]$, constructed in the previous subsection,
with respect to the initial data $(\rho_0^1, u_0^1, h_0^1)$ and $(\rho_0^2, u_0^2, h_0^2)$ respectively.

Let us set
$$
(\overline{\rho}, \overline{u}, \overline{v}, \overline{h}, \overline{g})
=(\rho_1-\rho_2, u_1-u_2, v_1-v_2, h_1-h_2, g_1-g_2),
$$
then they satisfy the following evolution
\begin{equation}\label{eq7}
\left\{
\begin{aligned}
&\p_t \orho+u_1 \p_x \orho+v_1 \p_y \orho+\ou \p_x \rho_2+ \ov \p_y \rho_2=0,\\
&\rho_1\p_t \ou+\rho_1 u_1 \p_x \ou+\rho_1 v_1 \p_y \ou+\rho_1 \ou \p_x u_2+\rho_1 \ov \p_y u_2-\mu \p_y^2 \ou \\
&\quad =-\orho \p_t u_2-\orho u_2 \p_x u_2-\orho v_2 \p_y u_2
        +\oh \p_x h_1+h_2 \p_x \oh+g_1 \p_y \oh+\og \p_y h_2,\\
&\p_t \oh+u_1 \p_x \oh+v_1 \p_y \oh+\ou \p_x h_2+\ov \p_y h_2-\k \p_y^2 \oh
 =\oh \p_x u_1+h_2 \p_x \ou+g_1 \p_y \ou+\og \p_y u_2,\\
&\p_x \ou+\p_y \ov=0,\quad \p_x \oh+\p_y \og=0,\\
\end{aligned}
\right.
\end{equation}
with the boundary condition and initial data
\begin{equation*}
(\ou, \ov, \p_y \oh, \og)|_{y=0}=\mathbf{0},
\quad \lim_{y \rightarrow +\infty}(\orho, \ou, \oh)=\mathbf{0}, \quad
(\orho, \ou, \oh)|_{t=0}=\mathbf{0}.
\end{equation*}
Here we assume the two solutions $(\rho_1, u_1, v_1, h_1, g_1)$ and $(\rho_2, u_2, v_2, h_2, g_2)$
have the same initial data $(\rho_0^1, u_0^1, h_0^1)=(\rho_0^2, u_0^2, h_0^2)$.
Denote by $\ophi:=\p_y^{-1}\oh=\p_y^{-1}(h_1-h_2)$, it follows
\begin{equation*}
\p_t \ophi+u_1 \p_x \ophi+ v_1 \p_y \ophi-\ou g_2+\ov h_2-\k \p_y^2 \ophi=0.
\end{equation*}
Define
$
\eta_1:=\frac{\p_y \rho_2}{h_2}, \
\eta_2:=\frac{\p_y u_2}{h_2}, \
\eta_3:=\frac{\p_y h_2}{h_2},
$
and introduce the new quantities:
\begin{equation}\label{eqi-quantity}
\irho:=\orho-\eta_1 \ophi,\quad  \iu:=\ou-\eta_2 \ophi,\quad \ih:=\oh-\eta_3 \ophi.
\end{equation}

Next, we can obtain that through direct calculation, $(\irho, \iu, \ih)$ satisfies the following
initial boundary value problem:
\begin{equation}\label{uniq-eq}
\left\{
\begin{aligned}
&\p_t \irho+u_1 \p_x \irho+v_1 \p_y \irho=-\k \eta_1 \p_y \ih
   -a_{11}\ou- a_{12}\oh-a_{13} \ophi,\\
&\rho_1\p_t \iu+\rho_1 u_1 \p_x \iu+\rho_1 v_1 \p_y \iu
-h_2 \p_x \ih-g_1 \p_y \ih-\mu \p_y^2 \iu,\\
&\quad =-\k \rho_1 \eta_2 \p_y \ih+\mu \p_y(\p_y \eta_2 \ophi+\eta_2 \oh)
        -a_{21}\orho-a_{22}\ou- a_{23}\oh-a_{24} \ophi,\\
&\p_t \ih+u_1 \p_x \ih+v_1 \p_y \ih-h_2 \p_x \iu-g_1 \p_y \iu-\k \p_y^2 \ih
 =-\k \eta_3 \p_y \ih-a_{31}\ou- a_{32}\oh-a_{33} \ophi,
\end{aligned}
\right.
\end{equation}
where
\begin{equation*}
\left\{
\begin{aligned}
&a_{11}=\p_x \rho_2+\eta_1 g_2, \quad a_{12}=\k \eta_1 \eta_3,
\quad a_{13}=\p_t \eta_1+u_1 \p_x \eta_1+v_1 \p_y \eta_1+\k \eta_1 \p_y \eta_3,\\
&a_{21}=\p_t u_2 +u_2 \p_x u_2+ v_2 \p_y u_2,\quad
 a_{22}=\rho_1 \p_x u_2+\rho_1 \eta_2 g_2,\quad
 a_{23}=\k \rho_1 \eta_2 \eta_3-\p_x h_1-g_1 \eta_3,\\
&
 a_{24}=\rho_1 \p_t \eta_2+\rho_1 u_1 \p_x \eta_2+\rho_1 v_1 \p_y \eta_2
        +\k \rho_1 \eta_2 \p_y \eta_3-h_2 \p_x \eta_3-g_1 \p_y \eta_3,\\
&a_{31}=\eta_3 g_2+\p_x h_2,\quad
 a_{32}=\k \eta_3^2+2\k \p_y \eta_3-g_1 \eta_2-\p_x u_1,\\
& a_{33}=\p_t \eta_3+u_1 \p_x \eta_3+v_1 \p_y \eta_3+\k \eta_3 \p_y \eta_3
             -h_2 \p_x \eta_2-g_1 \p_y \eta_2.
\end{aligned}
\right.
\end{equation*}

By virtue of the relation
$\ih:=\oh-\frac{\p_y h_2}{h_2}\ophi =h_2\p_y(\frac{\ophi}{h_2})$, it is easy to justify
$\frac{\ophi}{h_2}=\p_y^{-1}(\frac{\ih}{h_2})$,
and hence, we can apply the Hardy inequality to obtain the estimate:
$$
\|\frac{\ophi}{h_2}\|_{L^2_{-1}(\O)}\le C\|\frac{1}{h_2}\|_{L^\infty_0(\O)}\|\ih\|_{L^2_0(\O)}.
$$
Thus we can apply the standard energy method for the equation \eqref{uniq-eq}
to establish the following estimate, which we omit the proof for brevity of presentation.

\begin{proposition}\label{uniqueness-energy}
Let $(\rho_1, u_1, v_1, h_1, g_1)$ and $(\rho_2, u_2, v_2, h_2, g_2)$
be two solutions of MHD boundary layer equations \eqref{eq3}-\eqref{bc3} with the same initial data,
and satisfying the estimate \eqref{main-estimate} respectively.
Then, there exists a positive constant
$$
C=C(T_a, \d_0, \|(\rho_1, u_1, h_1)(t)\|_{\overline{\mathcal{B}}^m_l},
\|(\rho_2, u_2, h_2)(t)\|_{\overline{\mathcal{B}}^m_l})>0,
$$
such that the quantity $(\irho, \iu, \ih)$ given by \eqref{eqi-quantity} satisfies
\begin{equation}\label{unique-estimate}
\|(\irho, \iu, \ih)(t)\|_{L^2}^2
+\int_0^t \|\p_y(\sqrt{\mu}\ \iu, \sqrt{\k}\ \ih)(\t)\|_{L^2}^2 d\t
\le C\int_0^t \|(\irho, \iu, \ih)(\t)\|_{L^2}^2 d\t.
\end{equation}
\end{proposition}

Then, we can prove the uniqueness of the solutions to \eqref{eq3}-\eqref{bc3} as follows.

\begin{proof}[\textbf{Proof of Uniqueness.}]
Applying Gronwall's lemma to the estimate \eqref{unique-estimate}, we obtain $ (\irho, \iu, \ih)\equiv 0 $.
Then, we substitute $\ih \equiv 0$ into equality $\frac{\ophi}{h_2}=\p_y^{-1}(\frac{\ih}{h_2})$ to get $\ophi \equiv 0$.
From the definition \eqref{eqi-quantity}  we get $(\rho_1, u_1, h_1)\equiv (\rho_2, u_2, h_2)$ due to the fact
$ (\irho, \iu, \ih)\equiv 0 $ and $\ophi \equiv 0$.
Finally, it follows from the divergence-free condition and the boundary condition $\ov|_{y=0}=0$ that
$\ov=-\p_y^{-1} \p_x \ou=0$, which implies the fact $v_1 \equiv v_2$.
Similarly, it holds on $g_1 \equiv g_2$.
Therefore, we complete the proof of uniqueness of the solution for the
MHD boundary layer equations \eqref{eq3}-\eqref{bc3} completely.
\end{proof}

\appendix

\section{Calculus Inequalities}\label{appendixA}

In this appendix, we will introduce some basic inequalities that be used frequently in this paper.
First of all, we introduce the following Hardy type inequality, which can refer to \cite{Masmoudi}.

\begin{lemma}\label{Hardy}
Let the proper function $f:\mathbb{T}\times \mathbb{R}^+ \rightarrow \mathbb{R}$,
and satisfies $f(x, y)|_{y=0}=0$ and $ \underset{{y \to +\infty}}{\lim} f(x,y) = 0$.
If $\lambda > - \frac{1}{2}$, then it holds on
\begin{equation} \label{Hardy}
    \|f\|_{L^2_\lambda (\mathbb{T}\times \mathbb{R}^+)} \le \frac{2}{2\lambda +1}
    \|\partial_y f\|_{L^2_{\lambda+1} (\mathbb{T}\times \mathbb{R}^+)}.
\end{equation}
\end{lemma}

Next, we will state the following Sobolev-type inequality.

\begin{lemma}
Let the proper function $f:\mathbb{T}\times \mathbb{R}^+ \rightarrow \mathbb{R}$,
and satisfies $ \underset{{y \to +\infty}}{\lim} f(x,y) = 0$.
Then there exists a universal constant $C>0$ such that
\begin{equation}\label{sobolev-1}
\|f\|_{L^\infty_0(\mathbb{T}\times \mathbb{R}^+)}
\le C(\|\p_y f\|_{L^2_0(\mathbb{T}\times \mathbb{R}^+)}+\|\p_{xy}^2 f\|_{L^2_0(\mathbb{T}\times \mathbb{R}^+)})^{\frac{1}{2}}
     (\|f\|_{L^2_0(\mathbb{T}\times \mathbb{R}^+)}+\|\p_x f\|_{L^2_0(\mathbb{T}\times \mathbb{R}^+)})^{\frac{1}{2}},
\end{equation}
or equivalently
\begin{equation}\label{sobolev}
\|f\|_{L^\infty_0(\mathbb{T}\times \mathbb{R}^+)}
\le C(\|f\|_{L^2_0(\mathbb{T}\times \mathbb{R}^+)}+\|\partial_x  f\|_{L^2_0(\mathbb{T}\times \mathbb{R}^+)}
+\|\partial_y  f\|_{L^2_0(\mathbb{T}\times \mathbb{R}^+)}+\|\partial_{xy}^2 f\|_{L^2_0(\mathbb{T}\times \mathbb{R}^+)}).
\end{equation}
\end{lemma}
\begin{proof}
Indeed, the estimate \eqref{sobolev} follows directly from estimate
\eqref{sobolev-1} and the Cauchy-Schwartz inequality. Hence, we only give the proof
for the estimate \eqref{sobolev-1}.
On one hand, thanks to the one-dimensional Sobolev
inequality for the $y-$variable, we get
\begin{equation}\label{a21}
|f(x, y)|^2 \le C(\int_0^\infty |\p_\xi f(x, \xi)|^2d\xi)^{\frac{1}{2}}
              (\int_0^\infty |f(x, \xi)|^2d\xi)^{\frac{1}{2}}.
\end{equation}
On the other hand, we apply the following one-dimensional Sobolev
inequality for $x-$variable to get
\begin{equation}\label{a22}
|f(x, y)|^2\le C(\|f(y)\|_{L^2(\mathbb{T})}^2+\|\p_x f(y)\|_{L^2(\mathbb{T})}^2),
\quad
|\p_y f(x, y)|^2\le C(\|\p_y f(y)\|_{L^2(\mathbb{T})}^2+\|\p_{xy} f(y)\|_{L^2(\mathbb{T})}^2).
\end{equation}
Therefore, substituting the estimate \eqref{a22} into \eqref{a21}, we complete
the proof of estimate \eqref{sobolev-1}.
\end{proof}

Now we will state the Moser type inequality as follow:

\begin{lemma}\label{moser}
Denote $\O:=\mathbb{T}\times \mathbb{R}^+$, let the proper
functions $f(t, x, y): \mathbb{R}^+\times \O \rightarrow \mathbb{R}$
and $g(t, x, y): \mathbb{R}^+ \times \O   \rightarrow \mathbb{R}$.
Then, there exists a constant $C_m>0$ such that
\begin{equation}\label{ineq-moser}
\int_0^t \|(\z^{\b} f \z^{\ga}g)(\t)\|_{L^2_l(\O)}^2d\t
\le C_m(\|\la y \ra^{l_1} f\|_{L^\infty_{x,y,t}}^2\int_0^t \|g\|_{\H^{m}_{l_2}}^2 d\t
       +\|\la y \ra^{l_1} g\|_{L^\infty_{x,y,t}}^2\int_0^t \|f\|_{\H^{m}_{l_2}}^2 d\t),
\end{equation}
where $|\b+\ga|=m$ and $l_1+l_2=l$.
\end{lemma}

\begin{proof}
For any $p\ge 2$, due to the relation $|Z_2 f|^p=Z_2(f Z_2 f |Z_2 f|^{p-2})-(p-1)f Z_2^2 f|Z_2 f|^{p-2}$,
we find
\begin{equation*}
\int_{\mathbb{R}^+} \la y \ra^{\theta l p}|Z_2 f|^p dy
=\int_{\mathbb{R}^+}\la y \ra^{\theta l p} Z_2(f Z_2 f |Z_2 f|^{p-2}) dy
 -(p-1) \int_{\mathbb{R}^+}\la y \ra^{\theta l p}f Z_2^2 f|Z_2 f|^{p-2}dy.
\end{equation*}
Integrating by part and applying the H\"{o}lder inequality, we find for $0 \le \theta \le 1$
and  $0\le \theta_1 \le \frac{\theta}{2}$ that
\begin{equation*}
\begin{aligned}
\|\la y \ra^{\theta l} Z_2 f\|_{L_0^p(\mathbb{R}^+)}^p
&\le C_p \int_{\mathbb{R}^+}\la y \ra^{\theta l p-1} |f|(|Z_2 f|+|Z_2^2 f|)|Z_2 f|^{p-2} dy\\
&\le C_p\|\la y \ra^{\theta l} Z_2 f\|_{L_0^p(\mathbb{R}^+)}^{p-2}
     \|\la y \ra^{\theta_1 l} f\|_{L_0^q(\mathbb{R}^+)}
     \|\la y \ra^{(2\theta-\theta_1) l}(|Z_2 f|,|Z_2^2 f|)\|_{L_0^r(\mathbb{R}^+)},
\end{aligned}
\end{equation*}
and hence, it follows
\begin{equation*}
\|\la y \ra^{\theta l} Z_2 f\|_{L_0^p(\mathbb{R}^+)}^2
\le C_p \|\la y \ra^{\theta_1 l} f\|_{L_0^q(\mathbb{R}^+)}
     \sum_{1\le k \le 2}\|\la y \ra^{(2\theta-\theta_1) l}Z_2^k f\|_{L_0^r(\mathbb{R}^+)}.
\end{equation*}
Here $\frac{1}{q}+\frac{1}{r}=\frac{2}{p}$.
Integrating with $t$ and $x$ variables, and applying H\"{o}lder inequality, we get
\begin{equation*}
\|\la y \ra^{\theta l} Z_2 f\|_{L^p(Q_T)}^2
\le C_p \|\la y \ra^{\theta_1 l} f\|_{L^q(Q_T)}
     \sum_{1\le k \le 2}\|\la y \ra^{(2\theta-\theta_1) l}Z_2^k f\|_{L^r(Q_T)}.
\end{equation*}
Here $Q_T=\O \times [0, T]$.
Similarly, it is easy to justify for $i=0,1$,
\begin{equation*}
\|\la y \ra^{\theta l} Z_i f\|_{L^p_0(Q_T)}^2
\le C_p \|\la y \ra^{\theta_1 l} f\|_{L^q_0(Q_T)}
     \sum_{1\le k \le 2}\|\la y \ra^{(2\theta-\theta_1) l}Z_i^k f\|_{L^r_0(Q_T)}.
\end{equation*}
Here $\frac{1}{q}+\frac{1}{r}=\frac{2}{p}$.
By multiple application of the above inequality, we get(proof by induction)
\begin{equation*}
\|\la y \ra^{(|\a|\theta-(|\a|-1)\theta_1)l} \z^\a f\|_{L^{p_1}_0(Q_T)}
\le C_p \|\la y \ra^{\theta_1 l} f\|_{L^{q_1}_0(Q_T)}^{1-\frac{|\a|}{m}}
     \sum_{1\le |\b| \le m}\|\la y \ra^{(m\theta-(m-1)\theta_1) l} \z^{\b} f\|_{L^{r_1}_0(Q_T)}^{{\frac{|\a|}{m}}},
\end{equation*}
where $\frac{1}{p_1}=\frac{1}{q_1}(1-\frac{|\a|}{m})+\frac{|\a|}{r_1 m}$ and $1\le |\a| \le m-1$.
Then, we get for $|\beta|+|\gamma|=m$ that
\begin{equation*}
\begin{aligned}
\|\la y \ra^{l}\z^{\b}f \z^{\ga} g\|_{L^2_0(Q_T)}^2
&\le C \|\la y \ra^{\frac{|\b|}{m}l+(1-\frac{2|\b|}{m})l_1} \z^{\b}f \|_{L^{\frac{2m}{|\b|}}_0(Q_T)}^2
          \|\la y \ra^{\frac{|\ga|}{m}l+(\frac{2|\b|}{m}-1)l_1}\z^{\ga} f\|_{L^{\frac{2m}{|\ga|}}_0(Q_T)}^2\\
&\le C_m \|\la y \ra^{l_1}f\|_{L^\infty_0(Q_T)}^{2(1-\frac{|\b|}{m})}
          \sum_{1\le |\b| \le m}\|\la y \ra^{l-l_1}\z^{\b}f\|_{L^2_0(Q_T)}^{\frac{2|\b|}{m}}\\
&\quad \quad \times
          \|\la y \ra^{l_1} g\|_{L^\infty_0(Q_T)}^{2(1- \frac{|\ga|}{m})}
          \sum_{1\le |\ga| \le m}\|\la y \ra^{l-l_1}\z^{\ga} g\|_{L^2_0(Q_T)}^{\frac{2|\ga|}{m}}\\
&\le C_m \|f\|_{L^\infty_{l_1}(Q_T)}^2\sum_{1\le |\b| \le m}\|\z^\b g\|_{L^2_{l-l_1}(Q_T)}^2
          +C_m \|g\|_{L^\infty_{l_1}(Q_T)}^2\sum_{1\le |\b| \le m}\|\z^\b f\|_{L^2_{l-l_1}(Q_T)}^2.
\end{aligned}
\end{equation*}
Therefore, we complete the proof of this lemma.
\end{proof}

Finally, we establish the following $L^\infty-$estimate with weight for the heat equation.

\begin{lemma}\label{A-heat}
For the heat equation $\p_t F(t, x)-\es \p_x^2 F(t, x)=G(t, x), \ (t, x) \in \mathbb{R}^+\times \mathbb{R}^+$;
with the boundary condition $F(t, x)|_{x=0}=0$ and initial data $F(t, x)|_{t=0}=F_0$.
Then, it holds on
\begin{equation}\label{eheat}
\|x\p_x F\|_{L^\infty_0(\mathbb{R}^+)}\le  C(\|F_0\|_{L^\infty_0(\mathbb{R}^+)}+\|x\p_x F_0\|_{L^\infty_0(\mathbb{R}^+)})
+C\int_0^t (\|G\|_{L^\infty_0(\mathbb{R}^+)}+\|x\p_x G\|_{L^\infty_0(\mathbb{R}^+)}) d\t,
\end{equation}
where $C$ is a constant independent of the parameter $\es$.
\end{lemma}

\begin{proof}
First of all, let us consider the heat equation
\begin{equation}\label{a41}
\p_t H(t, x)-\es \p_x^2 H(t, x)=0,\quad (t, x) \in \mathbb{R}^+\times \mathbb{R}^+;
\end{equation}
with the initial data and boundary condition
\begin{equation*}
H(t, x)|_{t=0}=H_0(x),\quad x \in \mathbb{R}^+;
\quad H(t, x)|_{x=0}=0,\quad  t \in \mathbb{R}^+.
\end{equation*}
In order to transform the problem \eqref{a41} into a problem in the whole space,
let us define $\widetilde{H}(t, x)$ by
\begin{equation*}\label{a42}
\widetilde{H}(t, x)= H(t, x), \ x>0;\quad
\widetilde{H}(t, x)= -H(t, -x), \ x<0,
\end{equation*}
and define the initial data $\widetilde{H}_0(x)$ by
\begin{equation*}
\widetilde{H}_0(x)= H_0(x), \ x>0;\quad
\widetilde{H}_0(x)= -H_0(-x), \ x<0.
\end{equation*}
It is easy to justify that the function $\widetilde{H}(t, x)$ solves the following evolution equation
\begin{equation}\label{a43}
\p_t \widetilde{H}(t, x)-\es \p_x^2 \widetilde{H}(t, x)=0,\quad (t, x) \in \mathbb{R}^+\times \mathbb{R};
\quad \widetilde{H}(t, x)|_{t=0}=\widetilde{H}_0(x), \quad x \in \mathbb{R}.
\end{equation}
Define $S(t, x)=\frac{1}{\sqrt{4\pi \es t}}e^{-\frac{|x|^2}{\sqrt{4\es t}}}$,
then the solution of evolution \eqref{a43} can be expressed as
\begin{equation}\label{seq}
\widetilde{H}(t, x)=\int_{\mathbb{R}} \widetilde{H}_0(\xi)S(t, x-\xi)d\xi,
\end{equation}
which implies directly
\begin{equation*}\label{a44}
x \p_x \widetilde{H}(t, x)=\int_{\mathbb{R}} \widetilde{H}_0(\xi)x \p_x S(t, x-\xi)d\xi.
\end{equation*}
In view of the relation
$x \p_x S(t, x-\xi)=(x-\xi)\p_x S(t, x-\xi)+\xi\p_x S(t, x-\xi)$, we get
\begin{equation*}\label{a45}
x \p_x \widetilde{H}(t, x)=\int_{\mathbb{R}} \widetilde{H}_0(\xi)(x-\xi) \p_x S(t, x-\xi)d\xi
+\int_{\mathbb{R}} \widetilde{H}_0(\xi)\xi \p_x S(t, x-\xi)d\xi.
\end{equation*}
Due to $\int_{\mathbb{R}}|(x-\xi) \p_x S(t, x-\xi)|d\xi \le C$, it follows
\begin{equation*}\label{a46}
|\int_{\mathbb{R}} \widetilde{H}_0(\xi)(x-\xi) \p_x S(t, x-\xi)d\xi|
\le C\|\widetilde{H}_0\|_{L^\infty_0(\mathbb{R})}.
\end{equation*}
Using the equality $\p_x S(t, x-\xi)=-\p_{\xi} S(t, x-\xi)$, the integration by part yields directly
\begin{equation*}\label{a47}
|\int_{\mathbb{R}}\widetilde{H}_0(\xi)\xi \p_x S(t, x-\xi)d\xi|
\le C(\|\widetilde{H}_0\|_{L^\infty_0(\mathbb{R})}
+\|x \p_x \widetilde{H}_0\|_{L^\infty_0(\mathbb{R})}),
\end{equation*}
and hence, we get
\begin{equation*}\label{a48}
\begin{aligned}
\|x \p_x H (t, x)\|_{L^\infty_0(\mathbb{R}^+)} \le
\|x \p_x \widetilde{H}(t, x)\|_{L^\infty_0(\mathbb{R})}
\le C \|(\widetilde{H}_0, x \p_x \widetilde{H}_0)\|_{L^\infty_0(\mathbb{R})})
\le C (\|{H}_0\|_{L^\infty_0(\mathbb{R}^+)}
        +\|x \p_x {H}_0\|_{L^\infty_0(\mathbb{R}^+)}).
\end{aligned}
\end{equation*}
This, along with representation \eqref{seq} and the
well-known Duhamel formula, we complete the proof of this lemma.
\end{proof}

\section{Almost Equivalence of Weighted Norms}\label{appendixB}

In this subsection we will use the quantity $\h_m$ in weighted norm,
and $\h$ and its derivatives in $L^\infty$ norm to control the quantities $Z_\t^{\a_1}\h$ and $Z_\t^{\a_1}\ps$
in weighted norm. To derive these estimates, we shall apply the Lemma \ref{Hardy},
which was introduced previously in Section \ref{appendixA}.

\begin{lemma}\label{equivalent}
Let the stream function $\ps(t, x, y)$ satisfies
$\p_y \ps=\h, \p_x \ps = -\g, \ps|_{y=0}=0$.
There exists a constant $\d\in (0, 1)$, such that $\h(t, x, y)+1\ge \d, \forall(t, x, y)\in [0, T]\times \O$.
Then, for $l \ge 1$ and $|\a_1|=m$, we have the following estimates:
\begin{equation}\label{b11}
\|\frac{Z^{\a_1}_\t \ps}{\h+1}\|_{L^2_{l-1}(\O)}\le \frac{2 \d^{-1}}{2l-1} \| \h_m \|_{L^2_l(\O)},
\end{equation}
\begin{equation}\label{b12}
\|Z^{\a_1}_\t \h\|_{L^2_l}(\O)  \le \|\h_m\|_{L^2_l(\O)}
+\frac{2 \d^{-1}}{2l-1} \|\p_y \h\|_{L^\infty_1(\O)} \| \h_m \|_{L^2_l(\O)},
\end{equation}
\begin{equation}\label{b13}
\|\frac{\p_x Z^{\a_1}_\t \ps}{\h+1}\|_{L^2_{l-1}(\O)}
\le \frac{2\d^{-1}}{2l-1} \|\p_x \h_m\|_{L^2_{l}(\O)}
    +\frac{4\d^{-2}}{2l-1} \|\p_x \h\|_{L^\infty_0(\O)}\|\h_m\|_{L^2_{l}(\O)},
\end{equation}
\begin{equation}\label{b14}
\|\p_y Z^{\a_1}_\t \h\|_{L^2_l(\O)}
\le \|\p_y \h_m \|_{L^2_l(\O)}+C_l \d^{-1}(\|\p_y \h \|_{L^\infty_0(\O)}
    +\|Z_2 \p_y  \h\|_{L^\infty_1(\O)})\| \h_m\|_{L^2_l(\O)},
\end{equation}
where the constant $C_l$ depends only on $l$.
\end{lemma}

\begin{proof}
(i) By virtue of the definition $\h_m=Z_\t^{\a_1}\h-\frac{\p_y \h}{\h+1} Z_\t^{\a_1} \ps$,
it is easy to obtain $\h_m=(\h+1)\p_y(\frac{Z_\t^{\a_1} \ps}{\h+1})$.
Integrating over $[0, y]$ and applying the boundary condition $\ps|_{y=0}=0$, we have
\begin{equation}\label{b15}
\frac{Z_\t^{\a_1} \ps}{\h+1}= \int_0^y \frac{\h_m}{\h+1}d\xi,
\end{equation}
and along with the Hardy inequality \eqref{Hardy}, yields directly
\begin{equation}\label{b16}
\| \frac{Z_\t^{\a_1} \ps}{\h+1} \|_{L^2_{l-1}(\O)}
\le \| \int_0^y \frac{\h_m}{\h+1}d\xi\|_{L^2_{l-1}(\O)}
\le  \frac{2}{2l-1}\|\frac{\h_m}{\h+1} \|_{L^2_l(\O)}
\le  \frac{2 \d^{-1}}{2l-1} \| \h_m \|_{L^2_l(\O)},
\end{equation}
where we have used the fact $\h+1\ge \d$ in the last inequality.

(ii)In view of the relation $Z_\t^{\a_1}\h=\h_m+\frac{\p_y \h}{\h+1} Z_\t^{\a_1} \ps$, we get
\begin{equation*}
\|Z^{\a_1}_\t \h\|_{L^2_l(\O)}
\le \|\h_m\|_{L^2_l(\O)}+\|\p_y \h\|_{L^\infty_1(\O)}
    \|\frac{ Z^{\a_1}_\t \ps}{\h+1}\|_{L^2_{l-1}(\O)},
\end{equation*}
which, together with estimate \eqref{b16}, yields directly
\begin{equation}\label{b17}
\|Z^{\a_1}_\t \h\|_{L^2_l(\O)}
\le \|\h_m\|_{L^2_l(\O)}+\frac{2 \d^{-1}}{2l-1} \|\p_y \h\|_{L^\infty_1(\O)} \| \h_m \|_{L^2_l(\O)}.
\end{equation}

(iii)Differentiating the equality ${Z_\t^{\a_1} \ps}=(\h+1)\int_0^y \frac{\h_m}{\h+1}d\xi$
with $x$ variable, we find
\begin{equation*}
\p_x Z_\t^{\a_1} \ps
=\p_x \h \int_0^y \frac{\h_m}{\h+1} d\xi
+(\h+1)\int_0^y \frac{\p_x \h_m}{\h+1}d\xi
-(\h+1)\int_0^y \frac{\h_m \p_x \h}{(\h+1)^2}d\xi,
\end{equation*}
which, implies that
\begin{equation*}
\frac{\p_x Z_\t^{\a_1} \ps}{\h+1}
=\frac{\p_x \h}{\h+1}\int_0^y \frac{\h_m}{\h+1} d\xi
+\int_0^y \frac{\p_x \h_m}{\h+1}d\xi
-\int_0^y \frac{\h_m \p_x \h}{(\h+1)^2}d\xi,
\end{equation*}
and hence, we apply the Hardy inequality \eqref{Hardy} and $\h+1 \ge \d$ to get
\begin{equation}\label{b18}
\begin{aligned}
\|\frac{\p_x Z^{\a_1}_\t \ps}{\h+1}\|_{L^2_{l-1}(\O)}
&\le  \frac{4}{2l-1}\|\frac{\p_x \h}{\h+1}\|_{L^\infty_0(\O)}\|\frac{\h_m}{\h+1}\|_{L^2_{l}(\O)}
      +\frac{2}{2l-1} \| \frac{\p_x \h_m}{\h+1} \|_{L^2_{l}(\O)}\\
&\le  \frac{2\d^{-1}}{2l-1} \|\p_x \h_m\|_{L^2_{l}(\O)}
      +\frac{4\d^{-2}}{2l-1} \|\p_x \h\|_{L^\infty_0(\O)}\|\h_m\|_{L^2_{l}(\O)}.
\end{aligned}
\end{equation}
(iv)Differentiating the equality $Z_\t^{\a_1}\h=\h_m+\frac{\p_y \h}{\h+1} Z_\t^{\a_1} \ps$
with the $y$ variable, it follows
\begin{equation*}
\p_y Z^{\a_1}_\t \h
=\p_y \h_m +\p_y^2 \h \frac{Z^{\a_1}_\t \ps}{\h+1}+\p_y \h \p_y ( \frac{Z^{\a_1}_\t \ps}{\h+1} ),
\end{equation*}
which, together with the relation \eqref{b15}, yields
\begin{equation*}
\begin{aligned}
\p_y Z^{\a_1}_\t \h
&=\p_y \h_m +\p_y^2 \h \int_0^y \frac{\h_m}{\h+1}d\xi+\eta_h \h_m\\
&=\p_y \h_m +Z_2 \p_y \h \frac{1}{\varphi(y)}\int_0^y \frac{\h_m}{\h+1}d\xi+\eta_h \h_m,
\end{aligned}
\end{equation*}
where $\varphi(y)=\frac{y}{1+y}$ and  $\eta_h=\frac{\p_y \h}{\h+1}$.
The application of Hardy inequality \eqref{Hardy} and $\h+1\ge \d$ yields
\begin{equation}\label{b19}
\begin{aligned}
\|\p_y Z^{\a_1}_\t \h\|_{L^2_l(\O)}
\le&  \|\p_y \h_m \|_{L^2_l(\O)}+\frac{2\d^{-1}}{2l-1}\|Z_2 \p_y  \h\|_{L^\infty_1(\O)}
      \| \h_m \|_{L^2_{l}(\O)}\\
   &  +\d^{-1}\|\p_y \h \|_{L^\infty_0(\O)}\| \h_m\|_{L^2_l(\O)}\\
\le&  \|\p_y \h_m \|_{L^2_l(\O)}
      +\frac{(2l+1) \d^{-1}}{2l-1}(\|\p_y \h \|_{L^\infty_0(\O)}
      +\|Z_2 \p_y \h\|_{L^\infty_1(\O)})\| \h_m\|_{L^2_l(\O)}.
\end{aligned}
\end{equation}
Therefore, the estimates \eqref{b16}-\eqref{b19} imply the estimates \eqref{b11}-\eqref{b14}.
\end{proof}

Let us define
\begin{equation}\label{xdef}
Y_{m,l}(t):=1+{{\|(\vr, \u, \h)(t)\|_{\H^m_l}^2}}+\|\p_y(\vr, \u, \h)(t)\|_{\H^{m-1}_l}^2.
+\| \p_y \vr (t)\|_{\H^{1,\infty}_1}^2,
\end{equation}
and hence we will establish the following almost equivalent relation.

\begin{lemma}\label{equi-control}
Let $(\vr, \u, \v, \h, \g)$ be sufficiently smooth solution, defined on $[0, T^\es]$,
to the regularized MHD boundary layer equations \eqref{eq5}-\eqref{bc5}.
There exists a constant $\d\in (0, 1)$, such that $\h(t, x, y)+1\ge \d, \forall(t, x, y)\in [0, T]\times \O$.
Then, for $m \ge 4$ and $l \ge 1$, it holds on
\begin{equation}\label{ab1}
\Theta_{m,l}(t)
\le C\|(\rho_0, u_{10}, h_{10})\|_{\overline{\mathcal{B}}^m_l}^4
     +C t \|(\rho_0, u_{10}, h_{10})\|_{\widehat{\mathcal{B}}^m_l}^4
     +C_l \d^{-8}\n_{m,l}^{12}(t),
\end{equation}
and
\begin{equation}\label{ab2}
\n_{m,l}(t)
\le C\|(\rho_0, u_{10}, h_{10})\|_{\overline{\mathcal{B}}^m_l}^4
     +C t \|(\rho_0, u_{10}, h_{10})\|_{\widehat{\mathcal{B}}^m_l}^4
+C\d^{-8} \Theta_{m,l}^{12}(t),
\end{equation}
where $\Theta_{m,l}(t)$ and $\n_{m,l}(t)$ are defined in \eqref{3a1} and \eqref{3d2} respectively.
\end{lemma}

\begin{proof}
By virtue of the definition $\vr_m=Z_\t^{\a_1} \vr-\frac{\p_y \vr}{\h+1}Z_\t^{\a_1}\ps$
and the estimate \eqref{b11}, we find
\begin{equation*}\label{b21}
\begin{aligned}
\|Z_\t^{\a_1} \vr\|_{L^2_l(\O)}^2
&\le \|\vr_m\|_{L^2_l(\O)}^2+\|\p_y \vr\|_{L^\infty_1(\O)}^2\|\frac{Z_\t^{\a_1 }\ps}{\h+1}\|_{L^2_{l-1}(\O)}^2\\
&\le \|\vr_m\|_{L^2_l(\O)}^2+C_l \d^{-2}\|\p_y \vr\|_{L^\infty_1(\O)}^2\|\h_m\|_{L^2_{l}(\O)}^2.
\end{aligned}
\end{equation*}
Similarly, we can obtain for $|\a_1|=m$ that
\begin{equation*}\label{b22}
\|Z_\t^{\a_1}(\u, \h)\|_{L^2_l(\O)}^2
\le \|(\u_m, \h_m)\|_{L^2_l(\O)}^2+C_l \d^{-2}(1+\|\p_y (\u, \h)\|_{L^\infty_1(\O)}^2)\|\h_m\|_{L^2_{l}(\O)}^2.
\end{equation*}
The combination of the above two estimates yields directly
\begin{equation}\label{b22}
\|Z_\t^{\a_1}(\vr, \u, \h)\|_{L^2_l(\O)}^2
\le C_l \d^{-2}(1+\|\p_y (\vr, \u, \h)\|_{L^\infty_1(\O)}^2)\|(\vr_m, \u_m, \h_m)\|_{L^2_l(\O)}^2,
\end{equation}
and hence, we have for $m \ge 4, l \ge 1$
\begin{equation}\label{b23}
\begin{aligned}
\|Z_\t^{\a_1}(\vr, \u, \h)\|_{L^2_l(\O)}^2
&\le C\|\p_y(\u, \h)\|_{L^\infty_1(\O)}^4
     +C\d^{-4}(1+\|(\vr_m, \u_m, \h_m)\|_{L^2_l(\O)}^4+\|\p_y \vr\|_{L^\infty_1(\O)}^4)\\
&\le C\|(\rho_0, u_{10}, h_{10})\|_{\overline{\mathcal{B}}^m_l}^2
     +C t \|(\rho_0, u_{10}, h_{10})\|_{\widehat{\mathcal{B}}^m_l}^2+C_l \d^{-4}(1+X_{m,l}^6(t)),
\end{aligned}
\end{equation}
Due to the definition of
$X_{m,l}(t)$ and $Y_{m,l}(t)$ in \eqref{ydef} and \eqref{xdef} respectively, we get from \eqref{b23} that
\begin{equation}\label{requi-E1}
Y_{m,l}(t)\le C\|(\rho_0, u_{10}, h_{10})\|_{\overline{\mathcal{B}}^m_l}^2
     +C t \|(\rho_0, u_{10}, h_{10})\|_{\widehat{\mathcal{B}}^m_l}^2+C_l \d^{-4}(1+X_{m,l}^6(t)).
\end{equation}

On the other hand, by virtue of the definition of $\vr_m(t)$ and the estimate \eqref{b11}, we find
\begin{equation*}\label{b24}
\begin{aligned}
\|\vr_m(t)\|_{L^2_l(\O)}^2
&\le \|Z_\t^{\a_1} \vr(t)\|_{L^2_l(\O)}^2
  +\|\p_y \vr(t)\|_{L^\infty_1(\O)}^2\|\frac{Z_\t^{\a_1 }\ps}{\h+1}(t)\|_{L^2_{l-1}(\O)}^2\\
&\le \|Z_\t^{\a_1} \vr(t)\|_{L^2_l(\O)}^2+C_l\d^{-2}\|\p_y \vr(t)\|_{L^\infty_1(\O)}^2\|Z_\t^{\a_1}\h(t)\|_{L^2_{l}(\O)}^2,
\end{aligned}
\end{equation*}
and hence, we also have
\begin{equation*}\label{b25}
\begin{aligned}
\|(\u_m, \h_m)(t)\|_{L^2_l(\O)}^2
\le \|Z_\t^{\a_1}(\u, \h)(t)\|_{L^2_l(\O)}^2
    +C\d^{-2}(1+\|\p_y (\u, \h)(t)\|_{L^\infty_1(\O)}^2)\|Z_\t^{\a_1}\h(t)\|_{L^2_{l}(\O)}^2.
\end{aligned}
\end{equation*}
Then, the combination of the above estimates yields directly
\begin{equation} \label{b26}
\|(\vr_m, \u_m, \h_m)(t)\|_{L^2_l}^2
\le C\|(\rho_0, u_{10}, h_{10})\|_{\overline{\mathcal{B}}^m_l}^2
    +C t \|(\rho_0, u_{10}, h_{10})\|_{\widehat{\mathcal{B}}^m_l}^2
    +C_l \d^{-4}(1+Y_{m,l}^6(t)),
\end{equation}
where $m \ge 4, l \ge 1$. According to the definition of $X_{m,l}(t)$ and $Y_{m,l}(t)$,
we get from \eqref{b26} that
\begin{equation}\label{requi-E2}
X_{m,l}(t)\le C\|(\rho_0, u_{10}, h_{10})\|_{\overline{\mathcal{B}}^m_l}^2
     +C t \|(\rho_0, u_{10}, h_{10})\|_{\widehat{\mathcal{B}}^m_l}^2+C_l \d^{-4}(1+Y_{m,l}^6(t)).
\end{equation}

Next, by virtue of the definition $\vr_m(t)$ and estimate \eqref{3103}, we find
\begin{equation*}\label{b27}
\begin{aligned}
\|\p_x Z_\t^{\a_1} \vr\|_{L^2_l(\O)}^2
\le
&\|\p_x \vr_m\|_{L^2_l(\O)}^2+C_l \d^{-2}\|\p_y \vr\|_{L^\infty_1(\O)}^2\|\p_x \h_m\|_{L^2_l(\O)}^2\\
&+C_l \d^{-4}(\|\p_{xy} \vr\|_{L^\infty_1(\O)}^2
  +\|\p_{y} \vr\|_{L^\infty_1(\O)}^2\|\p_{x} \h\|_{L^\infty_1(\O)}^2)\|\h_m\|_{L^2_l(\O)}^2\\
\le
&\|\p_x \vr_m\|_{L^2_l(\O)}^2+C_l \d^{-2} X_{m,l}(t) \|\p_x \h_m\|_{L^2_l(\O)}^2
+C_l \d^{-4}(1+ X_{m,l}^3(t)).
\end{aligned}
\end{equation*}
Similarly, by routine checking, we may conclude that
\begin{equation*}\label{b28}
\begin{aligned}
\|\p_x Z_\t^{\a_1} (\u, \h)\|_{L^2_l(\O)}^2
\le
&C(\|(\rho_0, u_{10}, h_{10})\|_{\overline{\mathcal{B}}^m_l}^2
       +t \|(\rho_0, u_{10}, h_{10})\|_{\widehat{\mathcal{B}}^m_l}^2)
     +C_l\d^{-8}(1+X_{m,l}^6(t))\\
&\!+\!C_l \d^{-2}(1+\|(\rho_0, u_{10}, h_{10})\|_{\overline{\mathcal{B}}^m_l}
     +t \|(\rho_0, u_{10}, h_{10})\|_{\widehat{\mathcal{B}}^m_l}+X_{m,l}^3(t) )\|\p_x \h_m\|_{L^2_l(\O)}^2\\
&+\|\p_x (\u_m, \h_m)\|_{L^2_l(\O)}^2.
\end{aligned}
\end{equation*}
for $m\ge 5, l \ge 1$, and hence it follows
\begin{equation}\label{b29}
\begin{aligned}
\es\|\p_x(\vr, \u, \h)\|_{\H^m_l}^2
&\le C_l \d^{-2}(1+\|(\rho_0, u_{10}, h_{10})\|_{\overline{\mathcal{B}}^m_l}
     +t \|(\rho_0, u_{10}, h_{10})\|_{\widehat{\mathcal{B}}^m_l}+X_{m,l}^3(t) )D_x^{m,l}(t)\\
&+C(\|(\rho_0, u_{10}, h_{10})\|_{\overline{\mathcal{B}}^m_l}^2
       +t \|(\rho_0, u_{10}, h_{10})\|_{\widehat{\mathcal{B}}^m_l}^2)+C_l \d^{-8}(1+X_{m,l}^6(t)),
\end{aligned}
\end{equation}
where $D_x^{m,l}(t)$ is defined in \eqref{3d3}.
By virtue of the definition $\vr_m(t)$ and estimate \eqref{3101}, we get
\begin{equation*}\label{b210}
\begin{aligned}
\|\p_x \vr_m\|_{L^2_l(\O)}^2
\le
&\|\p_x Z_\t^{\a_1} \vr\|_{L^2_l(\O)}^2
  +C\d^{-2}\|\p_y \vr\|_{L^\infty_1(\O)}^2\|\p_x Z_\t^{\a_1}\h\|_{L^2_{l}(\O)}^2\\
&+C\d^{-4}(\|\p_{xy}\vr\|_{L^\infty_1(\O)}^2
  +\|\p_y \vr\|_{L^\infty_1(\O)}^2\|\p_x \h\|_{L^\infty_0(\O)}^2)\|Z_\t^{\a_1}\h\|_{L^2_l(\O)}\\
\le
&\|\p_x Z_\t^{\a_1} \vr\|_{L^2_l(\O)}^2
 +C\d^{-2} Y_{m,l}(t) \|\p_x Z_\t^{\a_1}\h\|_{L^2_{l}}^2+C\d^{-4}(1+ Y_{m,l}^3(t)).
\end{aligned}
\end{equation*}
Similarly, by routine checking, we may conclude that
\begin{equation*}\label{b211}
\begin{aligned}
\|\p_x (\u_m, \h_m)\|_{L^2_l(\O)}^2
\le
&C\d^{-2}(1+\|(\rho_0, u_{10}, h_{10})\|_{\overline{\mathcal{B}}^m_l}
     +t \|(\rho_0, u_{10}, h_{10})\|_{\widehat{\mathcal{B}}^m_l}+Y_{m,l}^3(t) )
      \|\p_x Z_\t^{\a_1}\h\|_{L^2_{l}(\O)}^2\\
&+C(\|(\rho_0, u_{10}, h_{10})\|_{\overline{\mathcal{B}}^m_l}^2
       +t \|(\rho_0, u_{10}, h_{10})\|_{\widehat{\mathcal{B}}^m_l}^2)+C\d^{-8}(1+Y_{m,l}^6(t))\\
&+\|\p_x Z_\t^{\a_1} (\u, \h)\|_{L^2_l(\O)}^2,
\end{aligned}
\end{equation*}
and hence, it follows
\begin{equation}\label{b212}
\begin{aligned}
D_x^{m,l}(t)
\le
&C\d^{-2}(1+\|(\rho_0, u_{10}, h_{10})\|_{\overline{\mathcal{B}}^m_l}
     +t \|(\rho_0, u_{10}, h_{10})\|_{\widehat{\mathcal{B}}^m_l}+ Y_{m,l}^3(t) )\es\|\p_x(\vr, \u, \h)\|_{\H^m_l}^2\\
&+C(\|(\rho_0, u_{10}, h_{10})\|_{\overline{\mathcal{B}}^m_l}^2
       +t \|(\rho_0, u_{10}, h_{10})\|_{\widehat{\mathcal{B}}^m_l}^2)+C\d^{-8}(1+Y_{m,l}^6(t)).
\end{aligned}
\end{equation}
where $D_x^{m,l}(t)$ is defined in \eqref{3d3}.
Similarly, we can justify the estimates
\begin{equation}\label{b213}
\begin{aligned}
\|\p_y(\sqrt{\es} \vr, \sqrt{\mu} \u, \sqrt{\k} \h)\|_{\H^m_l}^2
\le &C(\|(\rho_0, u_{10}, h_{10})\|_{\overline{\mathcal{B}}^m_l}^4
       +t \|(\rho_0, u_{10}, h_{10})\|_{\widehat{\mathcal{B}}^m_l}^4)\\
&+D_y^{m,l}(t)+C_l \d^{-8}(1+X_{m,l}^{12}(t)),
\end{aligned}
\end{equation}
and
\begin{equation}\label{b214}
\begin{aligned}
D_y^{m,l}(t)
\le
&C(\|(\rho_0, u_{10}, h_{10})\|_{\overline{\mathcal{B}}^m_l}^4
       +t \|(\rho_0, u_{10}, h_{10})\|_{\widehat{\mathcal{B}}^m_l}^4)\\
&+\|\p_y(\sqrt{\es} \vr, \sqrt{\mu} \u, \sqrt{\k} \h)\|_{\H^m_l}^2+C\d^{-8}(1+Y_{m,l}^{12}(t)).
\end{aligned}
\end{equation}
Therefore, the combination of estimates \eqref{requi-E1}, \eqref{requi-E2},
\eqref{b29}, \eqref{b212}, \eqref{b213} and \eqref{b214} can establish the
estimates \eqref{ab1} and \eqref{ab2}.
\end{proof}


\section*{Acknowledgements}
Jincheng Gao's research was partially supported by
Fundamental Research Funds for the Central Universities(Grants No.18lgpy66)
and NNSF of China(Grants No.11801586).
Daiwen Huang's research was partially supported by the NNSF of China(Grants No.11631008)
and National Basic Research Program of China 973 Program(Grants No.2007CB814800).
Zheng-an Yao's research was partially supported by NNSF of China(Grant No.11431015).


\end{document}